\documentclass[12pt]{amsart}

\usepackage[usenames, dvipsnames]{xcolor}
\usepackage{amsmath,amsfonts,amssymb,amsthm,mathtools,tikz-cd,xspace}
\usepackage{graphicx}
\usepackage[ruled,vlined,linesnumbered]{algorithm2e}

\usepackage[usenames,dvipsnames]{xcolor}
\definecolor{amethyst}{rgb}{0.6, 0.4, 0.8}

\usepackage[pagebackref=true, colorlinks, allcolors=amethyst]{hyperref}

\renewcommand*{\backref}[1]{}
\renewcommand*{\backrefalt}[4]{%
  \ifcase #1 %
    \relax
  \or
    $\uparrow$#2.%
  \else
    $\uparrow$#2.%
  \fi%
}

\usepackage{fullpage,url,amssymb,colonequals,algorithm2e}
\usepackage[shortlabels]{enumitem}
\usepackage{breakurl}
\usepackage{mathrsfs} 
\usepackage{tabularx} 
\usepackage{booktabs} 
\usepackage{pgfplots}
\pgfplotsset{compat=1.13}
\usetikzlibrary{arrows}

\usepackage[T1]{fontenc}

\usepackage{cleveref}

\newcommand{\A}{\mathbb{A}}

\newcommand{\F}{\mathbb{F}}

\renewcommand{\P}{\mathbb{P}}
\newcommand{\Q}{\mathbb{Q}}
\newcommand{\R}{\mathbb{R}}
\newcommand{\Z}{\mathbb{Z}}

\newcommand{\bA}{\mathbb{A}}
\newcommand{\bC}{\mathbb{C}}

\newcommand{\bG}{\mathbb{G}}
\newcommand{\bH}{\mathbb{H}}

\newcommand{\bP}{\mathbb{P}}
\newcommand{\bQ}{\mathbb{Q}}
\newcommand{\bR}{\mathbb{R}}
\newcommand{\bZ}{\mathbb{Z}}

\newcommand{\calD}{\mathcal{D}}

\newcommand{\calO}{\mathcal{O}}

\newcommand{\calR}{\mathcal{R}}

\newcommand{\calX}{\mathcal{X}}

\newcommand{\dD}{\mathcal{D}}
\newcommand{\cO}{\mathcal{O}}
\newcommand{\dH}{\mathcal{H}}
\newcommand{\dP}{\mathcal{P}}
\newcommand{\dX}{\mathcal{X}}
\newcommand{\dY}{\mathcal{Y}}
\newcommand{\dU}{\mathcal{U}}
\newcommand{\dZ}{\mathcal{Z}}

\newcommand{\rC}{\mathrm{C}}
\newcommand{\rd}{\mathrm{d}}
\newcommand{\rH}{\mathrm{H}}
\newcommand{\rM}{\mathrm{M}}

\newcommand{\mm}{\mathfrak{m}}
\renewcommand{\tt}{\mathfrak{t}}

\renewcommand{\ss}{\mathfrak{s}}

\newcommand{\fm}{\mathfrak{m}}

\newcommand{\fS}{\mathfrak{S}}
\newcommand{\fU}{\mathfrak{U}}

\newcommand{\XX}{\mathscr{X}}

\newcommand{\Qbar}{{\overline{\Q}}}

\newcommand{\kbar}{{\overline{k}}}
\newcommand{\Kbar}{{\overline{K}}}

\newcommand{\Fbar}{{\overline{\F}}}
\newcommand{\fbar}{{\bar{f}}}

\DeclareMathOperator{\argmin}{argmin}

\DeclareMathOperator{\Frac}{Frac}

\DeclareMathOperator{\Jac}{Jac}

\DeclareMathOperator{\red}{red}
\DeclareMathOperator{\Res}{Res}

\DeclareMathOperator{\Spec}{Spec}
\DeclareMathOperator{\Sp}{Sp}

\DeclareMathOperator{\Sym}{Sym}
\DeclareMathOperator{\tr}{tr}

\newcommand{\isomto}{\overset{\sim}{\rightarrow}}

\newcommand{\llbrack}{[\![}
\newcommand{\rrbrack}{]\!]}

\newcommand{\an}{\mathrm{an}}
\newcommand{\rg}{\mathrm{rg}}
\newcommand{\cl}{\mathrm{cl}}
\newcommand{\deR}{\mathrm{dR}}
\newcommand{\dR}{\mathrm{dR}}
\newcommand{\et}{\mathrm{\acute{e}t}}
\newcommand{\gr}{\mathrm{gr}}

\newcommand{\II}{\ensuremath{\mathrm{II}}\xspace}
\newcommand{\III}{\ensuremath{\mathrm{III}}\xspace}

\newcommand{\rig}{\mathrm{rig}}

\newcommand{\sk}{\mathrm{sk}}

\theoremstyle{plain}\newtheorem{ithm}{Theorem}
\theoremstyle{plain}\newtheorem{iprop}[ithm]{Proposition}
\theoremstyle{plain}\newtheorem{icor}[ithm]{Corollary}
\theoremstyle{remark}

\counterwithin{equation}{section}

\theoremstyle{plain}
\newtheorem{theorem}[equation]{Theorem}
\newtheorem{lemma}[equation]{Lemma}
\newtheorem{proposition}[equation]{Proposition}
\newtheorem{corollary}[equation]{Corollary}
\theoremstyle{definition}
\newtheorem{definition}[equation]{Definition}
\newtheorem{example}[equation]{Example}
\theoremstyle{remark}
\newtheorem{remark}[equation]{Remark}

\title{Local heights on hyperelliptic curves and quadratic Chabauty}

\author{L.\ Alexander Betts}
\address{Department of Mathematics, Cornell University, 212 Garden Avenue, Ithaca, NY 14850, USA}
\email{alex.betts@cornell.edu}

\author{Juanita Duque-Rosero}
\address{Department of Mathematics, Boston University, 665 Commonwealth Ave, Boston, MA 02215, USA}
\email{juanita@bu.edu}

\author{Sachi Hashimoto}
\address{Department of Mathematics, Brown University, Box 1917, 151 Thayer Street, Providence, RI 02912, USA}
\email{sachi\_hashimoto@brown.edu}

\author{Pim Spelier}
\address{Department of Mathematics, Utrecht University, P.O. Box 80010, 3508 TA Utrecht, The Netherlands}
\email{p.spelier@uu.nl}

\definecolor{ccc}{RGB}{55,156,158}
\definecolor{cccc}{RGB}{120,184,48}

\begin{document}
\begin{abstract}
Quadratic Chabauty is a $p$-adic method for determining rational points on curves.
Local heights are arithmetic invariants used in the quadratic Chabauty method.
We present an algorithm to compute these local heights for hyperelliptic curves at odd primes $\ell\neq p$.
This algorithm significantly broadens the applicability of quadratic Chabauty to curves which were previously inaccessible due to the presence of non-trivial local heights.  We provide numerous examples, including the first quadratic Chabauty computation for a curve having two primes with non-trivial local heights.
\end{abstract}

\maketitle


\section{Introduction}
Smooth projective curves over the rational numbers $X/\bQ$ provide a rich terrain for exploring diophantine questions.  The quadratic Chabauty method has emerged as a powerful $p$-adic technique for computing the set of rational points $X(\bQ)$ on curves of genus $g > 1$ with Jacobian $J$ \cite{QCCartan,AdzagaAtkinLehner,examplesandalg,AdzagaAtkinLehner456}.   
This method, introduced in \cite{QCI,QCII}, extends the Chabauty--Coleman method \cite{Chabauty,EffectiveChab,McCallumPoonen}. It applies when certain conditions are satisfied, including that the Picard number $\rho_J$ of $J$ is at least $2$ and that the Mordell--Weil rank $r$ of $J(\bQ)$ is less than $g + \rho$.

However, algorithmic computations of rational points using the quadratic Chabauty method have been limited to curves with potentially good reduction at every prime, curves whose special fibre at primes of bad reduction consists of a unique irreducible component: more generally, examples where the dual graph of the curve has a particularly simple shape  (e.g. see \cite[Section~5]{examplesandalg}).
Additionally, in the special case of genus 2 curves whose Jacobian is a product of elliptic curves the rational points can be algorithmically computed by reducing the problem to computing on the elliptic curve factors \cite{QCI,Bianchi(bi)elliptic,OanaFrancesca}
This restriction on the reducation type or dual graph of $X$ arises from an inability to compute an arithmetic invariant, namely the values attained by the \emph{local height functions} for the primes $\ell \neq p$.

We provide an algorithm to compute the local height function for hyperelliptic curves at odd primes, and therefore extend the applicability of quadratic Chabauty significantly.
As an application, we carry out the first quadratic Chabauty computation on a curve with more than one prime $\ell \neq p$ with non-trivial local heights (outside of the bielliptic case), showing the following theorem.
\begin{ithm}[cf. \S \ref{sec:QC}]\label{thm:qcExample}
There are precisely $10$ rational points on the curve $X: y^2 = x^6 + 18/5x^4 + 6/5x^3 + 9/5x^2 + 6/5x + 1/5$.
\end{ithm}
\noindent Furthermore, the techniques we develop here give a road map for computing local heights on non-hyperelliptic curves $X$:
the only step of our algorithm that does not generalize is the computation of a suitable analytic covering of $X$. There has been promising recent progress on this step in \cite{Helminck:semistableModels,Ossen:2023SemistableReductionParticular,Ossen:semistableReduction}.

\smallskip
The central object in quadratic Chabauty is a Nekov\'a\u{r} $p$-adic height function $h_Z$, which depends on a choice of trace zero correspondence $Z \subset X\times X$ invariant under the Rosati involution.  
The global $p$-adic height function $h_Z$ 
can be expressed as a sum of local height functions $h_Z = \sum_{\ell \text{ prime}} h_{Z,\ell}$. When $\ell \neq p$ is a prime of potentially good reduction for $X$, the local height function $h_{Z,\ell}$ vanishes. The local height at $p$ is a locally analytic function, and while there exist established algorithms and implementations for computing $h_{Z,p}$ \cite{QCCartan}, the situation for the local height $h_{Z,\ell}$ at $\ell \neq p$ stands in stark contrast.

Until now, algorithmic calculations of $h_{Z,\ell}$ have been limited to the case where $X$ is an elliptic curve or $X$ is bielliptic of genus $2$, in which case $h_{Z,\ell}$ is determined by arithmetic invariants of the regular model of the elliptic curve(s) \cite{Silverman,CremonaPrickettSiksek,Bianchi(bi)elliptic}. Beyond the case of elliptic and bielliptic curves, strategies for computing $h_{Z,\ell}$ have relied on constructing a regular model for $X$ over $\bZ_\ell$. For example, two recent local heights calculations hinged on both computing a regular semistable model for $X$ over $\bZ_\ell$ and the existence of an abundance of rational points on $X$ \cite{examplesandalg}.

The value of the local height function $h_{Z,\ell}$ at a point $z\in X(\bQ_\ell)$ can be defined in terms of the local N\'eron--Tate height pairing as the height pairing of two divisors depending on the correspondence~$Z$, the point~$z$, and a chosen basepoint~$b$: 
$h_{Z,\ell}(z) = h_\ell(b-z,D_Z(b,z))$ where~$D_Z(b,z)$ is a divisor depending on~$Z$, a basepoint~$b$, and a~$\bQ_\ell$-rational point~$z$. 
While algorithms exist for computing the local height \emph{pairing} given a regular model for $X$ over $\bZ_\ell$ \cite{Holmes,Mueller,vanBommelHolmesMueller}, computing the height \emph{function}~$h_{Z,\ell}$ in this way seems to be impractical, given that the defining equations for $Z$ typically have very large degree, see Remark \ref{rem:intersectionpair}.
Consequently, beyond the case of (bi)elliptic curves, quadratic Chabauty has almost exclusively been applied in cases where all local height functions away from $p$ vanish. 
Another interpretation of the local height in terms of $\ell$-adic analysis can be found in the recent preprint of \cite{BSM}. 

This paper introduces an algorithm that marks the first practical method for computing local heights away from $p$ for a class of curves, outside of the (bi)elliptic case. Our algorithm computes the Nekov\'a\u{r} local $p$-adic heights $h_{Z,\ell}$ for odd primes $\ell \neq p$ on hyperelliptic curves of genus $g > 1$. This algorithm significantly broadens the applicability of quadratic Chabauty to curves previously deemed inaccessible due to potentially having non-trivial local heights, offering a promising avenue for advancing our understanding of rational points on higher genus curves or those with larger conductors (see \Cref{thm:qcExample} and \Cref{cor:trivialHeightsIntro}).

We illustrate our algorithm by computing numerous examples of local heights on hyperelliptic curves with diverse reduction types.  The computations for these examples use our \textsf{Magma} implementation of the algorithm, available at \url{https://github.com/sachihashimoto/local-heights}. We revisit several Atkin--Lehner quotients of $X_0(N)$ from \cite{AdzagaAtkinLehner}, where they were not able to apply quadratic Chabauty due to the presence of potentially non-zero local heights away from $p$. We also study an Atkin--Lehner quotient of a Shimura curve $X_0(93,1)/\langle w_{93} \rangle$; rational points on this type of Shimura curve quotient parametrise abelian surfaces with potential quaternionic multiplication.
Our algorithm is practical even in high genus: we compute local heights on a genus 7 modular curve.

Our approach to computing local heights begins with a formula from \cite{BettsDogra} which describes~$h_{Z,\ell}$ in terms of the action of $Z_*$ on the homology of the reduction graph~$\Gamma$ of a semistable model of $X$ and certain integers $\tr_v(Z)$ attached to the vertices of this graph. The problem, as noted in \cite[\S3.1]{examplesandalg}, is that it is not \emph{a priori} clear how to compute the action of~$Z_*$ on~$\rH_1(\Gamma,\bZ)$, especially if~$X$ or~$Z$ has large genus, many components, or highly unstable reduction. Following a strategy suggested to us by Netan Dogra, we solve this problem using the Coleman--Iovita isomorphism for the curve~$X$ (\Cref{thm:coleman-iovita}, \cite{ColemanIovitaHiddenStructures}), which relates the homology of~$\Gamma$ to the de Rham cohomology of~$X$, where the action of~$Z_*$ is easier to compute. In order to use the Coleman--Iovita isomorphism, we verify that it commutes with the action of correspondences.

\begin{ithm}[Push-pull compatibility of Coleman--Iovita, cf.\ \Cref{prop:coleman-iovita_functoriality}]\label{ithm:c-i_functoriality}
	Let~$f\colon X\to X'$ be a finite morphism of smooth projective curves over~$\bC_\ell$. Then the pullback and pushforward maps on de Rham cohomology are compatible with the Coleman--Iovita isomorphism.
\end{ithm}

Making this strategy viable requires several innovations. The first problem is to determine the reduction graph~$\Gamma$ of~$X$ over a finite extension of~$\bQ_\ell$ where~$X$ acquires semistable reduction. We do this using the theory of cluster pictures, which allows one to read off the graph~$\Gamma$ for a hyperelliptic curve in odd residue characteristic directly from the valuations of the differences of the roots of~$f(x)$ \cite{DDMM}.

The second problem is that, in order to explicitly compute the Coleman--Iovita isomorphism, we need to write down a suitable analytic covering of~$X_{\bC_\ell}$. We explain how to read off this semistable covering from the cluster picture.
\begin{ithm}[cf.\ \Cref{thm:semistablecoverX}]\label{ithm:covering}
	Let~$X/K$ be a hyperelliptic curve with split semistable reduction over a finite extension~$K$ of~$\bQ_\ell$, given by an equation $y^2=f(x)$. Then there exists a semistable covering~$\fU=(\tilde U_\ss^\pm)_\ss$ of the rigid analytification $X^\an$ indexed by proper clusters~$\ss$ in the cluster picture of~$f$, where~$\tilde U_\ss^\pm$ is as defined in \Cref{def:covering_of_X}.
\end{ithm}
It is interesting to note that this same semistable covering has already appeared in the literature without the explicit link to cluster pictures \cite{StollUniformBounds,KatzKaya}. In \Cref{sec:ExplicitColemanIovitaHyp} we use the semistable covering to write down explicit formulas describing the Coleman--Iovita isomorphism in this case, placing \cite[Section 5]{KatzKaya} in the context of cluster pictures.
We also reinterpret and reprove the results of \cite[Section~8]{DDMM} on determining the homology of the dual graph and the special fibre from the cluster picture in the language of Berkovich spaces, see \Cref{the:minimalSkeleton}.

The end result of this calculation is an \emph{$\ell$-adic approximation} to the action of~$Z_*$ on the homology of~$\Gamma$. However, for applications to quadratic Chabauty, one needs an exact answer. The third innovation that we need is to prove bounds on the possible actions of~$Z_*$ and the traces $\tr_v(Z)$ attached to vertices. This allows us to certify that the approximate values coming from our calculation are correct. The following theorem applies to all (smooth, projective, geometrically integral) curves, not just hyperelliptic ones.
\begin{ithm}[Boundedness of norms and traces, cf. \Cref{thm:bounds}]\label{ithm:bounds}
	Let~$Z\subset X\times X$ be an effective correspondence, of degrees~$d_1$ and~$d_2$ over~$X$, respectively. Then the operator norm of $Z_*$ on $\rH_1(\Gamma,\bZ)$ and $|\tr_v(Z)|$ have explicit bounds depending only on the degrees $d_1$ and $d_2$.
\end{ithm}

In summary, our method for determining the action of~$Z_*$ on~$\rH_1(\Gamma,\bZ)$ proceeds in four steps. Here,~$K$ is a field over which~$X$ acquires split semistable reduction. In steps 1 and 2, we require $X$ to be hyperelliptic, but steps 3 and 4 are feasible for any curve.
\begin{enumerate}
	\item Use the cluster picture associated to $X$ to write down a semistable covering of~$X^\an$ and a basis of~$\rH_1(\Gamma,\bZ)$, using \Cref{ithm:covering}.
	\item Compute the matrix~$M_\deR$ representing the action of~$Z_*$ on~$\rH^1_\deR(X/K)$ in some basis of differentials of the second kind.
	\item Compute a matrix~$T$ which is, to suitably high precision, an $\ell$-adic approximation to the surjection $\rH^1_\deR(X/K)\to K\otimes_{\bZ}\rH_1(\Gamma,\bZ)$ coming from the Coleman--Iovita isomorphism.
	\item By \Cref{ithm:c-i_functoriality}, the matrix $M_\Gamma \coloneqq T\cdot M_\deR\cdot T^{-1}$ is then an $\ell$-adic approximation of the action of~$Z_*$ on~$\rH_1(\Gamma,\bZ)$, where~$T^{-1}$ is a right-inverse of~$T$. Compute the action of~$Z_*$ on~$\rH_1(\Gamma,\bZ)$ by rounding $M_{\Gamma}$ to the unique integer matrix that satisfies the bounds from \Cref{ithm:bounds}.
\end{enumerate}

One other consequence of our work is that, by linking together local heights and the combinatorics of the cluster picture, we obtain combinatorial constraints on local heights. This allows us to give various new criteria for when a hyperelliptic curve has all local heights equal to~$0$, for example the following check is independent of $Z$.
\begin{iprop}[cf.\ \Cref{ex:trivialHeights}]
\label{prop:trivialhts}
	Suppose that~$X/\bQ_\ell$ is a genus~$2$ curve with the cluster picture shown below and that the leading coefficient of~$X$ is a unit in~$\bZ_\ell$. Then the local height of any~$\bQ_\ell$-point on~$X$ is~$0$.
	\begin{center}
	\includegraphics[width=4cm]{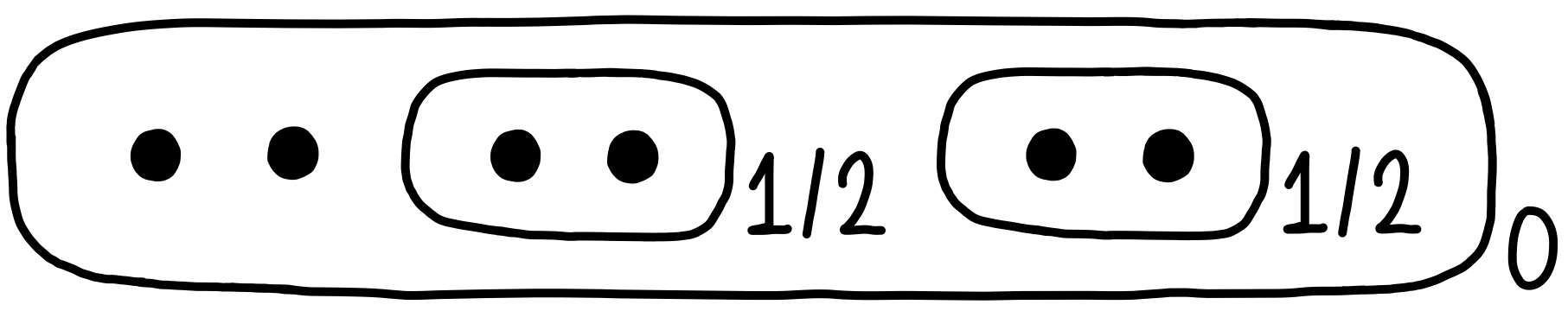}
	\end{center}
\end{iprop}
\begin{icor}[cf.\ \S\ref{subsec:shimuracurve}]\label{cor:trivialHeightsIntro}
	The Shimura curve quotient $X_0(93, 1)/\langle \omega_{93} \rangle$ has trivial local height at $31$.
\end{icor}

In broad strokes, this overall strategy of computing local heights via the Coleman--Iovita isomorphism is viable for any smooth projective curve $X$. Our method relies on $X$ being hyperelliptic in only one part. We specialise to hyperelliptic curves in order to use the machinery of cluster pictures to read off the reduction graph~$\Gamma$ and semistable covering of~$X$, i.e.\ the data of the Berkovich skeleton of~$X$. From this data, and the action of~$Z_*$ on $\rH^1_\deR(X/\bQ)$, we produce the local heights on~$X$. Thus, in order to generalise our algorithm to non-hyperelliptic curves, one would just need a method to determine the Berkovich skeleton. Moreover, our strategy seems promising for determining local heights without using explicit equations for~$X$: if one has some \emph{a priori} way of determining both the Berkovich skeleton and the action of~$Z_*$ on $\rH^1_\deR(X/\bQ)$, say for~$X$ a modular curve, then we show that this determines the local heights.

\smallskip

The structure of the paper is as follows. In \Cref{sec:semistable} we introduce the notions of semistable covering, semistable vertex set, and split semistable model, as well their equivalences. The definition of the local height function $h_{Z,\ell}$ and the local heights formula appear in \Cref{sec:localheights}. We also prove the boundedness of the operator norm and the traces $\tr_v(Z)$ for correspondences $Z \subset X \times X$ in this section. \Cref{sec:ColemanIovita} explains the Coleman--Iovita isomorphism and proves the compatibility with pushforward and pullback. We specialise to the case of hyperelliptic curves in \Cref{sec:ExplicitColemanIovitaHyp}, and give an explicit description of the Coleman--Iovita isomorphism for these curves in this section. To do this, we introduce the machinery of cluster pictures, use this to construct a semistable covering for hyperelliptic curves, and explicitly describe the Berkovich skeleton of $X^\an$ over~$\bC_\ell$. We expand on the details of the explicit computation of the Coleman--Iovita isomorphism in \Cref{sec:computations}. We show how to represent functions on an annulus, how to compute the action of $Z_*$ on $\rH^1_{\dR}(X/K)$, and finally how to turn an $\ell$-adic matrix for the action of $Z_*$ into an integer matrix, given the bounds from \Cref{thm:bounds}. Finally, \Cref{sec:examples} contains numerous worked examples.

\subsection*{Acknowledgements} 
We are very grateful to Netan Dogra for the inspiration for the ideas presented in this paper. We are also indebted to Jeroen Sijsling and John Voight for helping us with the implementation of the \textsf{Magma} code. We thank Jennifer Balakrishnan for providing numerous comments on an earlier draft of this paper. We are also grateful to Michael Stoll for sharing his code to demonstrate how to apply elliptic curve Chabauty on the example in \S \ref{sec:QC}. We would like to thank Matt Baker, Amnon Besser, Edgar Costa, Noam Elkies, Paul Helminck, David Holmes, Aashraya Jha, Steffen M\"uller, Padmavathi Srinivasan, and Jan Vonk for helpful conversations, Laurent Moret-Bailly for a helpful response to a MathOverflow question, and Ole Ossen for inspiring one of the pictures. We also thank the referee for numerous helpful suggestions for improving the paper. L.A.B. was supported by a grant from the Simons Foundation~(550031). J.D.R. was partially supported by a grant from the Simons Foundation~(550023). P.S. was supported by NWO grant VI.Vidi.193.006.


\section{Semistable models, semistable coverings, and semistable vertex sets}
\label{sec:semistable}

Before we begin the paper proper, we first discuss some preliminaries regarding reduction graphs~$\Gamma$ of smooth projective curves~$X$ defined over $\ell$-adic local fields. There are three equivalent ways that these graphs can be defined: in terms of a \emph{split semistable model} of~$X$; in terms of a \emph{semistable covering} of the rigid-analytification of~$X$; or in terms of a \emph{semistable vertex set} inside the Berkovich analytification of~$X$. All three perspectives have their advantages: split semistable models are the simplest conceptually and are the most widely known; semistable coverings are well-adapted to computations; and semistable vertex sets are combinatorial in nature and well-adapted to proving theoretical results. In this paper, it will be vital to switch freely between these three different perspectives, and so we will begin by recalling the three definitions and their interrelationships, with a particular emphasis on how these notions behave with respect to morphisms of curves.

For the rest of this paper, we fix the following notation. We will fix a prime number~$\ell$, and will work over a field~$K$ which is (for simplicity) either equal to~$\bC_\ell$ or a finite extension of~$\bQ_\ell$. We denote the ring of integers, maximal ideal, and residue field of~$K$ by~$\cO_K$, $\fm_K$, and~$k$, respectively. We always normalise the norm and valuation on~$K$ so that~$|\ell|=\ell^{-1}$ and~$v(\ell)=1$. A curve~$X$ over~$K$ is always assumed to be smooth, projective, and geometrically integral.

\begin{definition}
A 1-dimensional separated scheme~$\bar\dX$ of finite type over~$k$ is called a \emph{semistable curve} just when it is geometrically reduced and has at worst ordinary double points as singularities. It is called \emph{strongly semistable} just when additionally every irreducible component is smooth; it is called \emph{split semistable} just when additionally every component is geometrically irreducible, every singular point is $k$-rational, and the two tangent directions at every singular point are also $k$-rational.

A \emph{model} of a smooth projective curve~$X/K$ is a flat, proper, and finitely presented $\cO_K$-scheme~$\dX$ together with an isomorphism $\dX_K\cong X$ of its generic fibre with~$X$ over~$K$. (Any model~$\dX$ is automatically integral.) A model is called \emph{semistable} (resp.\ \emph{strongly semistable}, resp.\ \emph{split semistable}) just when its special fibre is. Note that we do not require our semistable models to be regular, nor do we require them to be minimal.
\end{definition}

\begin{remark}
	The definitions of semistable models used across the literature vary slightly in exactly which properties they require of~$\dX$. For example, \cite[Definition~10.3.27]{LiuAlgebraicCurves} requires (for~$\cO_K$ Dedekind) that~$\dX$ be flat, projective, and normal; \cite[Remark~4.2(2)]{Baker2013NonArchimedeanAnalyticCurves} requires (for~$\cO_K=\bC_\ell$) that~$\dX$ be flat, proper, and integral; and \cite[Definition~2.35]{McMurdyColeman} just requires that~$\dX$ be flat and proper. Our definition agrees with all three, as we will now show. Since~$\cO_K$ is an integral domain, it follows that any flat, finite type $\cO_K$-scheme is automatically finitely presented \cite[Corollaire~3.47]{RaynaudGrusonCriteresdePlatitude}, and also any flat $\cO_K$-scheme with integral generic fibre is integral. For normality, we note that normality is \'etale local, so semistable curves are normal. And finally a semistable model $\dX/\cO_K$ is projective, as we can give a relatively very ample Cartier divisor by taking a sum of smooth points that meet every component of the special fibre with sufficiently high multiplicity.
\end{remark}

We will need to say a little about the local structure of a split semistable model~$\dX$ at a singular point~$\bar x$. For this, if~$a\in\fm_K\smallsetminus\{0\}$, let
\[
\fS(a) \coloneqq \Spec(\cO_K[s,t]/(st-a))
\]
denote the standard algebraic annulus with parameter~$a$. We write~$\bar 0\in\fS(a)(k)$ for the point defined by the $\cO_K$-algebra homomorphism sending~$s$ and~$t$ to~$0\in k$. According to (the proof of) \cite[0CBY]{stacks-project}, for any singular point~$\bar x\in\bar\dX(k)$ of the special fibre of a split semistable model~$\dX$, there exists some~$a\in\fm_K\smallsetminus\{0\}$, an $\cO_K$-scheme $\dU$ with \'etale maps
\[
\dX \leftarrow \dU \rightarrow \fS(a) \,,
\]
and a $k$-point $\bar u\in \bar\dU(k)$ mapping to both $\bar x\in \bar\dX(k)$ and $\bar 0\in \bar\fS(a)(k)$. We call a diagram as above a \emph{chart} at~$\bar x$.

The maps appearing in a chart, being \'etale, induce isomorphisms on completed local rings. Thus, given two charts
\[
(\dX,\bar x) \leftarrow (\dU,\bar u) \rightarrow (\fS(a),\bar 0) \quad\text{and}\quad (\dX,\bar x) \leftarrow (\dU',\bar u') \rightarrow (\fS(a'),\bar 0) \,,
\]
we find that the completed local rings of~$\fS(a)$ and~$\fS(a')$ at~$\bar 0$ are isomorphic. This implies that~$a$ and~$a'$ differ by multiplication by a unit in~$\cO_K$ \cite[Lemma~2.1]{Kato2000Log-smooth-defo}. (\cite[Lemma~2.1]{Kato2000Log-smooth-defo} is only stated when~$\cO_K$ is noetherian, but the conclusion we need is still valid when~$K=\bC_\ell$. The final part of the proof of \cite[Lemma~2.1]{Kato2000Log-smooth-defo} still works, with the caveat that rather than considering e.g.~the completed local ring $\hat\cO_{\dX,\bar x}$, we instead need to consider the $I$-adic completion of~$\cO_{\dX,\bar x}$, where~$I$ is a \emph{finitely generated} ideal of~$\cO_{\dX,\bar x}$ whose radical is the maximal ideal. For instance, in the localisation of~$\cO_K[s,t]/(st-a)$, one could take~$I=(s,t,\varpi)$ where~$\varpi\in\fm_K\smallsetminus\{0\}$.) In any case, this implies that~$v(a)=v(a')$, and so the positive rational number~$v(a)$ is independent of the choice of chart at~$\bar x$.

\begin{definition}\label{def:thickness}
	The value $v(a)$ is called the \emph{thickness} of the singular point~$\bar x$.
\end{definition}

\begin{remark}
	When~$K$ is a finite extension of~$\bQ_\ell$, then~$\dX$ is noetherian, and so a chart $\dX\leftarrow\dU\rightarrow\fS(a)$ induces an isomorphism
	\[
	\hat\cO_{\dX,\bar x}\cong \cO_K\llbrack s,t\rrbrack/(st-a)
	\]
	on completed local rings \cite[Proposition~4.3.26]{LiuAlgebraicCurves}. Hence our definition of the thickness of~$\bar x$ agrees with the usual definition \cite[Definition~10.3.23]{LiuAlgebraicCurves} up to a factor of the ramification degree of~$K/\bQ_\ell$ (coming from the fact that we normalise our valuation on~$K$ differently to \cite{LiuAlgebraicCurves}). In particular, the thickness of~$\bar x$ as we have defined it need not be an integer.
\end{remark}

Much of the geometry of a split semistable model~$\dX$ can be captured by a combinatorial invariant known as the \emph{reduction graph}~$\Gamma=\Gamma_\dX$. The vertices of~$\Gamma$ are the irreducible components of the special fibre~$\bar\dX$, and the unoriented edges of~$\Gamma$ are the singular points of~$\bar\dX$. Moreover, one can attach to~$\Gamma$ certain extra data making it into a combinatorial object known as a \emph{metrised complex of $k$-curves}.

\begin{definition}[Metrised complex of $k$-curves, {\cite[Definition~2.17]{AminiBakerLiftingHarmonicI}}]
	In this paper, a \emph{graph} always means a finite graph in the sense of Serre, i.e.\ a quadruple $\Gamma=(V(\Gamma),E(\Gamma),\partial_0,(-)^{-1})$ where~$V(\Gamma)$ and~$E(\Gamma)$ are finite sets and
	\begin{equation}\label{eq:partial0}
		\partial_0\colon E(\Gamma) \to V(\Gamma) \quad\text{and}\quad (-)^{-1}\colon E(\Gamma) \to E(\Gamma)
	\end{equation}
	are functions, where~$(-)^{-1}$ is an involution without fixed points. Elements of $V(\Gamma)$ and $E(\Gamma)$ are known as \emph{vertices} and \emph{oriented edges} of~$\Gamma$, respectively. If~$e\in E(\Gamma)$ is an oriented edge, then $e^{-1}$, $\partial_0(e)$, and $\partial_1(e)\coloneqq \partial_0(e^{-1})$ are known as the \emph{inverse}, \emph{source}, and \emph{target} of~$e$, respectively. For a vertex~$v$, we write~$T_v(\Gamma)$ for the set of oriented edges with source~$v$, and call~$T_v(\Gamma)$ the set of \emph{tangent directions} at (or out of)~$v$. The set~$E(\Gamma)^+$ of \emph{unoriented edges} of~$\Gamma$ is the quotient of~$E(\Gamma)$ by the equivalence relation~$e\sim e^{-1}$.\smallskip

	\noindent A \emph{metrised graph} is a graph~$\Gamma$ endowed with:
	\begin{itemize}
		\item for each unoriented edge~$e\in E(\Gamma)^+$, a positive real number $l(e)\in\bR_{>0}$ called the \emph{length} of~$e$.
	\end{itemize}
	A \emph{metrised complex of $k$-curves} is a connected metrised graph~$\Gamma$ endowed with, additionally:
	\begin{itemize}
		\item for each vertex~$v\in V(\Gamma)$, a smooth projective curve~$\bar\dX_v/k$, called the \emph{vertex curve} at~$v$; and
		\item for each tangent direction $e\in T_v(\Gamma)$, a $k$-point $\bar x_e\in\bar\dX_v(k)$.
	\end{itemize}
	We require that the points~$\bar x_e$ for different tangent directions~$e$ are distinct.
\end{definition}

To make the reduction graph~$\Gamma$ of a split semistable model $\dX$ into a metrised complex of $k$-curves, we adopt the following conventions.
\begin{itemize}
	\item $V(\Gamma)$ is the set of irreducible components of the normalisation~$\bar\dX^{\sim}$ of the special fibre~$\bar\dX$. For a vertex~$v\in V(\Gamma)$, we write~$\bar\dX_v$ for the component of~$\bar\dX^{\sim}$ to which it corresponds.
	\item $E(\Gamma)$ is the set of~$k$-points of~$\bar\dX^{\sim}$ lying over singular points in~$\bar\dX$. For an oriented edge~$e\in E(\Gamma)$, we write~$\bar x_e\in\bar\dX^{\sim}(k)$ for the point to which it corresponds.
	\item For an oriented edge~$e\in E(\Gamma)$, $\partial_0(e)$ is the vertex such that~$\bar\dX_{\partial_0(e)}\ni\bar x_e$, and~$e^{-1}$ is the oriented edge different from~$e$ such that~$\bar x_{e^{-1}}$ and~$\bar x_e$ lie over the same singular point of~$\bar\dX$.
	\item The unoriented edges~$e\in E(\Gamma)^+$ correspond to singular points~$\bar x_e$ of~$\bar\dX$. The \emph{length} of~$e$ is defined to be the thickness of~$\bar x_e$ inside~$\dX$ (\Cref{def:thickness}).
	\item The points~$\bar x_e\in\bar\dX_v(k)$ attached to tangent directions~$e$ at~$v$ are the elements above.
\end{itemize}

We easily verify that the construction of a metrised complex of $k$-curves to a split semistable model of a curve is compatible with base change, as per the following lemma.

\begin{lemma}\label{lem:base_change_model_reduction_graph}
	Let~$K'/K$ be an extension of fields, each of which is either~$\bC_\ell$ or a finite extension of~$\bQ_\ell$, with residue field extension~$k'/k$. Let~$\dX/\cO_K$ be a split semistable model of a curve~$X/K$, with reduction graph~$\Gamma$ (viewed as a metrised complex of $k$-curves). Then~$\dX_{\cO_{K'}}$ is a split semistable model of~$X_{K'}$, whose reduction graph~$\Gamma'$ is the metrised complex of $k'$-curves obtained from~$\Gamma$ by base-changing all the curves attached to vertices from~$k$ to~$k'$, and leaving the underlying graph, metric and maps unchanged.
\end{lemma}
\begin{proof}
	This is clear. We remark that the fact that~$\Gamma$ and~$\Gamma'$ have the same metric is a consequence of our choice of normalisation of the valuations on~$K$ and~$K'$, which ensures that the thickness of singular points of~$\bar\dX$ is unchanged upon base-change to~$\cO_{K'}$.
\end{proof}

\subsection{Semistable models, semistable coverings and semistable vertex sets}

While semistable models of a curve~$X$ are well-behaved theoretically, they are rather difficult to perform explicit computations with -- indeed, even writing down a semistable model of~$X$ is difficult computationally. Instead, a central part of our approach in this paper will be to replace semistable models with certain analytic data attached to~$X$, equivalent to a choice of semistable model, but much more amenable to direct computation. We will work with two kinds of analytic data. Firstly, in \cite{ColemanReciprocityLaws,McMurdyColeman,ColemanMcCallumStableReductionFermatCurves}, Coleman defines the notion of a \emph{semistable covering} of the rigid analytification~$X^\an$, and shows that (with some small caveats) semistable coverings correspond to semistable models. Secondly, in the case~$K=\bC_\ell$, Baker, Payne, and Rabinoff define the notion of a \emph{semistable vertex set} inside the Berkovich analytification~$X^\an$ of~$X$, and again show that semistable vertex sets correspond to semistable models \cite{Baker2013NonArchimedeanAnalyticCurves}. We will presently recall the definitions of these semistable coverings and semistable vertex sets, but before we do so, we owe the reader a brief remark regarding our use of both rigid and Berkovich geometry.

For any complete valued extension~$K$ of~$\bQ_\ell$, the category of separated Berkovich strictly $K$-analytic spaces embeds as a full subcategory of the category of rigid $K$-analytic spaces \cite[Proposition~3.3.1]{Berkovich1990SpectralTheoryAnalyticGeometry}, and most of the key notions of Berkovich and rigid geometry (finite maps, \'etale maps, \dots) correspond under this embedding \cite[\S3.3]{Berkovich1990SpectralTheoryAnalyticGeometry}. Every analytic space we consider will be a separated Berkovich strictly~$K$-analytic space, so we will switch freely between the Berkovich and rigid perspectives.

Thus, if~$X/K$ is a smooth projective curve, we will feel free to use the notation~$X^\an$ for both the rigid analytification and the Berkovich analytification of~$X$. When we write~$|X^\an|$, we always mean the underlying topological space of~$X^\an$ as a Berkovich space; the underlying set of~$X^\an$ as a rigid space is the subset~$|X^\an|_\rg\subset|X^\an|$ of \emph{rigid points}, i.e.\ points~$x\in|X^\an|$ whose completed residue field~$\dH(x)$ is a finite extension of~$K$.

\subsubsection{Semistable coverings}

The first alternative perspective we will use is that split semistable models of a smooth projective curve~$X/K$ are equivalent to analytic coverings of~$X$ by rigid spaces of a certain kind. This perspective goes back to Coleman \cite{ColemanReciprocityLaws}.

\begin{definition}[{\cite[\S2B]{McMurdyColeman}}]\label{def:wide_open}
	Let~$W$ be a $1$-dimensional smooth rigid space. We say that~$W$ is a \emph{wide open} just when there exist affinoid subdomains~$W_0\subset W_1\subset W$ such that:
	\begin{itemize}
		\item $W\smallsetminus W_0$ is a disjoint union of finitely many open annuli;
		\item $W_0$ is relatively compact in~$W_1$ (see \cite[\S9.6.2]{BoschGuentzerRemmert}); and
		\item $W_1$ meets each component of~$W\smallsetminus W_0$ in a semi-open annulus.
	\end{itemize}
	The affinoid subdomains~$W_0$ and~$W_1$ are not part of the data of a wide open. We refer to~$W_0$ as an \emph{underlying affinoid} of~$W$; a pair~$(W,W_0)$ of a wide open~$W$ and an underlying affinoid~$W_0$ is called a \emph{wide open pair}.

	A \emph{basic wide open} (resp.\ \emph{strongly basic wide open}) is a pair~$W=(W,W_0)$ consisting of a connected wide open~$W$ and an underlying affinoid~$W_0$, such that:
	\begin{itemize}
		\item the supremum seminorm of any element of~$\cO(W_0)$ is the norm of an element of~$K$;
		\item the canonical reduction of~$W_0$ is an irreducible split\footnote{The definition in \cite[Definition~2.35]{McMurdyColeman} does not explicitly use the word ``split'', but their definition of ``ordinary double point'' includes the requirement that it should be a split node \cite[Definition~2.9]{McMurdyColeman}.} semistable curve (resp.\ a smooth curve); and
		\item the components of $W\smallsetminus W_0$ are isomorphic over~$K$ to open annuli of inner radius~$1$.
	\end{itemize}
	(The first condition is equivalent to requiring that~$\cO(W_0)^\circ/\fm_K\cO(W_0)^\circ$ is a reduced ring \cite[Theorem~6.4.3/1 \& Proposition~6.4.3/4]{BoschGuentzerRemmert}, where~$\cO(W_0)^\circ$ is the ring of power-bounded functions on~$W_0$. In this case, $\Spec(\cO(W_0)^\circ/\fm_K\cO(W_0)^\circ)$ is the canonical reduction of~$W_0$.)
\end{definition}

Basic wide opens arise as complements of discs in smooth projective curves. Let~$X/K$ be a smooth projective curve and let~$\dX/\cO_K$ be a split semistable model whose special fibre is irreducible. Let~$\bar x_1,\dots,\bar x_n$ be distinct smooth $k$-points of the special fibre of~$\dX$ for~$n\geq1$, with associated residue discs $D_i^o\coloneqq{]}\bar x_i{[}\subset X^\an$. For~$1\leq i\leq n$, let~$D_i^c\subset D_i^o$ be a closed subdisc. Then $W\coloneqq X^\an\smallsetminus\bigcup_iD_i^c$ and $W_0\coloneqq X^\an\smallsetminus\bigcup_iD_i^o$ form a basic wide open~$(W,W_0)$, and the canonical reduction of~$W_0$ is canonically isomorphic to~$\bar\dX\smallsetminus\{\bar x_1,\dots,\bar x_n\}$. In particular, $W$ is strongly basic if the model~$\dX$ was smooth. See \cite[Proposition~2.21 \& Corollary~2.23]{McMurdyColeman}.

If~$W=(W,W_0)$ is a basic wide open, then we refer to the components of~$W\smallsetminus W_0$ as the \emph{bounding annuli} of~$W$ (they are sometimes called the \emph{annulus ends} of~$W$ in the literature). Also, for any singular point~$\bar x$ of the canonical reduction of~$W_0$, its residue class $]\bar x[$ is an open annulus \cite[Proposition~2.10]{McMurdyColeman}. We refer to such annuli as \emph{internal annuli} of~$W$. We fix some terminology we will use when working with these annuli.

\begin{definition}\label{def:annuli}
	Let~$A$ be an open annulus over~$K$, i.e.\ a rigid space isomorphic over~$K$ to the standard open annulus
	\[
	A(r_1,r_2) = \{x\::\: r_1<|x|<r_2\}
	\]
	for some~$r_1<r_2\in \sqrt{|K^\times|}$. The quantity $\log_\ell(r_2/r_1)\in\bQ$ is called the \emph{width}\footnote{Also called the \emph{modulus} of~$A$ in \cite{Baker2013NonArchimedeanAnalyticCurves}.} of~$A$, and is an isomorphism invariant (follows from \cite[Proposition~2.2(2)]{Baker2013NonArchimedeanAnalyticCurves}).

	If~$\cO(A)^\times_1$ denotes the group of invertible analytic functions~$f$ such that $|1-f(x)|<1$ for all points~$x$ of~$A$, then the group
	\[
	\cO(A)^\times/K^\times\cO(A)^\times_1
	\]
	is infinite cyclic \cite[Proposition~2.2(1)]{Baker2013NonArchimedeanAnalyticCurves}. An \emph{orientation} of~$A$ is a choice of generator of this group (so there are two possible orientations of~$A$). A \emph{parameter} on an oriented annulus~$A$ is an element $t\in\cO(A)^\times$ which maps to the chosen generator of $\cO(A)^\times/K^\times\cO(A)^\times_1$. The map $A\to\bG_m^\an$ induced by a parameter~$t$ is an isomorphism onto a standard annulus $A(r_1,r_2)$ \cite[Proposition~2.2(2)]{Baker2013NonArchimedeanAnalyticCurves}.
\end{definition}

\begin{remark}\label{rmk:bounding_annulus_orientation}
	The bounding annuli~$A_i$ of a wide open pair~$(W,W_0)$ are oriented in a canonical way. Namely, if~$t$ is a parameter on~$A_i$ inducing an isomorphism $A\xrightarrow\sim A(r_1,r_2)$, then~$t$ maps~$W_1\cap A_i$ isomorphically onto a semi-open annulus, either $A(r_1,r']$ or $A[r',r_2)$ for some~$r'\in(r_1,r_2)$. These two possibilities correspond to the two possible orientations of~$A_i$; we will always orient~$A_i$ with the orientation where $W_1\cap A_i$ is identified with the outer annulus $A[r',r_2)$. This is the same convention as \cite[Corollary~3.7a]{ColemanReciprocityLaws}.
\end{remark}

\begin{example}\label{ex:wide_open_in_projective_line}
	Suppose that we are given closed discs $D^c,D_1^c,\dots,D_n^c$ inside $\bA^{1,\an}_K$, contained inside open discs $D^o,D_1^o,\dots,D_n^o$, respectively. Suppose that $D_i^o\subset D^c$ for all~$i$ and the $D_i^o$ are pairwise disjoint. Then the domain $U=D^o\smallsetminus\bigcup_{i=1}^nD_i^c$ is wide open, with underlying affinoid $U_0=D^c\smallsetminus\bigcup_{i=1}^nD_i^o$. The wide open pair $(U,U_0)$ has $n+1$ bounding annuli, namely $A=D^o\smallsetminus D^c$ and $A_i=D_i^o\smallsetminus D_i^c$ for $1\leq i\leq n$. The orientation on each~$A_i$ is the standard one, i.e.\ a parameter is~$t-\alpha_i$ where~$t$ is the standard coordinate on~$\bA^{1,\an}_K$ and~$\alpha_i$ is a centre of~$D_i^c$. The orientation on~$A$, however, is the \emph{opposite} of the standard one, i.e.\ a parameter on~$A$ is $(t-\alpha)^{-1}$  where~$\alpha$ is a centre of~$D^c$.
\end{example}

\begin{figure}[!h]
	\centering
	\includegraphics[width=6cm]{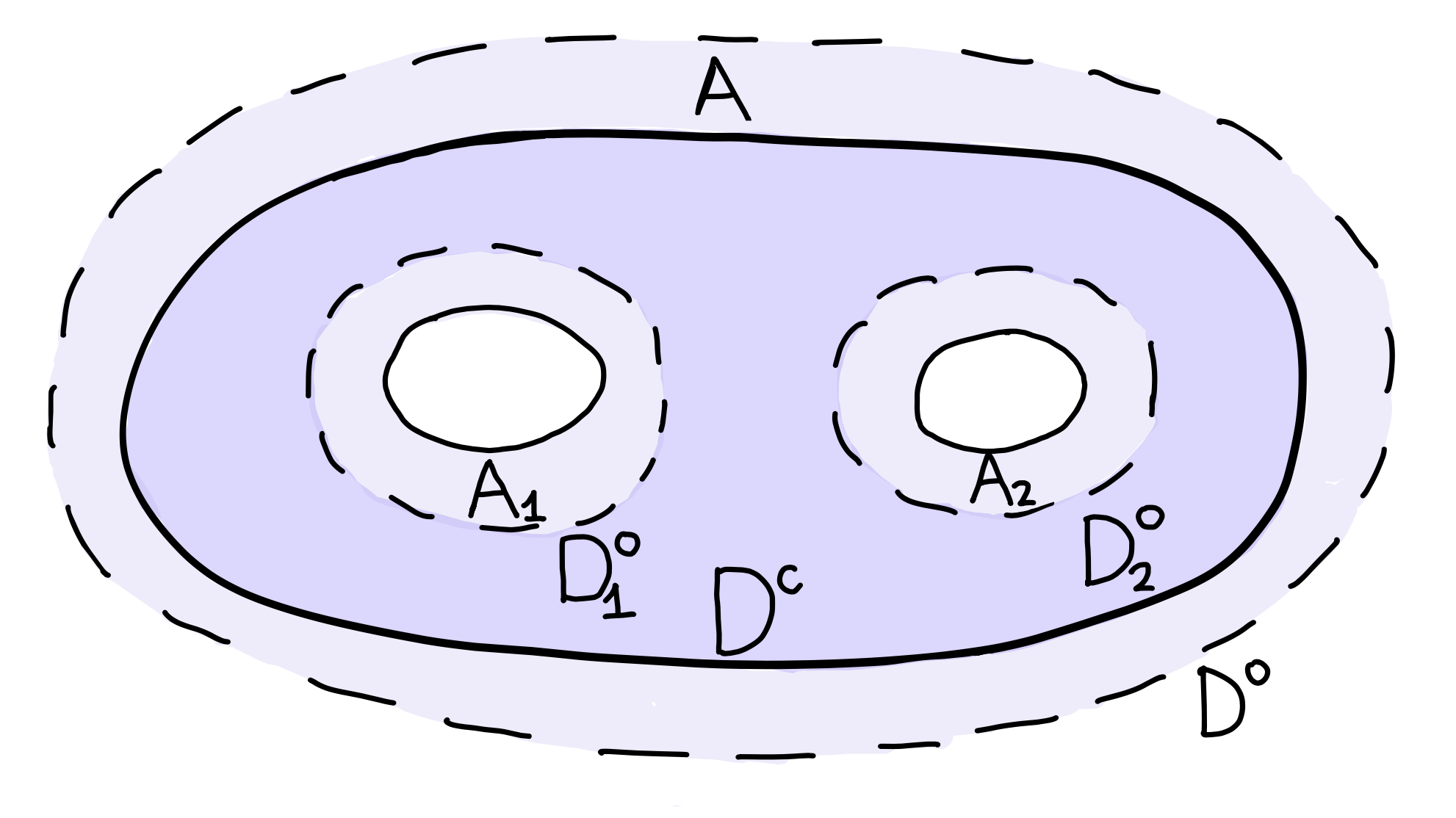}
	\caption{The wide open described in \Cref{ex:wide_open_in_projective_line} for $n = 2$. The open discs~$D^o$, $D_1^o$, and $D_1^o$ are represented by the dashed lines; the closed subdiscs~$D^c$, $D_1^c$, and~$D_2^c$ are represented by the solid lines. The open annuli~$A$, $A_1$, and $A_2$ are the lightly shaded regions between these open discs and their closed subdiscs, the wide open~$U$ is the whole shaded region, and its underlying affinoid~$U_0$ is the heavily shaded region.}
\end{figure}

One can study a curve $X/K$ with split semistable reduction by taking an analytic cover by wide opens which intersect each other in a well-behaved manner. Such coverings are known as semistable coverings.

\begin{definition}[{\cite[\S2C]{McMurdyColeman}}]\label{def:semistable_covering}
	Let~$X/K$ be a smooth projective curve. A \emph{(strongly) semistable covering} of~$X^\an$ is a finite admissible\footnote{\emph{Admissible} means that for any map $f: B \to X^\an$, the pullback $(f^*(W_v))_{v \in V}$ has a finite refinement by open affinoids.} covering~$\fU=(W_v)_{v\in V}$ of~$X^\an$ by (strongly) basic wide opens~$W_v=(W_v,W_{v,0})$ such that every pairwise intersection~$W_v\cap W_{v'}$ is a union of bounding annuli of~$W_v$, and every triplewise intersection~$W_v\cap W_{v'}\cap W_{v'}$ is empty.
\end{definition}

\begin{figure}[!h]
	\centering
	\includegraphics[width=7cm]{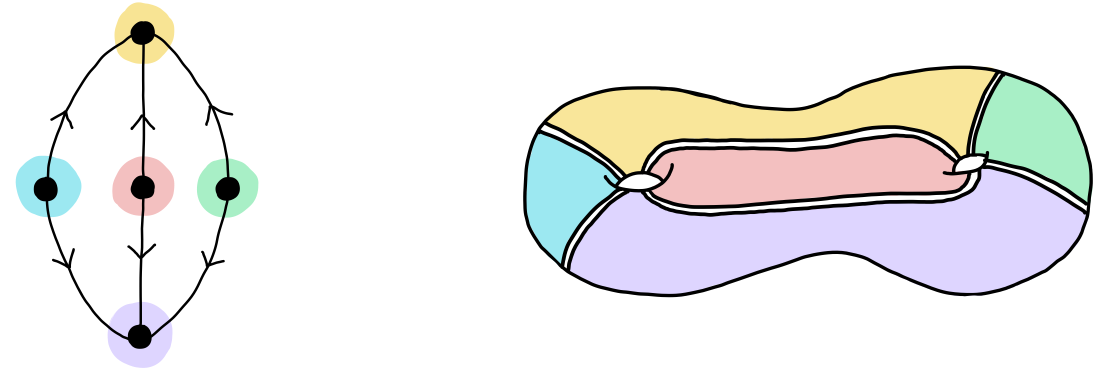}
	\caption{Caricature of a semistable covering and its attached graph (\Cref{def:graph_from_covering}).}
\end{figure}

For our purposes, the usefulness of semistable coverings is that they form a proxy for split semistable models, but are much easier to work with computationally. To explain the correspondence, if~$\dX$ is a split semistable model, then one has a reduction map
\begin{equation}\label{eq:rigid_reduction}
	\red\colon |X^\an|_\rg \to |\bar\dX|_\cl \,,
\end{equation}
where $|X^\an|_\rg$ is the underlying set of~$X^\an$ as a rigid space, and~$|\bar\dX|_\cl$ is the set of closed points of the special fibre~$\bar\dX$. The inverse image of any subset~$\bar\dZ\subseteq|\bar\dX|_\cl$ is denoted by $]\bar\dZ[$, and is called the \emph{residue class} (or \emph{tube}) of~$\bar\dZ$. It has the structure of a rigid space in a canonical way.

Provided that~$\bar\dX$ has at least two irreducible components, the residue class~$W_v\coloneqq{]}\bar\dX_v{[}$ of any irreducible component~$\bar\dX_v$ is a basic wide open. If we write~$\bar\dX_{v,0}\coloneqq\bar\dX\smallsetminus\bigcup_{v'\neq v}\bar\dX_{v'}$ for the complement of the other irreducible components of the special fibre, then the residue class~$W_{v,0}={]}\bar\dX_{v,0}{[}$ is an underlying affinoid, making~$W_v=(W_v,W_{v,0})$ into a basic wide open. The bounding annuli (resp.\ internal annuli) of~$W_v$ are the residue classes of the points where~$\bar\dX_v$ intersects other components of~$\bar\dX$ (resp.\ intersects itself). The wide opens~$W_v$ for~$\bar\dX_v$ running over irreducible components of~$\bar\dX$, form a semistable covering of~$X^\an$.

\begin{theorem}[{\cite[Theorem~2.36]{McMurdyColeman}}]\label{thm:semistable_models_and_coverings}
	Let~$K$ be either~$\bC_\ell$ or a finite extension of~$\bQ_\ell$, and let~$X/K$ be a smooth projective curve. Then the above construction sets up a bijective correspondence between split (strongly) semistable models of~$X$ whose special fibre has at least two irreducible components and (strongly) semistable coverings of~$X^\an$.
\end{theorem}

One can read off the reduction graph of a split semistable model~$\dX$ from the corresponding semistable covering. For this, we attach a metrised complex of $k$-curves to any semistable covering as follows.

\begin{definition}[{\cite[\S3.5.1]{ColemanIovitaHiddenStructures}}]\label{def:graph_from_covering}
	Suppose that~$\fU=(W_v)_{v\in V}$ is a semistable covering of~$X^\an$ indexed by a set~$V$. We define a metrised complex of $k$-curves $\Gamma=\Gamma_\fU$ by:
	\begin{itemize}
		\item $V(\Gamma)=V$ is the indexing set~$V$.
		\item $E(\Gamma)$ is the set of oriented open annuli $A\subset X^\an$ which are either a bounding annulus or an internal annulus of some~$W_v$. For~$e\in E(\Gamma)$, we write~$A_e$ for the corresponding oriented open annulus.
		\item For an oriented edge~$e\in E(\Gamma)$, we define~$\partial_0(e)$ to be the unique vertex for which~$A_e$ is either a bounding annulus of~$W_v$, equipped with the orientation described in \Cref{rmk:bounding_annulus_orientation}, or an internal annulus of~$W_v$. We define~$e^{-1}$ to be the edge for which~$A_{e^{-1}}$ is equal to~$A_e$ with the opposite orientation. (If~$A_e$ is a bounding annulus of~$W_v$, then it is a component of an intersection~$W_v\cap W_{v'}$ for~$v'\neq v$, and then~$A_{e^{-1}}$ is a bounding annulus of~$W_{v'}$.)
		\item For an unoriented edge~$e\in E(\Gamma)^+$, we define~$l(e)$ to be the width of the annulus~$A_e$.
		\item For a vertex~$v\in V(\Gamma)$, the canonical reduction of the underlying affinoid~$W_{v,0}$ is a reduced $k$-curve. We define~$\bar\dX_v$ to be the smooth compactification of its normalisation.
	\end{itemize}
\end{definition}

This construction recovers the reduction graph of a split semistable model, as we now prove carefully.

\begin{lemma}
	Let~$\fU$ be a semistable covering of~$X^\an$, corresponding to the split semistable model~$\dX$ under \Cref{thm:semistable_models_and_coverings}. Then we have a canonical isomorphism
	\[
	\Gamma_\fU \cong \Gamma_\dX
	\]
	of metrised complexes of $k$-curves.
\end{lemma}
\begin{proof}
	It is clear from the definition that there are canonical bijections
	\[
	V(\Gamma_\fU) \cong V(\Gamma_\dX) \quad\text{and}\quad E(\Gamma_\fU)^+ \cong E(\Gamma_\dX)^+
	\]
	between sets of vertices and unoriented edges, preserving incidence. There are three points which require some explanation:
	\begin{enumerate}
		\item How does one obtain a \emph{canonical} bijection $E(\Gamma_\fU)\cong E(\Gamma_\dX)$ between oriented edges?
		\item Why do the metrics on~$\Gamma_\fU$ and~$\Gamma_\dX$ agree?
		\item Why do the $k$-curves attached to vertices agree?
	\end{enumerate}

	For the first point, let~$\bar x\in\bar\dX(k)$ be a singular point, with residue class $]\bar x[$. Let us write~$||\cdot||$ for the sup norm on $\cO(]\bar x[)$, and $\cO^\circ(]\bar x[)\subset \cO(]\bar x[)$ for the subring of elements of sup norm~$\leq1$. It follows from \cite[Proposition~2.10]{McMurdyColeman} (or the argument below) that $]\bar x[$ is isomorphic to $A(|a|,1)$ for some~$a\in\fm_K\smallsetminus\{0\}$, and so any parameter~$t$ on~$]\bar x[$ can be rescaled so that $||t||=1$. We let~$s=at^{-1}$, so $||s||=1$ also. According to \cite[Lemma~2.8]{McMurdyColeman}, the reduction of~$\cO^\circ(]\bar x[)$ is canonically isomorphic to $\hat\cO_{\bar\dX,\bar x}$. So, the reductions of~$s$ and~$t$ determine elements of~$\hat\cO_{\bar\dX,\bar x}$ which generate the maximal ideal and whose product is zero (so $\hat\cO_{\bar\dX,\bar x}=k\llbrack \bar s,\bar t\rrbrack/(\bar s\bar t)$). The derivations $\frac{\partial}{\partial\bar s}$ and $\frac{\partial}{\partial\bar t}$ are then two tangent vectors to~$\bar\dX$ at~$\bar x$ which span the two tangent directions. Changing~$t$ by an element of~$K^\times\cO(]\bar x[)^\times_1$ only changes the tangent vector $\frac{\partial}{\partial\bar t}$ by a scalar, and so we have described a canonical bijection between the two orientations on the annulus $]\bar x[$ and the two tangent directions to~$\bar\dX$ at~$\bar x$ ($t$ corresponds to~$\frac{\partial}{\partial\bar t}$).

	For the second point, choose a chart
	\[
	\dX \leftarrow \dU \rightarrow \fS(a)
	\]
	for some~$a$. The choice of chart induces isomorphisms
	\[
	]\bar x[{} \xleftarrow\sim {}]\bar u[{} \xrightarrow\sim {}]\bar 0[{} = A(|a|,1)
	\]
	on the tubes of~$\bar x$, $\bar u$, and~$\bar 0$ inside the formal completions of~$\dX$, $\dU$, and~$\fS(a)$, respectively. So the width of $]\bar x[$ is equal to~$v(a)$, which is the thickness of~$\bar x$, and so the metrics on~$\Gamma_\fU$ and~$\Gamma_\dX$ agree.

	For the third point, for any vertex~$v$, the canonical reduction of~$W_{v,0}=\red^{-1}(\bar\dX_v^\circ)$ is $\bar\dX_v^\circ$, cf.\ the proof of \cite[Proposition~2.36]{McMurdyColeman}. So the smooth compactification of the normalisation of~$\overline W_{v,0}$ is~$\bar\dX_v$.
\end{proof}

In the proof, we used the following lemma, which is presumably well-known, but we could not find a reference in the literature.

\begin{lemma}
	Let~$\dU\to\dU'$ be an \'etale morphism of reduced, flat, locally finitely presented $\cO_K$-schemes, and let~$\bar u\in\bar\dU(k)$ be a $k$-point on the special fibre of~$\dU$ mapping to a point~$\bar u'\in\bar\dU'(k)$. Then the induced map
	\[
	]\bar u[{}\to{}]\bar u'[
	\]
	on residue classes is an isomorphism of rigid analytic spaces.
\end{lemma}
\begin{proof}
	Shrinking~$\dU'$ and~$\dU$, we may assume that $\dU'=\Spec(A)$ and $\dU=\Spec(B)$ are affine, and that~$\bar u$ is the unique point in the fibre above~$\bar u'$ in $\bar\dU\to\bar\dU'$. Choose~$f_1,\dots,f_n\in A$ whose reductions generate the ideal~$\fm_{\bar u'}$ defining~$\bar u'$. The residue class~$]\bar u'[$ of~$\bar u'$ is then the union of the affinoid rigid spaces $\Sp(K\otimes_{\cO_K}\hat A\langle\ell^{-1/m}f_1,\dots,\ell^{-1/m}f_n\rangle)$, where~$\hat A\langle\ell^{-1/m}f_1,\dots,\ell^{-1/m}f_n\rangle$ is the $\ell$-adic completion of the algebra
	\[
	A[\ell^{-1/m}f_1,\dots,\ell^{-1/m}f_n] \coloneqq A[z_1,\dots,z_n]/(\ell z_1-f_1^m,\dots,\ell z_n-f_n^m) \,.
	\]
	Since the reductions of~$f_1,\dots,f_m$ also generate the ideal in~$B$ defining~$\bar u$, a similar description holds for the residue class of~$\bar u$.

	Now, for any positive integers~$m$ and~$r$, the map
	\begin{align*}\label{eq:map_on_residue_classes}\tag{$\star$}
		A[\ell^{-1/m}f_1,\dots,\ell^{-1/m}f_n]/\ell^r \to B[\ell^{-1/m}f_1,\dots,\ell^{-1/m}f_n]/\ell^r
	\end{align*}
	is \'etale (it is a base-change of the map~$A\to B$), and its reduction modulo the ideal generated by~$f_1,\dots,f_n$ and~$\fm_K$ is an isomorphism (the identity on~$k[z_1,\dots,z_n]$). Since this ideal consists of nilpotent elements, this implies that~\eqref{eq:map_on_residue_classes} itself is an isomorphism \cite[Th\'eor\`eme~18.1.2]{EGA4.4}. Taking the inverse limit over~$r$ shows that the map $\hat A\langle\ell^{-1/m}f_1,\dots,\ell^{-1/m}f_n\rangle \to \hat B\langle\ell^{-1/m}f_1,\dots,\ell^{-1/m}f_n\rangle$ is an isomorphism, and so ${]}\bar u{[}\to{]}\bar u'{]}$ is an isomorphism as claimed.
\end{proof}

\subsubsection{Semistable vertex sets}\label{sss:vertex_sets}

For this part, we specialise to the case that~$K=\bC_\ell$. For a smooth projective curve $X/\bC_\ell$, write~$X^\an$ for the Berkovich analytification of~$X$, $|X^\an|$ for its underlying topological space, and~$|X^\an|_\II\subset|X^\an|$ for the set of type~\II points (see \cite[1.4.4]{Berkovich1990SpectralTheoryAnalyticGeometry} for a discussion of types of points, or \cite[\S2]{BakerIntroductionBerkovichSpaces} for a more accessible discussion in the case~$X=\bP^1$).

A \emph{semistable vertex set} in~$X^\an$ is a finite subset $V\subset|X^\an|_\II$ such that $X^\an\smallsetminus V$ is a disjoint union of open balls and finitely many open annuli \cite[Definition~3.1]{Baker2013NonArchimedeanAnalyticCurves}. These correspond to semistable models of~$X$ \cite[\S4]{Baker2013NonArchimedeanAnalyticCurves}. If~$\dX$ is a semistable model of~$X$, then there is an associated \emph{reduction map}
\begin{equation}\label{eq:berkovich_reduction}
	\red\colon |X^\an| \to |\bar\dX| \,,
\end{equation}
where~$|\bar\dX|$ denotes the underlying topological space of the special fibre. This extends the reduction map~\eqref{eq:rigid_reduction} on rigid points. The preimage of any generic point of~$\bar\dX$ under the reduction map is a single type~II point of~$X^\an$, and the set of all these points as we range over generic points of~$\bar\dX$ is a semistable vertex set \cite[Corollary~4.7 \&~Remark~4.2(2)]{Baker2013NonArchimedeanAnalyticCurves}.

\begin{theorem}[{\cite[Theorem~4.11]{Baker2013NonArchimedeanAnalyticCurves}}]\label{thm:semistable_models_and_vertex_sets}
	Let~$X/\bC_\ell$ be a smooth projective curve. Then the above construction sets up a bijective correspondence between semistable models of~$X$ and semistable vertex sets in~$X^\an$.
\end{theorem}

Again, one can also read off the reduction graph of a semistable model~$\dX$ from the corresponding semistable vertex set, as we will now recall. For this, it is convenient for us to view graphs as metric spaces in the natural way (cf.\ \cite[Definition~2.2]{AminiBakerLiftingHarmonicI}). That is, we can equivalently describe a metrised graph as a compact metric space\footnote{Our convention is that all metric spaces are length metric spaces, and the induced metric on a subspace means the induced length metric.}~$\Gamma$ together with a distinguished finite set~$V\subset \Gamma$ of \emph{vertices} such that~$\Gamma\smallsetminus V$ is isometric to a finite disjoint union of open intervals. The set $T_v(\Gamma)$ of oriented edges with source~$v$ is then identified with the set \emph{tangent directions} at~$v$: of germs of isometric embeddings $[0,\epsilon)\to\Gamma$ taking~$0$ to~$v$.

To produce a graph out of a semistable vertex set, recall that if~$A(r_1,r_2)$ is the standard annulus of inner and outer radii~$r_1$ and~$r_2$, then we define a map
\[
\sigma\colon (0,\log_\ell(r_2/r_1)) \to |A(r_1,r_2)|
\]
by sending~$s$ to the \emph{Gauss point} of the open ball centred on~$0$ with radius~$r_2\ell^{-s}$ \cite[\S2.3]{Baker2013NonArchimedeanAnalyticCurves}. Interpreting functions on $A(r_1,r_2)$ as power series $f = \sum_{i = - \infty}^{\infty} a_i t^i$, the Gauss point with radius $r$ is the valuation $v_s(f) = \min_i v(a_i) + i v(r)$.

The image of~$\sigma$ is called the \emph{skeleton} $\sk(A(r_1,r_2))$ of~$A(r_1,r_2)$. More generally, if~$A$ is an open annulus, then we may choose a parameter~$t$ defining an isomorphism $t\colon A\xrightarrow\sim A(r_1,r_2)$ for some~$r_1<r_2$, and then define the skeleton of~$A$ to be $\sk(A)\coloneqq t^{-1}\sk(A(r_1,r_2))$. The skeleton of~$A$ is independent of the choice of parameter~$t$ and is homeomorphic to the open interval~$(0,\log_\ell(r_2/r_1))$. If~$A$ is oriented, then~$\sk(A)$ is also oriented.

The \emph{skeleton} $\Gamma$ associated to a semistable vertex set~$V\subset|X^\an|_\II$ is the union of~$V$ and the skeleta of the annuli in~$|X^\an|\smallsetminus V$ \cite[Definition~3.3]{Baker2013NonArchimedeanAnalyticCurves}. For any annulus~$A$ in~$|X^\an|\smallsetminus V$, its skeleton~$\sk(A)$ is an open interval inside~$\Gamma$ whose closure is either a closed interval connecting two different elements in~$V$, or a circle connecting an element in~$V$ to itself \cite[Lemma~3.2]{Baker2013NonArchimedeanAnalyticCurves}. In this way, the skeleton $\Gamma$ is a topological graph with vertex-set~$V$.

In fact, the skeleton~$\Gamma$ is a metrised complex of $k$-curves in a natural way. If we write~$\bH_\circ(X)\subset|X^\an|$ for the set of points of types~\II and~\III, then there is a canonical length metric on~$\bH_\circ(X)$ constructed by Baker--Payne--Rabinoff \cite[\S5.3]{Baker2013NonArchimedeanAnalyticCurves}, and the restriction of this metric to~$\Gamma$ makes it into a metrised graph. This metric is characterised by the fact that the map~$\sigma$ tracing out the skeleton of an annulus is an isometry. Additionally, any element~$v$ of the semistable vertex set is a type~\II point of~$X^\an$, which means that the completed residue field~$\dH(v)$ is a complete valued extension of~$K$ whose residue field~$\tilde\dH(v)$ is a finitely generated extension of~$k$ of transcendence degree~$1$. We write~$\bar\dX_v$ for the unique smooth projective curve over~$k$ with function field~$\tilde\dH(v)$. Attaching to each point~$v\in V$ the curve~$\bar\dX_v$ above makes the skeleton~$\Gamma$ into a metrised complex of curves \cite[3.22]{AminiBakerLiftingHarmonicI}.

\begin{remark}\label{rmk:skeleton_independent_of_vertex_set}
	The perspective of semistable vertex sets makes it eminently clear that the reduction graph~$\Gamma$ associated to a semistable model of a curve~$X/\bC_\ell$ is independent of the choice of model~$\Gamma$, up to certain simple operations. Indeed, let~$V\subset|X^\an|_\II$ be a semistable vertex set, and let~$v\in|X^\an|_\II\smallsetminus V$ be another type~\II point. Then:
	\begin{itemize}
		\item if~$v\in\sk(A)$ lies in the skeleton of an open annulus~$A$ in~$X^\an\smallsetminus V$, then~$A\smallsetminus\{v\}$ is a disjoint union of two open annuli and infinitely many open balls;
		\item if~$v\in|B|$ lies in an open ball~$B$ in~$X^\an\smallsetminus V$, then~$B\smallsetminus\{v\}$ is a disjoint union of one open annulus and infinitely many open balls.
	\end{itemize}
	In either case, $V'\coloneqq V\cup\{v\}$ is another semistable vertex set in~$X^\an$. One can check (e.g.\ using \cite[Lemma~3.2(1)]{Baker2013NonArchimedeanAnalyticCurves}) that the skeleton~$\Gamma'$ associated to~$V'$ is obtained from the skeleton~$\Gamma$ associated to~$V$ by either:
	\begin{itemize}
		\item (edge subdivision) adding a new vertex (namely~$v$) at a point partway along an edge of~$\Gamma$, without changing the underlying metric space; or
		\item (budding off a leaf) adding a new vertex~$v$ and a single edge connecting~$v$ to a vertex in~$V$,
	\end{itemize}
	respectively. In either case, the curve~$\bar\dX_v$ attached to the new vertex~$v$ is isomorphic to~$\bP^1_k$.

	Any two semistable vertex sets can be obtained from one another by doing and undoing operations of this form, and so the skeleton~$\Gamma$ associated to a semistable vertex set~$V$ is independent of~$V$ up to the above operations.
\end{remark}

\subsubsection{Semistable coverings versus semistable vertex sets}

In the case $K=\bC_\ell$, we know that semistable models of~$X$ correspond bijectively to semistable vertex sets in~$X^\an$, and semistable models whose special fibre has at least two irreducible components correspond bijectively to semistable coverings of~$X^\an$. In particular, semistable vertex sets of size~$\geq2$ in~$X^\an$ correspond bijectively to semistable coverings of~$X^\an$. Since this will be important later, we now describe this correspondence directly, without going via semistable models.

\begin{proposition}\label{prop:semistable_vertex_set_to_covering}
	Let~$X/\bC_\ell$ be a smooth projective curve, and let~$\dX$ be a semistable model whose special fibre has at least two irreducible components. Let~$V\subset|X^\an|$ be the semistable vertex set corresponding to~$\dX$, and let~$\fU=(W_v)_{v\in V}$ be the semistable covering of~$X^\an$ corresponding to~$\dX$.

	Then for all~$v\in V$:
	\begin{itemize}
		\item the wide open~$W_v$ is the subspace of~$X^\an$ given by the union of~$\{v\}$ and all discs and annuli in $|X^\an|\smallsetminus V$ whose closure contains~$v$;
		\item the bounding annuli of~$W_v$ are exactly the oriented annuli in~$|X^\an|\smallsetminus V$ corresponding to non-loop edges~$e$ in the skeleton with~$\partial_0(e)=v$;
		\item the internal annuli of~$W_v$ are exactly the annuli in~$|X^\an|\smallsetminus V$ corresponding to loop edges~$e$ with endpoints~$v$; and
		\item the function field of the canonical reduction of $W_{v,0}$ is $\tilde\dH(v)$.
	\end{itemize}
\end{proposition}
\begin{proof}
	For the first point, if~$\bar\dX_v$ is the irreducible component of the special fibre corresponding to a vertex~$v$, then its inverse image under the reduction map is $v$ (inverse image of the generic point~$\bar\eta_v$) together with the open discs and annuli which are the inverse images of the closed points of~$\bar\dX_v$. Suppose first that~$\bar x\in\bar\dX_v(\Fbar_\ell)$ is a point which is smooth in~$\bar\dX$, so its residue class~$]\bar x[$ is an open disc. The topological boundary of~$]\bar x[$ inside $|X^\an|$ is~$\{v'\}$ for some~$v'\in V$ \cite[Lemma~3.2(1)]{Baker2013NonArchimedeanAnalyticCurves}. Since the reduction map is anti-continuous \cite[Corollary~2.4.2]{Berkovich1990SpectralTheoryAnalyticGeometry} and the subset~$\{\bar x,\bar\eta_v\}\subset|\bar\dX|$ is an intersection of open subsets, it follows that ${]}\bar x{[}\cup\{v\}$ is closed in~$|X^\an|$, and so we must have~$v'=v$, i.e.\ $v$ is the unique boundary point of~$]\bar x[$.

	If instead~$\bar x$ is singular in~$\bar\dX$, then its residue class~$]\bar x[$ is an annulus, whose topological boundary consists of one or two elements of~$V$ \cite[Lemma~3.2(2)]{Baker2013NonArchimedeanAnalyticCurves}. So, a similar argument establishes that if~$\bar x$ is a self-intersection point of~$\bar\dX_v$, then the topological boundary of~$]\bar x[$ is $\{v\}$, while if~$\bar x$ is an intersection point of~$\bar\dX_v$ and~$\bar\dX_{v'}$, then the topological boundary of~$]\bar x[$ is\footnote{Actually, there is a small amount of justification missing here, which seems also to be omitted in \cite[\S4.9]{Baker2013NonArchimedeanAnalyticCurves}. Specifically, the argument given shows that the boundary is contained inside~$\{v,v'\}$, but does not show that the two sets are equal. To show this, let~$\dX'$ be an admissible blowup of~$\dX$ centred at~$\bar x$ (or rather, at a closed, finitely presented subscheme whose reduced subscheme is~$\{\bar x\}$). The corresponding semistable vertex set~$V'$ consists of~$V$ and one new vertex~$v''$ lying on the skeleton of~$]\bar x[$. This splits the skeleton of~$]\bar x[$ into two open intervals, which are the skeleta of two annuli in~$X^\an\smallsetminus V'$ corresponding to intersections of the new component~$\bar\dX'_{v''}$ with~$\bar\dX'_v$ and~$\bar\dX'_{v'}$. Accordingly, the boundaries of these two annuli are contained in~$\{v,v''\}$ and~$\{v',v''\}$. This forces the two limit points of the skeleton of~$]\bar x[$ to be distinct, equal to~$v$ and~$v'$.}~$\{v,v'\}$. Put together, this tells us that the inverse image of~$\bar\dX_v$ under the reduction map is exactly the union of~$v$ and all of the open discs and annuli whose closures contain~$v$. This gives the first three points.

	For the final point, we first note that the underlying affinoid~$W_{v,0}$ of~$W_v$ is the union of~$\{v\}$ and all discs in~$|X^\an|\smallsetminus V$ whose closure contains~$v$. For any~$f\in \cO(W_{v,0})$, the maximum value of~$|f|$ on~$W_{v,0}$ is attained at~$v$, e.g.\ because~$|f|$ is a continuous function $|W_{v,0}|\to\bR$, and on each open disc in~$|W_{v,0}|\smallsetminus\{v\}$, $|f|$ increases monotonically towards the boundary point~$v$ (this follows from the fact that the maximum value of~$|f|$ on any \emph{closed} disc is attained at its Gauss point). In other words, the supremum norm on $R\coloneqq \cO(W_{v,0})$ is just the multiplicative norm~$|\cdot|_v$ attached to the Berkovich point~$v$.

	It then follows that the canonical reduction of~$W_{v,0}$ is $\Spec(R^\circ/R^{\circ\circ})$, where~$R^\circ$ (resp.\ $R^{\circ\circ}$) denotes the subring of~$R$ (resp.\ ideal of~$R^\circ$) consisting of elements of~$|\cdot|_v$-norm~$\leq1$ (resp.\ $<1$). The residue field~$\tilde\dH(v)$, on the other hand, is $\Frac(R)^\circ/\Frac(R)^{\circ\circ}$, where the multiplicative norm~$|\cdot|_v$ is extended uniquely to~$\Frac(R)$ \cite[Remark~1.2.2(i)]{Berkovich1990SpectralTheoryAnalyticGeometry}. Thus, we want to show that~$\Frac(R)^\circ/\Frac(R)^{\circ\circ}$ is the fraction field of~$R^\circ/R^{\circ\circ}$. There is certainly an $\Fbar_\ell$-algebra homomorphism
	\begin{align*}\label{eq:double_residue_field}\tag{$\ast$}
	R^\circ/R^{\circ\circ} \to \Frac(R)^\circ/\Frac(R)^{\circ\circ} \,.
	\end{align*}
	An element of the kernel of this homomorphism would be represented by some~$f\in R^\circ$ for which~$|\frac{f}{1}|_v<1$, which implies that~$f=0$ in~$R^\circ/R^{\circ\circ}$. So~\eqref{eq:double_residue_field} is injective, and thus induces a map
	\begin{align*}\label{eq:double_residue_field_fields}\tag{$\ast\ast$}
		\Frac(R^\circ/R^{\circ\circ}) \to \Frac(R)^\circ/\Frac(R)^{\circ\circ} \,.
	\end{align*}
	Given any element $\frac fg\in\Frac(R)^\circ$, we can rescale~$f$ and~$g$ by elements of~$\bC_\ell$ so that~$|g|_v=1$ and~$|f|_v\leq1$. So~$f$ and~$g$ determine an element $\frac{\bar f}{\bar g}\in\Frac(R^\circ/R^{\circ\circ})$ mapping to the class of~$\frac fg$ in~$\Frac(R)^\circ/\Frac(R)^{\circ\circ}$. Thus~\eqref{eq:double_residue_field_fields} is surjective, i.e.\ is an isomorphism of fields over~$\Fbar_\ell$. This is what we wanted to prove.
\end{proof}

As a consequence, the skeleton attached to the semistable vertex set~$V$ agrees with the graph attached to the semistable covering~$\fU$ as a metrised complex of $\Fbar_\ell$-curves (cf.\ \cite[\S4.9]{Baker2013NonArchimedeanAnalyticCurves} for the statement without metrics or vertex curves).

In \cite[Definition~6.2]{AminiBakerLiftingHarmonicI}, the notion of a \emph{star-shaped curve} is introduced, which is a Berkovich space with a marked type \II point isomorphic to a good reduction curve minus a finite number of closed discs. As explained in \cite[\S2.2]{helminck:skeletal_filtrations}, if~$V\subset|X^\an|$ is a strongly semistable vertex set, then for all~$v\in V$, the component of~$X^\an\smallsetminus(V\smallsetminus\{v\})$ containing~$v$ is a star-shaped curve, and these star-shaped curves together cover~$X^\an$. We obtain the following corollary of \Cref{prop:semistable_vertex_set_to_covering}.
\begin{corollary}
Let~$X/\bC_\ell$ be a smooth projective curve. Strongly semistable coverings of $X^\an$ are in bijective correspondence with coverings by affinoid star-shaped curves.
\end{corollary}

\subsection{Harmonic morphisms of metrised complexes of curves}\label{ss:harmonic_morphisms}

The reduction graph~$\Gamma$ attached to a curve~$X/K$ via a choice of split semistable model is functorial with respect to finite morphisms~$f\colon X\to X'$ of curves, at least after suitable choices of models. This functoriality turns out to be surprisingly subtle, so we devote some time to explaining this carefully. The first subtlety lies in the correct notion of morphisms of graphs.

\begin{definition}[{\cite[Definitions~2.4 \&~2.19]{AminiBakerLiftingHarmonicI}}]
	\label{def:finitegraphmap}
	Let~$\Gamma$ and~$\Gamma'$ be metrised graphs. A \emph{finite morphism} $f\colon \Gamma \to \Gamma'$ is a pair of functions $V(\Gamma) \to V(\Gamma')$ and $E(\Gamma) \to E(\Gamma')$ compatible with edge-inversion and source maps, such that the quantity
	\[
	d_e(f) \coloneqq \frac{l(f(e))}{l(e)}
	\]
	is a positive integer for all edges~$e$. The quantity~$d_e(f)$ is called the \emph{degree} of~$f$ along~$e$.

	Equivalently, from the metric perspective, a finite morphism $f\colon \Gamma \to \Gamma'$ is a continuous map $f\colon \Gamma \to \Gamma'$ such that $f^{-1}(V(\Gamma'))=V(\Gamma)$, and such that every connected component~$e$ of~$\Gamma\setminus V(\Gamma)$ (which is isometric to an open interval and maps homeomorphically onto a connected component $f(e)$ of~$\Gamma'\setminus V(\Gamma')$) maps onto its image via a dilation by the factor $d_e(f)$.
	A finite morphism $f\colon\Gamma\to\Gamma'$ of metrised graphs is said to be \emph{harmonic} of \emph{degree $d_v(f)$} at $v \in V(\Gamma)$ if for every $e' \in T_{f(v)}(\Gamma')$ we have
	\[
	\sum_{\substack{e \in T_v(\Gamma) \\ f(e) = e'}} d_e(f) = d_v(f).
	\]
	The map $f$ is said to be \emph{harmonic} if it is surjective and harmonic at every $v \in V(\Gamma)$. Then, for any $v' \in V(\Gamma')$ the sum $\deg(f) = \sum_{v\colon f(v) = v'} d_v(f)$ is independent of $v'$ and called the \emph{degree} of $f$.

	A \emph{finite harmonic morphism} $f\colon\Gamma \to \Gamma'$ of metrised complexes of $k$-curves is a finite morphism between the underlying metrised graphs endowed with a choice of finite morphism
	\[
	\fbar_v\colon \bar\dX_v \to \bar\dX'_{f(v)}
	\]
	of $k$-curves for each $v\in V(\Gamma)$ satisfying the following three conditions.
	\begin{enumerate}
		\item For every vertex $v \in V(\Gamma)$ and every $e \in T_v(\Gamma)$, we have that $\bar x_{f(e)} = \fbar_v(\bar x_e)$, and the map $\fbar_v$ is ramified of degree $d_e(f)$ at $\bar x_e$.
		\item Conversely, for every vertex $v\in V(\Gamma)$ and every $e' \in T_{f(v)}(\Gamma')$, every point in~$\fbar_v^{-1}(\bar x_{e'})$ is $\bar x_e$ for some~$e\in T_v(\Gamma)$ with~$f(e)=e'$.
		\item For every vertex $v \in V(\Gamma)$ we have $d_v(f) = \deg(\fbar_v)$. (This is automatic from the preceding conditions as soon as~$\Gamma$ has at least one edge.)
	\end{enumerate}
\end{definition}

A finite morphism~$f\colon X\to X'$ of smooth projective curves induces, after suitable choices, a finite harmonic morphism $f\colon \Gamma\to\Gamma'$ between their reduction graphs. Describing this in terms of split semistable models is rather subtle -- for example, a finite morphism $\dX\to\dX'$ between models need not induce a finite map on reduction graphs -- so we instead follow \cite{AminiBakerLiftingHarmonicI} and describe this in terms of semistable vertex sets. The key result is the following.
\begin{theorem}[{\cite[Theorem~A, Corollary~4.26 \& \S4.27]{AminiBakerLiftingHarmonicI}}]\label{thm:harmonic_maps_on_skeleta}
	Let~$f\colon X\to X'$ be a finite morphism of smooth projective curves over~$\bC_\ell$. Let~$V'_0$ be a finite set of type~\II points in~$X^{\prime\an}$. Then there exists a semistable vertex set~$V'\subset|X^{\prime\an}|_\II$ for~$X^{\prime\an}$ containing~$V'_0$ such that the preimage~$V\coloneqq f^{-1}(V')\subset|X^\an|_\II$ is a semistable vertex set for~$X^\an$ and such that the skeleton~$\Gamma$ associated to~$V$ is the preimage of the skeleton~$\Gamma'$ associated to~$V'$.

	Moreover, for any such~$V'$, the induced map~$f\colon\Gamma\to\Gamma'$ on skeleta is a finite harmonic morphism of metrised complexes of $k$-curves, of degree equal to~$\deg(f)$.
\end{theorem}
In the final part, for each vertex $v\in\Gamma$, the morphism $\fbar_v\colon\bar\dX_v\to\bar\dX'_{f(v)}$, part of the data of a morphism of metrised complexes of $k$-curves, is the one whose induced map on function fields is the pullback map $\tilde\dH(f(v))\to\tilde\dH(v)$ induced by $f\colon X^\an\to X^{\prime\an}$, see \cite[\S4.20]{AminiBakerLiftingHarmonicI}.

\subsubsection{Morphisms of curves and semistable coverings}

We will also need to translate \Cref{thm:harmonic_maps_on_skeleta} into the language of semistable coverings.

Let~$f\colon X\to X'$ be a finite morphism of smooth projective curves over~$\bC_\ell$, and choose a semistable vertex set~$V'\subset|X^{\prime\an}|_\II$ satisfying the conditions of \Cref{thm:harmonic_maps_on_skeleta}. Enlarging~$V'$ if necessary, we may assume that~$\#V'\geq2$ and that~$V'$ is \emph{strongly} semistable, meaning that the corresponding skeleton is loopless. This implies the same conditions for the semistable vertex set~$V\coloneqq f^{-1}V'\subset|X^\an|_\II$. So there are associated semistable coverings~$\fU=(W_v)_{v\in V}$ and~$\fU=(W'_{v'})_{v'\in V'}$ of~$X^\an$ and~$X^{\prime\an}$.

The relationship between semistable coverings and semistable vertex sets in \Cref{prop:semistable_vertex_set_to_covering} implies that the semistable coverings~$\fU$ and~$\fU'$ are compatible in the following sense.

\begin{proposition}\label{prop:morphisms_and_coverings}
	In the setup of \Cref{thm:harmonic_maps_on_skeleta}, if~$\#V'\geq2$ and~$V'$ is strongly semistable, then:
	\begin{itemize}
		\item for any vertex~$v'\in V'$, we have
		\[
		f^{-1}W'_{v'} = \bigcup_{v\in f^{-1}(v')}W_v \,,
		\]
		and for any $v\in f^{-1}(v')$, the restriction of~$f$ to a map $W_v\to W'_{v'}$ is a finite morphism of degree~$d_v(f)$;
		\item for any oriented edge~$e'$ of the skeleton attached to~$V'$, we have
		\[
		f^{-1}A'_{e'} = \bigcup_{e\in f^{-1}(e')}A_e \,,
		\]
		and for any $e\in f^{-1}(e')$, the restriction of~$f$ to a map~$A_e\to A'_{e'}$ is a finite morphism of degree~$d_e(f)$ which preserves orientations; and
		\item for any vertex~$v'\in V'$, we have
		\[
		f^{-1}W'_{v',0} = \bigcup_{v\in f^{-1}(v')}W_{v,0} \,,
		\]
		and for any $v\in f^{-1}(v')$, the restriction of~$f$ to a map $W_{v,0}\to W'_{v',0}$ is a finite morphism of degree~$d_v(f)$ whose canonical reduction is the restriction of the map $\fbar_v\colon\bar\dX_v\to\bar\dX'_{f(v)}$.
	\end{itemize}
	\begin{proof}
		The statements regarding the inverse images follow from the descriptions in \Cref{prop:semistable_vertex_set_to_covering}. The fact that the maps on wide opens and bounding annuli are finite is immediate since~$f$ is finite. For the assertions regarding the degree, consider first the induced morphism $f|_{A_e}\colon A_e\to A'_{e'}$ on bounding annuli. We know that the induced map on skeleta is orientation-preserving, and is dilation by the scale factor~$d_e(f)$ by definition. Using the classification of morphisms between annuli in \cite[Proposition~2.2(2)]{Baker2013NonArchimedeanAnalyticCurves}, this implies that~$f|_{A_e}$ has degree~$d_e(f)$ as claimed. By harmonicity, this implies that the map $f|_{W_v}\colon W_v\to W'_{v'}$ has degree~$d_v(f)$ (compute the total degree over any point in a bounding annulus).

		For the final assertion, for any vertex~$v\in V$, the morphism $f\colon X^\an\to X^{\prime\an}$ on Berkovich analytifications induces an extension $\dH(f(v))\hookrightarrow\dH(v)$ on completed residue fields, and so an extension $\tilde\dH(f(v))\hookrightarrow\tilde\dH(v)$ on their residue fields. The corresponding map $\bar\dX_v\to\bar\dX'_{f(v)}$ of smooth projective $\Fbar_\ell$-curves is, by definition, the morphism~$\fbar_v$ \cite[\S4.20]{AminiBakerLiftingHarmonicI}. So it follows from \Cref{prop:semistable_vertex_set_to_covering} that this~$\fbar_v$ restricts to the canonical reduction of~$f|_{W_{v,0}}\colon W_{v,0}\to W_{f(v),0}'$. This completes the proof.
	\end{proof}
\end{proposition}

\section{The local heights formula}
\label{sec:localheights}
If~$X$ is a smooth projective curve over a finite extension~$K$ of~$\bQ_\ell$ with a chosen basepoint~$b\in X(K)$, then the \emph{normalised local height} $\tilde h_{Z,\ell}$ with respect to a trace zero correspondence~$Z\subset X\times X$ is defined as follows. Let~$z\in X(K)$ be a $K$-rational point, and define a divisor~$D_Z(b,z)\subset X$ by
\begin{equation}\label{eq:DZ}
D_Z(b,z) \coloneqq i_\Delta^*Z - i_1^*Z - i_2^*Z \,,
\end{equation}
where~$i_\Delta\colon X\hookrightarrow X\times X$ is the inclusion of the diagonal, and $i_1,i_2\colon X\hookrightarrow X\times X$ are the inclusions of the subvarieties $\{b\}\times X$ and $X\times\{z\}$ \cite[Definition~6.2]{QCI}\footnote{Equivalently, we have $D_Z(b,z) \coloneqq i_2^* Z'$ where $Z'$ is the unique correspondence such that it differs from $Z$ by vertical and horizontal correspondences, and is trivial along the diagonal and $b \times X$, and the degree of the fibers of the second projection are $0$. See \cite[Algorithm 4.3, Lemma 4.4]{GQCHeights}}. Let~$\calX/\cO_K$ be a regular proper model, and let~$\dD_Z(b,z)\subset\calX$ be a divisor with coefficients in~$\bQ$ whose generic fibre is~$D_Z(b,z)$ and whose intersection multiplicity with any vertical divisor is zero. Then the normalised local height is given by the intersection multiplicity
\[
\tilde h_{Z,\ell}(z) \coloneqq (z-b)\cdot\dD_Z(b,z) \in \bQ \,,
\]
where we extend~$(z-b)$ to a divisor on $\dX$ in the obvious way. This height has the property that $\chi(\varpi_K)\tilde h_{Z,\ell}(z)$ is the Coleman--Gross height pairing of~$z-b$ and $D_Z(b,z)$, where~$\chi\colon K^\times\to\bQ_p$ is the character appearing in the definition of the Coleman--Gross height \cite[Proposition~1.2]{ColemanGross}.

As a result of \cite[Corollary~0.2]{KimTamagawa} the functions $\tilde h_{Z,\ell}(z)$ take only finitely many values.
\begin{remark}
\label{rem:intersectionpair}
There are algorithms for computing the intersection pairing on divisors given a regular model for $X$ over $\bZ_\ell$ \cite{Holmes,Mueller,vanBommelHolmesMueller}.
These algorithms depend on equations for $Z$ and $X$. 
Equations for $Z$ can be extremely complicated even for simple genus $2$ curves \cite[Appendix]{GQCHeights}, and so the complexity of the computation skyrockets as the genus and the conductor of $X$ grows. 
In principle, these intersection pairing methods have the potential to work, but from a practical standpoint, working with equations for the correspondence $Z$ to construct an explicit divisor on a regular model is difficult.
 We did attempt to use these intersection theory algorithms to compute local heights on a genus $2$ curve with RM, but were unable to surmount the difficulties that arose in practice with dealing with the complicated equations for the correspondence. 
 As far as we know, there are no examples in the literature where these methods have been used to compute local heights in the setting of quadratic Chabauty for rational points.
\end{remark}

Our approach to computing local heights in this paper is based on a formula due to the first author and Netan Dogra \cite[Corollary~12.1.3]{BettsDogra}, which avoids the need to compute a regular model over~$K$, and instead gives a direct combinatorial interpretation of $\tilde h_{Z,\ell}(z)$ in terms of reduction graphs. In this section, we will recall the statement of this formula.

\subsection{Homology of the reduction graph}\label{ss:homology}

Let~$\Gamma$ denote the metrised reduction graph of a split semistable model of the curve~$X/K$, thought of as a metric space as in \S\ref{sss:vertex_sets}. Equivalently, $\Gamma$ isthe skeleton associated with a semistable vertex set in~$X_{\bC_\ell}^\an$ (after embedding~$K$ inside~$\bC_\ell$). The local heights formula will be expressed primarily in terms of the homology group~$\rH_1(\Gamma,\bZ)$, which can be described explicitly in terms of the vertices and edges of~$\Gamma$ as follows. Let~$\bZ\cdot V(\Gamma)$ denote the free $\bZ$-module on the set~$V(\Gamma)$ of vertices, and let~$\bZ\cdot E(\Gamma)^-$ denote the free $\bZ$-module on the set~$E(\Gamma)$ of oriented edges, modulo the relation $e+e^{-1}=0$. There is a chain complex
\[\cdots \to 0\to \bZ\cdot E(\Gamma)^-\xrightarrow{\partial}\bZ\cdot V(\Gamma) \to 0\to\cdots \,,\]
where $\partial=\partial_1-\partial_0$ with~$\partial_0$ and~$\partial_1$ the source and target functions of \eqref{eq:partial0}. The homology group $\rH_1(\Gamma,\bZ)$ is the first homology group of this complex, i.e.~it is the group of formal linear combinations of edges with boundary~$0$. The cohomology group $\rH^1(\Gamma,\bZ)$ is the first cohomology group of the dual cochain complex, i.e.~it is the group of functions~$f\colon E(\Gamma)\to\bZ$ satisfying $f(e^{-1})=-f(e)$, modulo the subspace generated by functions of the form
\[
f(e) \coloneqq 
\begin{cases}
	1 & \text{if $\partial_0(e)=v$ and~$\partial_1(e)\neq v$,} \\
	-1 & \text{if $\partial_1(e)=v$ and~$\partial_0(e)\neq v$,} \\
	0 & \text{otherwise.}
\end{cases}
\]

The homology group~$\rH_1(\Gamma,\bZ)$ comes with a perfect $\bR$-valued pairing, depending only on the underlying metric space of the graph~$\Gamma$, and combinatorial in nature. This is induced by the pairing on oriented edges given by
\[
\langle e_1,e_2\rangle\coloneqq
\begin{cases}
	l(e_1) & \text{if $e_2=e_1$,} \\
	-l(e_1) & \text{if $e_2=e_1^{-1}$,} \\
	0 & \text{else.}
\end{cases}
\]
This is a positive-definite pairing on $\rH_1(\Gamma,\bZ)$, called the \emph{intersection length pairing} (or \emph{cycle pairing} in \cite{Chenge2021CombinatorialIntegralsHarmonic}).

The homology group~$\rH_1(\Gamma,\bZ)$ comes with a perfect $\bR$-valued pairing, depending only on the underlying metric space of the graph~$\Gamma$, and combinatorial in nature. This is induced by the pairing on oriented edges given by
\[
\partial\colon \rC_1(\Gamma,\bZ) \to \rC_0(\Gamma,\bZ)
\]
sending an oriented edge~$e$ to~$\partial(e)=\partial_0(e)-\partial_0(e^{-1})$.

There is a canonical positive-definite symmetric pairing
$\langle\cdot,\cdot\rangle\colon\Sym^2\rC_1(\Gamma,\bZ) \to \bR$
given by
\[
\langle e_1,e_2\rangle\coloneqq
\begin{cases}
	l(e_1) & \text{if $e_2=e_1$,} \\
	-l(e_1) & \text{if $e_2=e_1^{-1}$,} \\
	0 & \text{else.}
\end{cases}
\]
Since this pairing is positive-definite, its restriction to $\rH_1(\Gamma,\bZ)\subset\rC_1(\Gamma,\bZ)$ also defines a positive-definite symmetric pairing on $\rH_1(\Gamma,\bZ)$, called the \emph{intersection length pairing} (or \emph{cycle pairing} in \cite{Chenge2021CombinatorialIntegralsHarmonic}). We remark that when~$\Gamma$ is the reduction graph associated to a split semistable model of~$X/K$, the intersection length pairing on~$\rH_1(\Gamma,\bZ)$ is also independent of the choice of model.

\subsubsection{Functoriality}

The assignment of the group~$\rH_1(\Gamma,\bZ)$ is functorial with respect to finite morphisms of curves, both co- and contravariantly.

\begin{proposition}
	The assignment $X\mapsto\rH_1(\Gamma,\bZ)$ is a functor
	\[
	\Bigl\{\parbox{3cm}{curves over~$K$ \& finite morphisms}\Bigr\} \to \Bigl\{\parbox{6cm}{lattices with an inner product \& adjoint pairs of homomorphisms}\Bigr\} \,.
	\]
\end{proposition}
\begin{proof}
The covariant functoriality is clear: given a finite morphism $f\colon X\to X'$ of curves, without loss of generality defined over~$\bC_\ell$, we have by \Cref{thm:harmonic_maps_on_skeleta} that~$f$ induces a finite harmonic morphism $\Gamma\to\Gamma'$ on suitably chosen skeleta, and hence a homomorphism
\[
f_*\colon \rH_1(\Gamma,\bZ) \to \rH_1(\Gamma',\bZ)
\]
(which again is independent of the choice of skeleta).

We can also define a map
\[
f^*\colon\rH_1(\Gamma',\bZ)\to\rH_1(\Gamma,\bZ)
\]
in the other direction, namely the adjoint of~$f_*$ with respect to the intersection length pairings on the homology groups. One thing which is not apparent from the definition is that~$f^*$ is defined over~$\bZ$, rather than over~$\bR$ or~$\bQ$. This is a consequence of Proposition~\ref{prop:pullback_formula}.
\end{proof}

\begin{proposition}\label{prop:pullback_formula}
	Let~$f\colon \Gamma \to \Gamma'$ be a finite harmonic morphism of metrised graphs. Then the pullback map
	\[
	f^*\colon \rH_1(\Gamma',\bR) \to \rH_1(\Gamma,\bR)
	\]
	(adjoint to the pushforward) is given by
	\[
	f^*(\sum_{e'}\lambda_{e'}\cdot e') = \sum_e\lambda_{f(e)}d_e(f)\cdot e \,.
	\]
	In the above, summations over~$e$ or~$e'$ implicitly mean sums over unoriented edges of~$\Gamma$ or~$\Gamma'$ (the summands are independent of the orientation of edges).

	In particular, since each~$d_e(f)$ is an integer, $f^*$ restricts to a map
	\[
	f^*\colon \rH_1(\Gamma',\bZ) \to \rH_1(\Gamma,\bZ) \,.
	\]
	\begin{proof}
		For a homology class
		\[
		\gamma' = \sum_{e'}\lambda_{e'}\cdot e'\in\rH_1(\Gamma',\bR) \,,
		\]
		let us write
		\[
		f'(\gamma') \coloneqq \sum_e\lambda_{f(e)}d_e(f)\cdot e \,.
		\]
		We claim that $f'(\gamma')\in\rH_1(\Gamma,\bR)$. For this, we compute
		\begin{align*}
			\partial(f'(\gamma')) &= \sum_e\lambda_{f(e)}d_e(f) \partial(e) \\
			&= \sum_v\sum_{e\in T_v(\Gamma)}\lambda_{f(e)}d_e(f)\cdot v \\
			&= \sum_v\sum_{e'\in T_{f(v)}(\Gamma')}\lambda_{e'}\sum_{e\in T_v(\Gamma),f(e)=e'}d_e(f)\cdot v \\
			&= \sum_v\sum_{e'\in T_{f(v)}(\Gamma')}\lambda_{e'}\cdot\deg_v(f)\cdot v = 0
		\end{align*}
		using harmonicity and the fact that $\sum_{e'\in T_{v'}(\Gamma')}\lambda_{e'}=0$ for all~$v'\in V(\Gamma')$ since~$\gamma'$ is a homology class. Hence $f'(\gamma')\in\rH_1(\Gamma',\bR)$ as claimed.

		Now for any
		$\gamma = \sum_e\mu_e\cdot e\in \rH_1(\Gamma,\bR)$
		we have
		\begin{align*}
			\langle f'(\gamma'),\gamma\rangle &= \sum_e\lambda_{f(e)}\mu_e\cdot l(f(e))
			= \sum_{e'}\lambda_{e'}\left(\sum_{f(e)=e'}\mu_e\right)\cdot l(e')
			= \langle\gamma',f_*(\gamma)\rangle
		\end{align*}
		using that~$d_e(f)=\frac{l(f(e))}{l(e)}$ in the first line. Hence~$f'=f^*$ as desired.
	\end{proof}
\end{proposition}

\subsubsection{Correspondences and component traces}\label{subsubsec:correspondences}

Combining the pushforward and pullback functoriality, we find that the homology group
$\rH_1(\Gamma,\bZ)$
is functorial with respect to correspondences between curves. For our purposes, by a \emph{correspondence} from~$X$ to~$Y$, we mean a Weil divisor~$Z\subset X\times Y$. The pushforward~$Z_*$ along a correspondence is defined as follows. If~$Z$ is a geometrically integral divisor and the projections $\pi_1\colon Z\to X$ and $\pi_2\colon Z\to Y$ are finite, then we define $Z_* \coloneqq \tilde\pi_{2*}\circ\tilde\pi_1^*$
where~$\tilde\pi_1$ and~$\tilde\pi_2$ are the two projections from the normalisation~$\tilde Z$ of~$Z$. If~$Z$ is geometrically integral and one projection is not finite, then we set~$Z_*=0$. In general, we first extend the base field so that all integral components of~$Z$ are geometrically integral, and then extend by linearity.
We will be especially interested in correspondences~$Z$ from a curve~$X$ to itself.

These correspondences also have so-called traces on the components, which we will need later for the local height formula.

\begin{definition}\label{def:component_traces}
	Let $Z\subset X\times X$ be a correspondence defined over~$\bC_\ell$, and let~$\Gamma$ be a skeleton in~$X^\an$. Let $v \in V(\Gamma)$, and let $S \subset |Z^\an|_\II$ be the set of points $w$ that map to $(v,v)$. If we write~$\bar\dZ_w$ for the $\Fbar_\ell$-curve attached to a point~$w\in S$, then $\bar\dZ_{v,v} \coloneqq \bigsqcup_{w \in S} \bar\dZ_w$ is an effective divisor inside $\bar\dX_v \times \bar\dX_v$, with projections $\pi_1, \pi_2$ to $\bar\dX_v$ of degree $d_1$ and $d_2$ in total. We define the \emph{trace} at $v$ to be $\tr_v(Z) = d_1 + d_2 - \Delta \cdot \bar\dZ_{v,v} \in \bZ$, where $\Delta$ is the diagonal in $\bar\dX_v \times \bar\dX_v$.
\end{definition}

The above definition makes it clear that the component traces~$\tr_v(Z)$ are integers. They can also be interpreted as traces of the action of~$\bar\dZ_{v,v,*}$ on cohomology, explaining why we give them the name of ``traces''.

\begin{lemma}\label{lem:traceonH1}
	In the above notation, $\tr_v(Z)$ is equal to the trace of~$\bar\dZ_{v,v,*}$ acting on the rigid cohomology group~$\rH^1_\rig(\bar\dX_v/\bC_\ell)$. For~$\ell'\neq\ell$ it is also the trace of~$\bar\dZ_{v,v,*}$ acting on the \'etale cohomology group~$\rH^1_\et(\bar\dX_{v,\kbar},\bQ_{\ell'})$.
\end{lemma}
\begin{proof}
	The Lefschetz trace formula for rigid cohomology \cite[Th\'eor\`eme~VII.3.1.9]{berthelot:cohomologie_cristalline}\footnote{Strictly speaking, \cite[Th\'eor\`eme~VII.3.1.9]{berthelot:cohomologie_cristalline} only applies when~$\bar\dZ_{v,v}$ is the graph of an endomorphism of~$\bar\dX_v$, but the argument is easily generalised to the case of correspondences by taking covers. Also, \cite[Th\'eor\`eme~VII.3.1.9]{berthelot:cohomologie_cristalline} is a statement about crystalline cohomology rather than rigid cohomology, but these agree for smooth proper varieties \cite[Proposition~2]{berthelot:cohomologie_rigide_et_cohomologie_des_varieties}.} gives that
	\[
	\Delta\cdot\bar\dZ_{v,v} = \sum_{i=0,1,2}(-1)^i\tr(\bar\dZ_{v,v,*}|\rH^i_\rig(\bar\dX_v/K)) \,.
	\]
	The trace of~$\bar\dZ_{v,v,*}$ acting on~$\rH^i_\rig(\bar\dX_v/\bC_\ell)$ is~$d_2$ for~$i=0$ and is~$d_1$ for~$i=2$. Rearranging gives the desired identity for~$\tr_v(Z)$. The same argument applies for \'etale cohomology, using the corresponding trace formula \cite[Theorem~25.1]{milne:etale_cohomology}.
\end{proof}

To obtain certifiably correct outputs from our algorithms, it will be necessary to have some \emph{a priori} control on the possible actions of a correspondence~$Z$ on $\rH_1(\Gamma,\bZ)$ and the traces $\tr_v(Z)$ attached to vertices. This is what we establish here.

\begin{theorem}\label{thm:bounds}
	Let~$Z\subset X\times X$ be an effective correspondence, of degrees~$d_1$ and~$d_2$ over~$X$, respectively. Then:
	\begin{enumerate}[label = \alph*), ref = (\alph*)]
		\item\label{thmpart:bounds_graph_homology} $Z_*\colon\rH_1(\Gamma,\bZ)\to\rH_1(\Gamma,\bZ)$ has operator norm~$\leq\sqrt{d_1d_2}$ with respect to the intersection length pairing on~$\rH_1(\Gamma,\bZ)$.
		\item\label{thmpart:bounds_traces} For all~$v\in V(\Gamma)$, we have
		\[
		|\tr_v(Z)| \leq 2g(v)\cdot\max\{d_1,d_2\} \,,
		\]
		where~$g(v)$ is the genus of~$v$.
	\end{enumerate}
\end{theorem}

Now we prove \Cref{thm:bounds}, dealing with the two parts separately. For the first part, it suffices to deal separately with pushforward and pullback of homology classes.

\begin{lemma}\label{lem:harmonic_pushforward_bound}
	Suppose that~$f\colon\Gamma\to\Gamma'$ is a harmonic map of metrised graphs of degree~$d$. Then the pushforward map~$f_*\colon\rH_1(\Gamma,\bZ)\to\rH_1(\Gamma,\bZ)$ has operator norm at most~$\sqrt{d}$, i.e.\ we have
	\[
	\langle f_*(\gamma),f_*(\gamma)\rangle \leq d\cdot\langle \gamma,\gamma\rangle
	\]
	for all~$\gamma\in\rH_1(\Gamma,\bZ)$.
	\begin{proof}
		If we write
		$\gamma = \sum_e\lambda_e\cdot e$
		as usual, then we have
		\begin{align*}
			\langle f_*(\gamma),f_*(\gamma)\rangle &= \sum_{e'}\left(\sum_{f(e)=e'}\lambda_e\right)^{\!\!\!2}\cdot l(e') \\
			&\leq \sum_{e'}\left(\sum_{f(e)=e'}\frac1{l(e)}\right)\left(\sum_{f(e)=e'}\lambda_e^2l(e)\right)l(e') \\
			&= \sum_{e'}\left[\Bigl(\!\!\sum_{f(e)=e'}d_e(f)\Bigr)\Bigl(\sum_{f(e)=e'}\lambda_e^2l(e)\Bigr)\right] \\
			&= d\sum_e\lambda_e^2l(e) = d\langle\gamma,\gamma\rangle \,,
		\end{align*}
		using the Cauchy--Schwarz inequality in the second line.
	\end{proof}
\end{lemma}

\begin{lemma}\label{lem:harmonic_pullback_bound}
	Suppose that~$f\colon\Gamma\to\Gamma'$ is a harmonic map of metrised graphs of degree~$d$. Then we have
	\[
	\langle f^*(\gamma_1'),f^*(\gamma_2')\rangle = d\cdot\langle \gamma_1',\gamma_2'\rangle
	\]
	for all~$\gamma_1',\gamma_2'\in\rH_1(\Gamma',\bZ)$. In particular, $f^*$ has operator norm exactly~$\sqrt{d}$.
	\begin{proof}
		Let us write
		\[
		\gamma_1' = \sum_{e'}\lambda_{1,e'}'\cdot e' \quad\text{and}\quad \gamma_2' = \sum_{e'}\lambda_{2,e'}\cdot e' \,.
		\]
		According to the pullback formula
		\[
		f^*(\gamma_i') = \sum_e\lambda_{if(e)}d_e(f)\cdot e
		\]
		for~$i=1,2$, and hence
		\begin{align*}
			\langle f^*(\gamma_1'),f^*(\gamma_2')\rangle &= \sum_e\lambda_{1f(e)}\lambda_{2f(e)}d_e(f)^2l(e) \\
			&= \sum_{e'}\Bigl(\sum_{f(e)=e'} d_e(f)\Bigr)\cdot\lambda_{1e'}\lambda_{2e'}\cdot l(e') \\
			&= d\sum_{e'}\lambda_{1e'}\lambda_{2e'}\cdot l(e') = d\langle\gamma_1',\gamma_2'\rangle \,.
		\end{align*}
	\end{proof}
\end{lemma}

\begin{proof}[Proof of \Cref{thm:bounds}\ref{thmpart:bounds_graph_homology}]
	We assume that all irreducible components of~$Z$ are geometrically irreducible, which can always be achieved after a finite extension of the base field. If~$Z$ itself is geometrically irreducible, then~$Z_* = \pi_{2*}\pi_1^*$ is the composite of two maps which have operator norms at most $\sqrt{d_1}$ and~$\sqrt{d_2}$, respectively, by the above lemmas and \Cref{thm:harmonic_maps_on_skeleta}. So we are done in this case.

	To prove the general case, we need to show that if we know the result for two effective correspondences~$Z_1$ and~$Z_2$, then the result holds for $Z=Z_1+Z_2$. For this, let $(d_{11},d_{12})$ and $(d_{21},d_{22})$ be the degrees of~$Z_1$ and~$Z_2$ over~$X$, respectively, so that~$d_1=d_{11}+d_{21}$ and $d_2=d_{12}+d_{22}$. For any~$\gamma\in\rH_1(\Gamma,\bZ)$ we have
	\begin{align*}
		\langle Z_*(\gamma),Z_*(\gamma)\rangle &= \langle (Z_{1*}+Z_{2*})(\gamma),(Z_{1*}+Z_{2*})(\gamma)\rangle \\
		&\leq \left(d_{11}d_{12} + 2\sqrt{d_{11}d_{12}d_{21}d_{22}} + d_{21}d_{22}\right)\cdot\langle\gamma,\gamma\rangle \\
		&\leq (d_{11}+d_{21})(d_{12}+d_{22})\cdot\langle\gamma,\gamma\rangle \,,
	\end{align*}
	using the Cauchy--Schwarz inequality in the second line and the AM--GM inequality in the last line. This proves the result we want for~$Z$.
\end{proof}

Now we turn to the second part of \Cref{thm:bounds}. The main calculation is the following.

\begin{lemma}
	\label{lem:tracebound}
	Let~$\bar\dX$ be a smooth projective curve of genus~$g$ over~$k$, and~$\bar\dZ\subset\bar\dX\times\bar\dX$ an effective correspondence, of degrees~$d_1$ and~$d_2$ over~$\bar\dX$. Then $d_1 + d_2 - \bar\dZ \cdot \Delta$ has absolute value at most~$2g\max\{d_1,d_2\}$.
	\begin{proof}
		Note that by \Cref{lem:traceonH1} we have $\tr(\bar\dZ_*|\rH^1_\rig(\bar\dX/K)) = d_1 + d_2 - \bar\dZ \cdot \Delta$ for a correspondence $\bar\dZ$ on $\bar\dX$. The space $\rH^1_\rig(\bar\dX/K)$ is a $2g$-dimensional vector space over $K$.
		We may assume that~$g\geq1$, else there is nothing to prove.
		If we write $\bar\dZ=\bar\dZ_0+m\Delta$ where~$m\geq0$ and~$\bar\dZ_0$ is an effective divisor on~$\bar\dX\times \bar\dX$ not containing~$\Delta$, then we have $\bar\dZ_0\cdot\Delta\geq0$ by \cite[Proposition~V.1.4]{Hartshorne1977AlgebraicGeometry}, and so $\bar\dZ\cdot\Delta\geq m\Delta\cdot\Delta = 2m(1-g)\geq(d_1+d_2)(1-g)$. Combined with the above we obtain
		\[
		\tr(\bar\dZ_*|\rH^1_\rig(X/K)) \leq g(d_1 + d_2) \,.
		\]

		Now applying the same logic to the $n$th iterate~$\bar\dZ^n$ of~$\bar\dZ$, we find that~$\bar\dZ^n_*$ has integer trace on~$\rH^1_\rig(\bar\dX/K)$, which implies that all of the eigenvalues~$\lambda_1,\dots,\lambda_{2g}$ of~$\bar\dZ_*$ on~$\rH^1_\rig(\bar\dX/K)$ are algebraic integers. Moreover, we have
		\[
		\tr(\bar\dZ_*^n|\rH^1_\rig(\bar\dX/K)) \leq g(d_1^n + d_2^n)
		\]
		for all~$n\geq0$. So, if we view the eigenvalues~$\lambda_1,\dots,\lambda_n$ as elements of~$\bC$ via some complex embedding~$\bQ(\lambda_1,\dots,\lambda_{2g})\hookrightarrow\bC$, then we have
		\[
		\Re\left(\sum_{i=1}\lambda_i^n\right) \leq g(d_1^n+d_2^n)
		\]
		for all~$n\geq0$. Since~$(S^1)^{2g}$ is compact, we can choose~$n\gg0$ such that the argument of each~$\lambda_i^n$ is sufficiently close to~$0$ that~$\Re(\lambda_i^n)\geq\frac12|\lambda_i^n|$ for all~$i$. So for these~$n$ we have
		\[
		\sum_{i=1}^{2g}|\lambda_i^n| \leq 2g(d_1^n+d_2^n) \,;
		\]
		taking~$n$ sufficiently large with this property we obtain that~$|\lambda_i|\leq\max\{d_1,d_2\}$ for all~$i$. Hence
		\[
		|\tr(\bar\dZ_*|\rH^1_\rig(\bar\dX/K))| = |\sum_i\lambda_i| \leq 2g\max\{d_1,d_2\} \,.
		\]
	\end{proof}
\end{lemma}

\begin{proof}[Proof of \Cref{thm:bounds}\ref{thmpart:bounds_traces}]
	Again, we may and do assume $Z$ is geometrically irreducible. Letting~$\bar\dZ_{v,v} = \bigsqcup_{w \in S} \bar\dZ_w$ be as in \Cref{def:component_traces}, we obtain by \Cref{lem:tracebound} that
	\[
	|\tr_v(Z)|\leq g(v)\cdot\max\{\sum_{w\in S}d_w(\pi_1),\sum_{w\in S}d_w(\pi_2)\} \,.
	\]
	But we have
	\[
	\sum_{w\in S}d_w(\pi_i) \leq \deg(\pi_i) = d_i
	\]
	for~$i=1,2$ by \Cref{prop:morphisms_and_coverings}, and so we are done.
\end{proof}

\subsection{The local heights formula}\label{subsec:localHeightFormula}

We are now finally in a position to state the formula for local heights from \cite{BettsDogra}. To do so, recall that if~$\Gamma$ is a metrised graph (viewed as a metric space), then by a \emph{piecewise polynomial function} we mean a continuous function $f\colon\Gamma\to\bR$
which is given by a polynomial in the arc-length when restricted to any edge of~$\Gamma$. By a \emph{piecewise polynomial measure} we mean a formal sum
\[
\mu = \sum_eg_e\cdot|\rd s_e| + \sum_v\lambda_v\cdot\delta_v \,,
\]
where the first sum is taken over unoriented edges of~$\Gamma$ (i.e.\ connected components of~$\Gamma\setminus V(\Gamma)$), $g_e$ is a function $e\to\bR$ which is a polynomial in arc length, and~$\lambda_v\in\bR$. We think of~$|\rd s_e|$ as the unit length measure on edge~$e$ and~$\delta_v$ as a delta measure supported at the vertex~$v$. Such a measure has a \emph{total mass}
\[
\sum_{e \in E(\Gamma)} \int_{0}^{\ell(e)} g_e(s_e) |\rd s_e| + \sum_{v \in V(\Gamma)} \lambda_v.
\]

If~$f$ is a piecewise polynomial function and~$\vec{v}$ is a tangent direction at a vertex~$v\in V(\Gamma)$, then one can make sense of \emph{derivative $D_{\vec{v}}f(v)$ of~$f$ at~$v$ in direction~$\vec{v}$} \cite[Definition~3]{Baker2006MetrizedGraphs}. The \emph{Laplacian} of~$f$ is the piecewise polynomial measure $\nabla^2(f)$ defined by
\[
\nabla^2(f) \coloneqq -\sum_e(f|_e)''\cdot|\rd s_e| - \sum_v\left(\sum_{\vec{v}\in T_v(\Gamma)}D_{\vec{v}}f(v)\right)\cdot\delta_v \,,
\]
where $(f|_e)''$ denotes the second derivative of~$f$ with respect to arc length along~$e$ \cite[Definition~5]{Baker2006MetrizedGraphs}. The Laplacian defines an $\bR$-linear map from the space of piecewise polynomial functions to the space of piecewise polynomial measures; its kernel is the one-dimensional space of constant functions, and its image is the codimension-one space of measures of total mass~$0$.

Finding an explicit piecewise polynomial function~$f$ whose Laplacian is a given measure~$\mu$ of total mass~$0$ is not difficult. By formally double-integrating $g_e$ along each edge of~$\Gamma$ one finds a piecewise polynomial function~$f_0$ such that~$\mu-\nabla^2(f_0)$ is a sum of delta-measures supported at vertices of~$\Gamma$. As in \cite[\S5]{Baker2006MetrizedGraphs}, by finding a right inverse for the weighted Laplacian matrix \cite[Definition~7]{Baker2006MetrizedGraphs} one finds a piecewise affine function~$f_1$ such that~$\nabla^2(f_1)=\mu-\nabla^2(f_0)$, and then~$f=f_0+f_1$ is the desired Laplacian inverse of~$\mu$, well defined up to a constant function.

The local height formula of \cite{BettsDogra} gives a formula for the Laplacian of the normalised local height associated to a correspondence~$Z$ as an explicit piecewise polynomial measure on the reduction graph.

\begin{theorem}[{\cite[Corollary~12.1.3]{BettsDogra}\footnote{There are a couple of errors in the statement of \cite[Corollary~12.1.3]{BettsDogra} which we correct here.}}]\label{thm:local_heights_formula}
	Let~$K$ be a finite extension of~$\bQ_\ell$, $X/K$ a smooth projective curve with base point $b \in X(K)$ and~$Z\subset X\times X$ a correspondence representing a trace~$0$ endomorphism of~$\Jac(X)$ which is fixed by the Rosati involution. Let~$\Gamma$ be the reduction graph of a split semistable model of~$X$. Then the normalised local height function $\tilde h_{Z,\ell}\colon X(K) \to \bQ$	factors through a piecewise polynomial function
$\tilde h_{Z,\ell}\colon \Gamma \to \bR$
	whose Laplacian is given by
	\[
	\nabla^2(\tilde h_{Z,\ell}) = 2\sum_{e\in E(\Gamma)^+}\frac1{l(e)^2}\langle e,Z_*(\pi(e))\rangle\cdot|\rd s_e| + \sum_{v\in V(\Gamma)} \tr_v(Z)\cdot\delta_v \,, 
	\]
	where $\langle\cdot,\cdot\rangle$ denotes the intersection length pairing on~$\rC_1(\Gamma,\bZ)$ and~$\pi\colon\rC_1(\Gamma,\bR)\to\rH_1(\Gamma,\bR)$ is the orthogonal projection. Furthermore, $\tilde h_{Z,\ell}$ vanishes at the reduction of the base point $b$, and these properties determine~$\tilde h_{Z,\ell}$ completely.
\end{theorem}

\begin{remark}
	Since all of the edge lengths in~$\Gamma$ are rational, it follows that all of the quantities appearing in the above expression are rational numbers. Hence the function~$\tilde h_{Z,\ell}$ has the property that it takes points lying a rational distance along edges of~$\Gamma$ to rational numbers.
\end{remark}

\begin{remark}
We have refined the statement from \cite{BettsDogra} so that our result is purely combinatorial, rather than a formula in terms of \'etale cohomology. \Cref{thm:local_heights_formula} as stated above does not directly follow from~\cite{BettsDogra}.  The issue is the following. If we let~$\rH^1_\et(X_{\Kbar},\bQ_{\ell'})$ denote the \'etale cohomology group of~$X_{\Kbar}$, then this comes with a natural monodromy filtration, and the theorem of Picard--Lefschetz gives identifications of its graded pieces as
	\[
	\gr^{\rM}_i\rH^1_\et(X_{\Kbar},\bQ_{\ell'}) \cong
	\begin{cases}
		\rH_1(\Gamma,\bQ_{\ell'}) & \text{if $i=2$,} \\
		\bigoplus_{v\in V(\Gamma)}\rH^1_\et(\bar\dX_{v,\kbar},\bQ_{\ell'}) & \text{if $i=1$,} \\
		\rH^1(\Gamma,\bQ_{\ell'}) & \text{if $i=0$,} \\
		0 & \text{else.}
	\end{cases}
	\]
	By functoriality of \'etale cohomology, the correspondence~$Z$ induces a pushforward endomorphism~$Z_*$ of the group $\rH^1_\et(X_{\Kbar},\bQ_{\ell'})$ preserving the monodromy filtration, and this in turn induces endomorphisms of~$\rH_1(\Gamma,\bQ_{\ell'})$ and each~$\rH^1_\et(\bar\dX_{v,\kbar},\bQ_{\ell'})$. The statement proved in~\cite{BettsDogra} is valid for~$Z_*$ the endomorphisms induced this way via Picard--Lefschetz, rather than the endomorphisms we have defined graph-theoretically above. So to carefully deduce \Cref{thm:local_heights_formula} in the form stated, we need to prove that the Picard--Lefschetz isomorphism is compatible with pushforwards and pullbacks.
	
	At the suggestion of the referee, we give some indications of how to prove this, leaving full details to a future expository article. The theory developed in~\cite{BettsDogra} relies on a certain ``non-abelian Picard--Lefschetz isomorphism'' describing the whole monodromy-graded unipotent fundamental groupoid of~$X$ in terms of the graph~$\Gamma$; the usual Picard--Lefschetz isomorphism is then recovered by taking abelianisations. The construction of the non-abelian Picard--Lefschetz isomorphism in \cite{BettsDogra} is rather indirect, using Abhyankar's Lemma to pass to holomorphic families of curves over a punctured disc, and ultimately requiring many non-canonical choices. However, the theory in~\cite{BettsDogra} is completely agnostic to this particular construction, and the results would be equally valid for any other construction of the non-abelian Picard--Lefschetz isomorphism.
	
	Thus, the proof we have in mind is to give a new construction of the non-abelian Picard--Lefschetz isomorphism, using the Seifert--van Kampen Theorem with respect to a semistable covering of~$X$. This construction will involve no auxiliary choices, and will enable one to describe the abelian Picard--Lefschetz isomorphism very explicitly in terms of the covering. In fact, this description of the Picard--Lefschetz isomorphism should be highly analogous to the construction of the Coleman--Iovita isomorphism in the next section. Once this is done, an argument analogous to the proof of \Cref{prop:coleman-iovita_functoriality} should establish the desired functoriality.
\end{remark}


\section{The Coleman--Iovita isomorphism}
\label{sec:ColemanIovita}

In order to apply the local height formula, we need to be able to compute the action of a correspondence $Z\subset X\times X$ on the homology~$\rH_1(\Gamma,\bZ)$ of the dual graph~$\Gamma$, as well as the traces $\tr_v(Z)$ of~$Z$ at vertices~$v$. As outlined in the introduction, our strategy for doing this proceeds by first computing the action of~$Z_*$ on the de Rham cohomology of~$X$, and then using this to deduce the action on~$\rH_1(\Gamma,\bZ)$ via the Coleman--Iovita isomorphism. This isomorphism, whose statement we recall below, relates the de Rham cohomology of~$X$ to the homology of~$\Gamma$ and the rigid cohomology of the $k$-curves attached to vertices of~$\Gamma$.

\begin{theorem}[Coleman--Iovita isomorphism, {\cite[\S3.5]{ColemanIovitaHiddenStructures}, \cite[\S3.1]{DarmonRotger}}]\label{thm:coleman-iovita}
	Let~$K$ be~$\bC_\ell$ or a finite extension of~$\bQ_\ell$. Let~$X/K$ be a smooth projective curve, and let~$\dX/\cO_K$ be a split strongly\footnote{The same result is true without the word ``strongly'', but the explicit description of the Coleman--Iovita isomorphism is a little more complicated.} semistable model, with reduction graph~$\Gamma$. Then the de Rham cohomology~$\rH^1_\deR(X/K)$ carries a canonical \emph{monodromy filtration}
	\[
	0=\rM_{-1}\rH^1_\deR(X/K) \leq \rM_0\rH^1_\deR(X/K) \leq \rM_1\rH^1_\deR(X/K) \leq \rM_2\rH^1_\deR(X/K)=\rH^1_\deR(X/K)
	\]
	along with canonical identifications of the graded pieces
	\[
	\gr^{\rM}_n\rH^1_\deR(X/K) \cong
	\begin{cases}
		\rH_1(\Gamma,K) & \text{if $n=2$,} \\
		\bigoplus_v\rH^1_\rig(\bar\dX_v/K) & \text{if $n=1$,} \\
		\rH^1(\Gamma,K) & \text{if $n=0$.}
	\end{cases}
	\]
	Here, $\rH_1(\Gamma,K)$ and~$\rH^1(\Gamma,K)$ denote the homology and cohomology of the graph~$\Gamma$ with coefficients in~$K$ (see \S\ref{ss:homology}), $\rH^1_\rig(\bar\dX_v/K)$ denotes the rigid cohomology of the irreducible component~$\bar\dX_v$ of the special fibre~$\bar\dX$ attached to a vertex~$v$ of~$\Gamma$, and the direct sum is taken over all vertices of~$\Gamma$.
\end{theorem}

Since we will be using this isomorphism in computations, we will need to explicitly recall its construction, and prove some elementary (but fiddly) compatibility properties. For this, we first need an explicit description of the de Rham cohomology of wide opens.

\subsection{Cohomology of wide opens}

Let~$(W_v,W_{v,0})$ be a wide open pair, with bounding annuli $A_1,\dots,A_n$, oriented as per \Cref{rmk:bounding_annulus_orientation}. Choose a parameter~$t_i$ on each oriented annulus~$A_i$, giving an isomorphism $A_i\cong A(r_{i,1},r_{i,2})$. Let~$X_v$ be the rigid space obtained by gluing together~$W_v$ and the open discs $D(r_{i,2})$ of radius $r_{i,2}$ for~$1\leq i\leq n$, along the isomorphisms $t_i\colon A_i\cong A(r_{i,1},r_{i,2})$. Then~$X_v$ is the rigid analytification of a smooth projective curve over~$K$ \cite[Theorem~2.18]{McMurdyColeman}. We call~$X_v$ a \emph{compactification} of~$W_v$. (The space $W_v$ has many non-isomorphic compactifications, arising from different choices of the parameters~$t_i$.)

For each bounding annulus~$A_i$ of~$W_v$, its de Rham cohomology is isomorphic to~$K$ via the residue map
\[
\Res_{A_i}\colon\rH^1_\deR(A_i/K) \xrightarrow\sim K \,,
\]
see \cite[Lemma~2.13]{McMurdyColeman}. This map depends only on the orientation on~$A_i$, not on the particular parameter~$t_i$.

\begin{lemma}\label{lem:cohomology_of_wide_open_de_rham}
	Let~$(W_v,W_{v,0})$ be a wide open pair as above. Then for all $1$-forms $\omega\in\Omega^1(W_v)$, we have
	\[
	\sum_i\Res_{A_i}(\omega) = 0 \,.
	\]
	\begin{proof}
		The proof of \cite[Corollary~2.33]{McMurdyColeman} gives an exact sequence
		\begin{equation}\label{eq:wide_open_cohomology_de_rham}
		0 \to \rH^1_\deR(X_v/K) \to \rH^1_\deR(W_v/K) \to \bigoplus_iK \xrightarrow\Sigma K \to 0 \,,
		\end{equation}
		in which the middle arrow is the direct sum of the residue maps $\Res_{A_i}$. This implies the result.
	\end{proof}
\end{lemma}

The fact that the left-hand group~$\rH^1_\deR(X_v/K)$ is non-canonical, i.e.\ depends on the parameters~$t_i$, will cause us significant headaches when it comes to considering morphisms between curves in the next section. Thus, it is convenient for us to reinterpret this de Rham cohomology group as a rigid cohomology group in a manner which is independent of choices.

Recall that if~$\bar\dY$ is a smooth variety over~$k$, then the rigid cohomology of~$\bar\dY$ can be defined by choosing a smooth proper frame
\[
\bar\dY \hookrightarrow \bar\dX \hookrightarrow \dP \,,
\]
i.e.\ an open embedding of~$\bar\dY$ in a proper $k$-variety $\bar\dX$ and a closed embedding of~$\bar\dX$ in a formal $\cO_K$-scheme $\dP$ which is smooth in a neighbourhood of~$\bar\dY$ \cite[Definitions~3.1.5, 3.3.5, 3.3.10]{LeStumRigidCohomology}. The rigid cohomology\footnote{The notation from \cite{LeStumRigidCohomology} would be rather $\rH^\bullet_\rig(\bar\dY/\cO_K)$, but we prefer $\rH^\bullet_\rig(\bar\dY/K)$ to emphasise that it is a $K$-vector space.} $\rH^\bullet_\rig(\bar\dY/K)$ of~$\bar\dY$ is then defined by
\[
\rH^\bullet_\rig(\bar\dY/K) \coloneqq \rH^\bullet_\deR(\dP_K,j^\dagger\cO_{]\bar\dY[_\dP}) \,,
\]
i.e.\ it is the de Rham cohomology of the Raynaud generic fibre~$\dP_K$ with coefficients in the sheaf $j^\dagger\cO_{]\bar\dY[_\dP}$ of overconvergent sections of the structure sheaf of the tube of~$\bar\dY$ inside~$\dP$ \cite[Definition 8.2.5, Proposition 5.1.14]{LeStumRigidCohomology}. As the notation suggests, the rigid cohomology of~$\bar\dY$ is independent of the choice of smooth proper frame up to canonical isomorphism \cite[Propositions~7.4.2, 8.2.1]{LeStumRigidCohomology}. We will only be interested in the case that~$\bar\dY$ is a smooth curve over~$k$, $\bar\dX$ is its smooth completion, and $\dP$ is a deformation of~$\bar\dX$ to a smooth proper formal curve over~$\cO_K$.

\begin{proposition}\label{prop:cohomology_of_wide_open_rigid}
	Let~$(W_v,W_{v,0})$ be a strongly basic wide open pair, let~$\bar\dY_v$ be the smooth canonical reduction of $W_{v,0}$, with smooth completion $\bar\dX_v$. Choose parameters on the bounding annuli of~$(W_v,W_{v,0})$ giving rise to a compactification~$X_v$ of~$W_v$.
	
	Then there are canonical isomorphisms
	\[
	\rH^1_\rig(\bar\dX_v/K) \cong \rH^1_\deR(X_v/K) \quad\text{and}\quad \rH^1_\rig(\bar\dY_v/K) \cong \rH^1_\deR(W_v/K)
	\]
	which fit into a commuting square
	\begin{center}
	\begin{tikzcd}
		\rH^1_\deR(X_v/K) \arrow[r,hook]\arrow[d,equal,"\wr"] & \rH^1_\deR(W_v/K) \arrow[d,equal,"\wr"] \\
		\rH^1_\rig(\bar\dX_v/K) \arrow[r] & \rH^1_\rig(\bar\dY_v/K)
	\end{tikzcd}
	\end{center}
	in which the horizontal maps are induced by the evident inclusions.
\end{proposition}

\begin{proof}[Proof of \Cref{prop:cohomology_of_wide_open_rigid}]
	According to \cite[Theorem~2.27]{McMurdyColeman}, the compactification~$X_v$ has a proper formal model~$\dX_v$ whose special fibre is~$\bar\dX_v$. Such a model is necessarily smooth (since~$\bar\dX_v$ is smooth), so
	\[
	\bar\dX_v = \bar\dX_v \hookrightarrow \dX_v \quad\text{and}\quad \bar\dY_v \hookrightarrow \bar\dX_v \hookrightarrow \dX_v
	\]
	are smooth proper frames for~$\bar\dX_v$ and~$\bar\dY_v$, respectively.
	
	Calculating the rigid cohomology of~$\bar\dX_v$ with respect to this frame gives the desired identification
	\[
	\rH^1_\rig(\bar\dX_v/K) \cong \rH^1_\deR(X_v/K) \,,
	\]
	cf.\ \cite[Proposition~8.2.6(ii)]{LeStumRigidCohomology}.
	
	Calculating the rigid cohomology of~$\bar\dY_v$ with respect to the frame above gives an identification
	\[
	\rH^1_\rig(\bar\dY_v/K) \cong \varinjlim_{W'}\rH^1_\deR(W'/K) \,,
	\]
	where the colimit is taken over strict neighbourhoods~$W'$ of~$W_{v,0}$ inside~$X_v$ \cite[Proposition~5.1.12(ii)]{LeStumRigidCohomology}. For $\lambda<1$ in~$\sqrt{|K^\times|}$, let~$W^\lambda\subset X_v$ denote the complement of the closed discs $D[\lambda r_{1,i}^{-1}]$ inside the bounding discs $D(r_{1,i}^{-1})$. It follows from \cite[Proposition~3.3.2]{LeStumRigidCohomology} that the~$W^\lambda$ form a cofinal system of strict neighbourhoods of~$W_{v,0}$ inside~$X_v$.
	
	Moreover, for~$\lambda$ sufficiently close to~$1$, we have $W^\lambda\subseteq W_v$, and~$W_v$ is obtained from~$W^\lambda$ by gluing the annuli $A(r_{2,i}^{-1},r_{1,i}^{-1})$ along the subannuli $A(\lambda r_{1,i}^{-1},r_{1,i}^{-1})$. Since each inclusion $A(\lambda r_{1,i}^{-1},r_{1,i}^{-1})\hookrightarrow A(r_{2,i}^{-1},r_{1,i}^{-1})$ induces an isomorphism on de Rham cohomology, it follows by a Mayer--Vietoris argument that the inclusion $W^\lambda\hookrightarrow W_v$ also induces an isomorphism on de Rham cohomology. Hence the natural map
	\[
	\rH^1_\deR(W_v/K) \to \varinjlim_\lambda\rH^1_\deR(W^\lambda/K) = \varinjlim_{W'}\rH^1_\deR(W'/K)
	\]
	coming from the fact that~$W_v$ is a strict neighbourhood of~$W_0$ is an isomorphism, and so we have 
	\[
	\rH^1_\rig(\bar\dY_v/K) \cong \rH^1_\deR(W_v/K) \,.
	\]
	
	Finally, the compatibility between these identifications is \cite[Proposition~8.2.11]{LeStumRigidCohomology}. This requires a little extra explanation, since \cite[Proposition~8.2.11]{LeStumRigidCohomology} concerns the compatibility of base-change maps in relative rigid and de Rham cohomology, while we wish to know compatibility between the pullback maps $\rH^1_\deR(X_v/K)\to\rH^1_\deR(W_v/K)$ and $\rH^1_\rig(\bar\dX_v/K)\to\rH^1_\rig(\bar\dY_v/K)$ in de Rham and rigid cohomology. In fact, the pullback maps are just special cases of base-change maps. To see this for de Rham cohomology, consider a commuting square
	\begin{center}
	\begin{tikzcd}
		X' \arrow[r,"f"]\arrow[d,"\pi'"] & X \arrow[d,"\pi"] \\
		S' \arrow[r,"g"] & S
	\end{tikzcd}
	\end{center}
	of rigid analytic spaces, not necessarily Cartesian, where~$\pi$ and~$\pi'$ are smooth. Assume for simplicity that~$g$ is flat. In this context, we have relative de Rham complexes $\Omega^\bullet_{X/S}$ and $\Omega^\bullet_{X'/S'}$, which are complexes of $\pi^{-1}\cO_S$- and $\pi^{\prime-1}\cO_{S'}$-modules, respectively. There is a natural $f^{-1}\pi^{-1}\cO_S$-linear map $f^{-1}\Omega^\bullet_{X/S}\to\Omega^\bullet_{X'/S'}$, inducing a $\pi^{\prime-1}\cO_{S'}$-linear map $\pi^{\prime-1}\cO_{S'}\otimes_{f^{-1}\pi^{-1}\cO_S}f^{-1}\Omega^\bullet_{X/S}\to\Omega^1_{X'/S'}$. If we choose quasi-isomorphisms $\Omega^\bullet_{X/S}\to\mathcal{I}^\bullet$ and $\Omega^\bullet_{X'/S'}\to\mathcal{I}^{\prime\bullet}$ with $\mathcal{I}^\bullet$ and~$\mathcal{I}^{\prime\bullet}$ bounded below complexes of~$\pi^{-1}\cO_S$- and $\pi^{\prime-1}\cO_{S'}$-modules, respectively, then $f_*\mathcal{I}^{\prime\bullet}$ is a bounded below complex of injectives, and there is a dashed arrow~$\beta$ making the square
	\begin{center}
	\begin{tikzcd}
		\Omega^\bullet_{X/S} \arrow[r]\arrow[d] & f_*\Omega^\bullet_{X'/S'} \arrow[d] \\
		\mathcal{I}^\bullet \arrow[r,dashed,"\beta"] & f_*\mathcal{I}^{\prime\bullet}
	\end{tikzcd}
	\end{center}
	of morphisms of complexes of $\pi^{-1}\cO_S$-modules commute. Moreover, $\beta$ is unique up to homotopy. Taking the pushforward of~$\beta$ along~$\pi$ yields a morphism
	\[
	\pi_*(\beta)\colon \mathrm{R}\pi_*\Omega^\bullet_{X/S} \to \mathrm{R}g_*\mathrm{R}\pi'_*\Omega^\bullet_{X'/S'}
	\]
	in the derived category of~$\cO_S$-modules. Transposing along the push--pull adjunction for~$g$ gives a morphism
	\[
	g^*\mathrm{R}\pi_*\Omega^\bullet_{X/S} \to \mathrm{R}\pi'_*\Omega^\bullet_{X'/S'}
	\]
	in the derived category of~$\cO_{S'}$-modules. This final map is the base change map in relative de Rham cohomology. If we specialise to the case that~$S=S'=\Spec(K)$ and~$g$ is the identity, this definition of the base change map recovers the definition of the pullback map $f^*\colon \rH^\bullet_\deR(X'/K)\to\rH^\bullet_\deR(X/K)$ in de Rham cohomology.
	
	The pullback maps in rigid cohomology are also special cases of base-change maps (by definition, in the conventions of \cite{LeStumRigidCohomology}), and so the compatibility that we want is \cite[Proposition~8.2.11]{LeStumRigidCohomology} as claimed.
\end{proof}

\begin{remark}
	When we say that the isomorphisms in \Cref{prop:cohomology_of_wide_open_rigid} are canonical, we mean that the isomorphism $\rH^1_\rig(\bar\dX_v/K)\cong\rH^1_\deR(X_v/K)$ depends only on~$(W_v,W_{v,0})$ and the chosen parameters~$t_i$ used to define~$X_v$, while the isomorphism $\rH^1_\rig(\bar\dY_v/K)\cong\rH^1_\deR(W_v/K)$ depends only on~$(W_v,W_{v,0})$ and \emph{not} on the parameters~$t_i$. In particular, the inclusion $\rH^1_\rig(\bar\dX_v/K)\hookrightarrow\rH^1_\deR(W_v/K)$ does not depend on the choice of parameters, so we can rewrite the exact sequence \eqref{eq:wide_open_cohomology_de_rham} as
	\begin{equation}\label{eq:wide_open_cohomology_rigid}
		0 \to \rH^1_\rig(\bar\dX_v/K) \to \rH^1_\deR(W_v/K) \to \bigoplus_iK \xrightarrow\Sigma K \to 0 \,,
	\end{equation}
	which depends only on the strongly basic wide open~$(W_v,W_{v,0})$.
\end{remark}

\subsection{Description of the Coleman--Iovita isomorphism}
\label{subsec:descriptionCI}

Equipped with this description of the cohomology of wide opens, we are now in a position to explicitly describe the Coleman--Iovita isomorphism, following \cite[\S3]{DarmonRotger}.  See \cite{deShalit} for an approachable reference. Our description here will be purely rigid analytic. Let~$\fU=(W_v)_v$ be the strongly semistable covering of~$X^\an$ corresponding to a split strongly semistable model under \Cref{thm:semistable_models_and_coverings}. If~$[\omega]\in\rH^1_\deR(X/K)$ is a de Rham cohomology class, then for any oriented edge~$e$ of~$\Gamma$, we can restrict~$[\omega]$ to the annulus~$A_e$ to obtain a class in~$\rH^1_\deR(A_e/K)$, and then take its residue~$\Res_{A_e}(\omega)$. The residue of~$[\omega]$ along the inverse annulus $A_{e^{-1}}$ is $\Res_{A_{e^{-1}}}(\omega)=-\Res_{A_e}(\omega)$, so these residues determine a well-defined $1$-chain
\[
\phi_2(\omega)\coloneqq \sum_e\Res_{A_e}(\omega)\cdot e \in \rC_1(\Gamma,K) \,.
\]
By \Cref{lem:cohomology_of_wide_open_de_rham}, this $1$-chain is a $1$-cycle, hence a homology class. The assignment $\omega\mapsto\sum_e\Res_{A_e}(\omega)\cdot e$ thus defines a $K$-linear map
\[
\phi_2\colon\rH^1_\deR(X/K) \to \rH_1(\Gamma,K) \,.
\]

If~$[\omega]\in\rH^1_\deR(X/K)$ lies in the kernel of~$\phi_2$, then the restriction of~$[\omega]$ to the wide open~$W_v$ has residue~$0$ over all bounding annuli, so lies in~$\rH^1_\deR(X_v/K)=\rH^1_\rig(\bar\dX_v/K)$ using the exact sequence~\eqref{eq:wide_open_cohomology_de_rham} (see Theorem~\ref{thm:coleman-iovita} for notation). This construction defines a $K$-linear map
\[
\phi_1\colon\ker(\phi_2) \to \bigoplus_v\rH^1_\rig(\bar\dX_v/K) \,.
\]

Finally, if~$[\omega]\in\rH^1_\deR(X/K)$ lies in the kernel of~$\phi_2$ and~$\phi_1$, then~$[\omega]$ is represented by a meromorphic $1$-form~$\omega$ of the second kind on~$X^\an$ whose restriction to any wide open~$W_v$ is an exact form, so~$\rd f_v$ for some meromorphic function~$f_v$ on~$W_v$. If~$e$ is an oriented edge of~$\Gamma$, then the difference~$f_{\partial_0(e)}|_{A_e} - f_{\partial_1(e)}|_{A_e}$ is a constant function on~$A_e$. The assignment
\[
e \mapsto f_{\partial_0(e)}|_{A_e} - f_{\partial_1(e)}|_{A_e} \in K
\]
is thus a $1$-cocycle on~$\Gamma$, and different choices of the~$f_v$ yield $1$-cochains which differ by a $1$-coboundary. Thus the cohomology class of this cocycle is well-defined, and we have defined a $K$-linear map
\[
\phi_0\colon\ker(\phi_1) \to \rH^1(\Gamma,K) \,.
\]

\begin{definition}
	The \emph{monodromy filtration} on~$\rH^1_\deR(X/K)$ is defined by
	\[
	\rM_n\rH^1_\deR(X/K) \coloneqq
	\begin{cases}
		\rH^1_\deR(X/K) & \text{if~$n\geq2$,} \\
		\ker(\phi_2) & \text{if~$n=1$,} \\
		\ker(\phi_1) & \text{if~$n=0$,} \\
		0 & \text{if~$n<0$.}
	\end{cases}
	\]
\end{definition}

It will be convenient in what follows to adopt the shorthand that if~$\Gamma$ is a metrised complex of~$k$-curves, then
\[
\rH^1_\rig(\Gamma/K)
\]
denotes the graded $K$-vector space given in degrees~$0$, $1$ and~$2$ by
\[
\rH^1(\Gamma,K),\quad \bigoplus_v\rH^1_\rig(\bar\dX_v/K), \quad\text{and}\quad \rH_1(\Gamma,K) \,,
\]
respectively, and vanishing in all other degrees. When~$\Gamma$ is the reduction graph of a curve~$X/K$ with semistable reduction, the maps~$\phi_2$, $\phi_1$ and~$\phi_0$ thus induce a graded $K$-linear map
\begin{equation}\label{eq:coleman-iovita_map}
\phi_\bullet\colon \gr^\rM_\bullet\rH^1_\deR(X/K) \to \rH^1_\rig(\Gamma/K) \,.
\end{equation}

\begin{lemma}\label{lem:coleman-iovita_isomorphism}
	The map \eqref{eq:coleman-iovita_map} is the Coleman--Iovita isomorphism of \Cref{thm:coleman-iovita}.
\end{lemma}

\begin{proof}
	Even though this is essentially contained in \cite[\S3.1]{DarmonRotger}, we will recap the proof, since our statement in \Cref{thm:coleman-iovita} does not exactly match the corresponding statement in \cite{DarmonRotger}, and in any case, we will need to use parts of the proof again shortly. Let~$\fU=(W_v)_v$ be a strongly semistable covering of~$X^\an$, then the de Rham cohomology of~$X^\an$ can be computed as the \u Cech hypercohomology of the complex~$\Omega^\bullet_{X^\an/K}$ over the covering~$\fU$, i.e.\ as the cohomology groups of the \u Cech complex
	\[
	\check\rC^0(\fU,\Omega^\bullet_{X^\an/K}) \to \check\rC^1(\fU,\Omega^\bullet_{X^\an/K}) \to \check\rC^2(\fU,\Omega^\bullet_{X^\an/K}) \,.
	\]
	Here, for clarity, we are using the alternating \u Cech complex of \cite[01FG]{stacks-project}, i.e.
	\begin{align*}
		\check\rC^0(\fU,\Omega^\bullet_{X^\an/K}) &= \bigoplus_v\cO(W_v) \\
		\check\rC^1(\fU,\Omega^\bullet_{X^\an/K}) &= \bigoplus_v\Omega^1(W_v)\oplus\Bigl(\bigoplus_e\cO(A_e)\Bigr)^- \\
		\check\rC^2(\fU,\Omega^\bullet_{X^\an/K}) &= \Bigl(\bigoplus_e\Omega^1(A_e)\Bigr)^- \,,
	\end{align*}
	where the direct sums are taken over vertices~$v$ or oriented edges~$e$ of~$\Gamma$, and the superscript~$(\cdot)^-$ denotes the $-1$ eigenspace with respect to the involution given by switching the orientation on each edge. So, for example, elements of $\check\rC^1(\fU,\Omega^\bullet_{X^\an/K})$ are represented by tuples~$\bigl((\omega_v)_v,(f_e)_e\bigr)$ where~$\omega_v\in\Omega^1(W_v)$ for each vertex~$v$ and~$f_e\in\cO(A_e)$ for each oriented edge~$e$, with~$f_{e^{-1}}=-f_e$. The differentials on the \u Cech complex are given by
	\[
	\rd\bigl((f_v)_v\bigr) = \bigl((\rd f_v)_v,(f_{\partial_1(e)}|_{A_e}-f_{\partial_0(e)}|_{A_e})\bigr) \quad\text{and}\quad \rd\bigl((\omega_v)_v,(f_e)_e\bigr) = \bigl((\omega_{\partial_1(e)}|_{A_e}-\omega_{\partial_0(e)}|_{A_e}-\rd f_e)_e\bigr) \,.
	\]
	
	Now, the subspace~$\rM_0\rH^1_\deR(X/K)$ is the space of cohomology classes~$[\omega]$ represented by \u Cech cocycles $((\omega_v)_v,(f_e)_e)$ where each~$\omega_v$ is exact, or equivalently the space of cohomology classes which can be represented by \u Cech cocycles with~$\omega_v=0$ for all~$v$. For all such cohomology classes, $f_e$ is constant for all~$e$, and the map~$\phi_0$ sends the cohomology class~$[\omega]$ to the class of the map sending an edge~$e$ to~$f_e$. It is clear from this description that~$\phi_0$ is injective.
	
	Hence the map~\eqref{eq:coleman-iovita_map} is injective in all degrees. Since $\dim_K\rH^1_\deR(X/K)=\dim_K\rH^1_\rig(\Gamma/K)$ by \cite[Proposition~2.34]{McMurdyColeman}, it is an isomorphism.
\end{proof}

\begin{lemma}\label{lem:coleman-iovita_indept_of_covering}
	Suppose that~$K=\bC_\ell$. Then the group~$\rH^1_\rig(\Gamma/K)$, the monodromy filtration~$\rM_\bullet\rH^1_\deR(X/K)$ and the Coleman--Iovita isomorphism $\phi_\bullet\colon\gr^\rM_\bullet\rH^1_\deR(X/K)\xrightarrow\sim\rH^1_\rig(\Gamma/K)$ are all independent of the choice of strongly semistable covering.
\end{lemma}

\begin{remark}
	\Cref{lem:coleman-iovita_indept_of_covering} should also be true for~$K$ a finite extension of~$\bQ_\ell$.
\end{remark}

\begin{proof}
	By \Cref{rmk:skeleton_independent_of_vertex_set}, any two semistable vertex sets can be obtained from one another by repeatedly subdividing edges, budding off leaves, or performing the inverse of these operations. Thus it suffices to show invariance under both of these operations; we explain the argument for edge subdivision, leaving leaf budding to the reader.
	
	So suppose that~$V$ is a semistable vertex set with skeleton~$\Gamma$, and that~$e_{02}$ is an oriented edge of~$\Gamma$ with source~$v_0=\partial_0(e_{02})$ and target~$v_2=\partial_1(e_{02})$. Let~$v_1$ be a type \II point partway along~$e_{02}$. Then~$V'\coloneqq V\cup\{v_1\}$ is also a semistable vertex set, and its skeleton~$\Gamma'$ is obtained from~$\Gamma$ by splitting the edge~$e_{02}$ into two new edges~$e_{01}$ and~$e_{12}$ with $\partial_0(e_{01})=v_0$, $\partial_0(e_{12})=\partial_1(e_{01})=v_1$ and~$\partial_1(e_{12})=v_2$. See \Cref{fig:edge_subdivision}. In particular, $\Gamma$ and~$\Gamma'$ have the same underlying topological space, and so~$\rH^1(\Gamma,K)=\rH^1(\Gamma',K)$ and~$\rH_1(\Gamma,K)=\rH_1(\Gamma',K)$. Moreover, because~$v_1$ is a type \II point inside an annulus, the associated $k_v$-curve~$\bar\dX_{v_1}$ has genus~$0$, which implies that~$\rH^1_\rig(\bar\dX_{v_1}/K)=0$. So~$\rH^1_\rig(\Gamma/K)=\rH^1_\rig(\Gamma'/K)$.
	
	Next, let~$\phi_2'\colon\rH^1_\deR(X/K)\to\rH_1(\Gamma,K)$, $\phi_1'\colon\ker(\phi_2')\to\bigoplus_{v\in V'}\rH^1_\rig(\bar\dX_v/K)$ and~$\phi_0'\colon\ker(\phi_1')\to\rH^1(\Gamma,K)$ be the maps inducing the Coleman--Iovita isomorphism and the monodromy filtration for~$V'$. The annuli~$A_e$ in~$X^\an\smallsetminus V'$ are the same as those in~$X^\an\smallsetminus V$, except that the annulus~$A_{e_{02}}$ is removed and replaced with the two annuli~$A_{e_{01}}$ and~$A_{e_{12}}$. Geometrically, if $A_{e_{02}}$ is oriented-isomorphic to the standard annulus $A(r_0,r_2)$, then $A_{e_{01}}=A(r_0,r_1)$ and~$A_{e_{12}}=A(r_1,r_2)$ for some~$r_0<r_1<r_2$ in~$\sqrt{|K^\times|}$. So if~$[\omega]\in\rH^1_\deR(X/K)$ is a de Rham cohomology class, then we have $\Res_{A_{e_{02}}}(\omega)=\Res_{A_{e_{01}}}(\omega)=\Res_{A_{e_{12}}}(\omega)$, and hence
	\[
	\phi'_2(\omega)-\phi_2(\omega) = \Res_{A_{e_{02}}}(\omega)\cdot\bigl(e_{01}+e_{12}-e_{02}\bigr) \,.
	\]
	The right-hand side is a $1$-cycle in~$\Gamma'$ (it is the boundary of a $2$-simplex with vertices~$v_0$, $v_1$ and~$v_2$), and so we have~$\phi_2'=\phi_2$ as functions~$\rH^1_\deR(X/K)\to\rH_1(\Gamma,K)$.
	
	Next, the wide opens~$W_v'$ appearing in the semistable covering associated to~$V'$ are the same as those for~$V$, except for:
	\begin{itemize}
		\item $W_{v_0}'$ has the same underlying affinoid and the same bounding annuli as $W_{v_0}$, except for the bounding annulus corresponding to the edge~$e_{02}$, which is~$A_{e_{02}}$ in~$W_{v_0}$ and is~$A_{e_{01}}$ in~$W_{v_0}'$;
		\item $W_{v_2}'$ has the same underlying affinoid and the same bounding annuli as $W_{v_2}$, except for the bounding annulus corresponding to the edge~$e_{02}^{-1}$, which is~$A_{e_{02}^{-1}}$ in~$W_{v_2}$ and is~$A_{e_{12}^{-1}}$ in~$W_{v_2}'$; and
		\item $W_{v_1}'=A_{e_{02}}$ does not appear in the semistable covering associated to~$V$.
	\end{itemize}
	In particular, $W_{v_0}$ and~$W_{v_0}'$ have the same compactification~$X_{v_0}$, as do~$W_{v_2}$ and~$W_{v_2}'$. Thus, if~$[\omega]\in\ker(\phi_2)=\ker(\phi_2')$, then the $v_0$th and~$v_2$th components of~$\phi_1(\omega)$ and~$\phi_1'(\omega)$ agree, as do the $v$th components for all other~$v\in V$. Thus~$\phi_1=\phi_1'$ as functions~$\ker(\phi_2)\to\bigoplus_{v\in V}\rH^1_\rig(\bar\dX_v/K)$.
	
	Finally, suppose that~$[\omega]\in\ker(\phi_1)=\ker(\phi_1')$ is the class of some meromorphic $1$-form~$\omega$ of the second kind on~$X^\an$, and choose antiderivatives~$f_v$ of~$\omega|_{W_v}$ for each~$v\in V$. For the corresponding set of antiderivatives for~$V'$, we take~$f_v$ if~$v\notin\{v_0,v_1,v_2\}$, $f_{v_0}|_{W_{v_0}'}$ if~$v=v_0$, $f_{v_2}|_{W_{v_2}'}$ if~$v=v_2$, and choose the antiderivative~$f_{v_1}$ arbitrarily. If~$\xi$ and~$\xi'$ are the corresponding $1$-cocycles on~$\Gamma$ and~$\Gamma'$, then we have~$\xi(e)=\xi'(e)$ for all edges~$e\neq e_{02}^{\pm1}$, and~$\xi(e_{02})=\xi'(e_{01})+\xi'(e_{12})$. This implies that~$\xi$ and~$\xi'$ represent the same class in~$\rH^1(\Gamma,K)$, and so we have~$\phi_0'=\phi_0$ as functions~$\ker(\phi_1)\to\rH^1(\Gamma,K)$. This completes the proof.
\end{proof}

\begin{figure}[!h]
	\includegraphics[width=10cm]{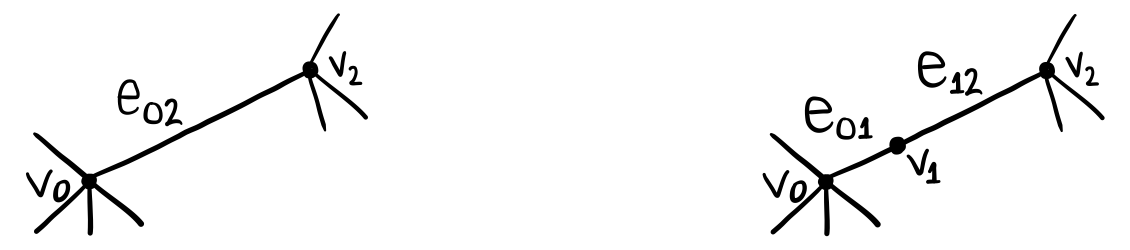}
	\caption{Subdividing the edge~$e_{02}$ in the proof of \Cref{lem:coleman-iovita_indept_of_covering}.}\label{fig:edge_subdivision}
\end{figure}

\subsection{Compatibility with Poincar\'e duality}\label{ss:poincare_duality}

In what follows, it will be important for us to know that the Coleman--Iovita isomorphism is compatible with all of the usual structures on cohomology: Poincar\'e duality and pullbacks and pushforwards along finite morphisms of curves. We begin by discussing compatibility with Poincar\'e duality. For this, if~$\Gamma$ is a metrised complex of~$k$-curves, we equip~$\rH^1_\rig(\Gamma/K)$ with an antisymmetric pairing by defining
\begin{equation}\label{eq:combinatorial_pairing_formula}
\langle\gamma_1+\sum_v[\omega_{1,v}]+\gamma_1^*,\gamma_2+\sum_v[\omega_{2,v}]+\gamma_2^*\rangle \coloneqq \gamma_1^*(\gamma_2) + \sum_v\langle[\omega_{1,v}],[\omega_{2,v}]\rangle - \gamma_2^*(\gamma_1)
\end{equation}
where~$\gamma_1$ and~$\gamma_2$ are homology classes on~$\Gamma$, $[\omega_{1,v}]$ and $[\omega_{2,v}]$ are rigid cohomology classes on~$\bar\dX_v$, and~$\gamma_1^*$ and~$\gamma_2^*$ are cohomology classes on~$\Gamma$. The pairing~$\langle[\omega_{1,v}],[\omega_{2,v}]\rangle$ on the right-hand side denotes the Poincar\'e pairing on rigid cohomology.

Compatibility of the Coleman--Iovita isomorphism with Poincar\'e duality then amounts to the following.

\begin{proposition}\label{prop:poincare_duality_compatibility}
	The subspaces $\rM_0\rH^1_\deR(X/K)$ and $\rM_1\rH^1_\deR(X/K)$ are exact annihilators for the Poincar\'e pairing on~$\rH^1_\deR(X/K)$. Moreover, the Coleman--Iovita isomorphism
	\[
	\phi_\bullet\colon \gr^\rM_\bullet\rH^1_\deR(X/K) \xrightarrow\sim \rH^1_\rig(\Gamma/K)
	\]
	is an isomorphism of $K$-vector spaces equipped with an antisymmetric pairing (the pairing on the left being induced by Poincar\'e duality, and the pairing on the right being~\eqref{eq:combinatorial_pairing_formula}).
\end{proposition}

We want to give a description of the Poincar\'e pairing on~$\rH^1_\deR(X/K)$ in terms of \u Cech hypercohomology, as in the proof of \Cref{lem:coleman-iovita_isomorphism}. For this, we first need to describe the trace map $\rH^2_\deR(X/K) \isomto K$. Any class $[\xi]\in\rH^2_\deR(X/K)$ can be represented by a \u Cech $2$-cocycle $(\omega_e)_e\in\check\rC^2(\fU,\Omega^\bullet_{X^\an/K})$, i.e.\ a tuple consisting of a $1$-form $\omega_e\in\Omega^1(A_e)$ for each oriented edge~$e$ of~$\Gamma$, subject to the constraint~$\omega_{e^{-1}}=-\omega_e$.

\begin{lemma}\label{lem:trace_cocycle_formula}
	The trace of the class~$[\xi]\in\rH^2_\deR(X/K)$ is given by
	\[
	\tr([\xi]) = \sum_{e\in E(\Gamma)^+}\Res_{A_e}(\omega_e) \,.
	\]
\end{lemma}

\begin{remark}\label{rmk:trace_normalisation}
	The trace map~$\tr\colon\rH^2_\deR(X/K)\to K$ on algebraic de Rham cohomology is determined, up to scaling by a global constant independent of~$X$, by the requirement that for a finite morphism $f\colon X\to X'$ of curves, we have~$\tr(f^*[\xi'])=\deg(f)\cdot\tr([\xi'])$ for all~$[\xi']\in\rH^2_\deR(X'/K)$. However, finding the correct choice of normalisation (i.e.\ global constant) in the literature is surprisingly difficult. (For example, the definition in \cite{HartshorneDeRham} involves composing several isomorphisms between one-dimensional vector spaces, and it is not clear, at least to the authors, how to normalise all of these isomorphisms, at least up to sign.)
	
	The normalisation we will take for the trace map in this paper is as follows. Let~$\fU$ be the usual affine covering of~$\bP^1_K$ by~$\bP^1_K\smallsetminus\{\infty\}$ and~$\bP^1_K\smallsetminus\{0\}$, and let~$\xi\in\check{\rC}^2(\fU,\Omega^1_{X/K})$ be the \u Cech $2$-cocycle given by $\frac{\rd t}t$ on~$\bG_m$. Then we normalise the trace map so that
	\[
	\tr([\xi]) = 1 \,.
	\]
	We remark that this is the choice of normalisation for which, over the complex numbers, if~$[\xi]$ is the class of a smooth $2$-form~$\eta$ on~$X(\bC)$, then $\tr([\xi]) = \frac1{2\pi i}\int_{X(\bC)}\eta$.
\end{remark}

\begin{proof}[Proof of \Cref{lem:trace_cocycle_formula}]
	We first observe that the quantity $\sum_e\Res_{A_e}(\omega_e)$ is independent of how we represent~$[\xi]$ as a \u Cech $2$-cocycle, and is also independent of the chosen strongly semistable covering by \Cref{lem:coleman-iovita_indept_of_covering}. So the assignment $[\xi]\mapsto\sum_e\Res_{A_e}(\omega_e)$ defines a $K$-linear map~$\psi\colon\rH^2_\deR(X/K)\to K$.
	
	In order to prove that~$\psi$ is equal to the trace map, it suffices to consider only the case that~$K=\bC_\ell$. Given a finite morphism~$f\colon X\to X'$, choose strongly semistable coverings~$\fU$ and~$\fU'$ of~$X^\an$ and~$X^{\prime,\an}$ as in \Cref{prop:morphisms_and_coverings}. If~$[\xi']\in\rH^2_\deR(X'/K)$ is represented by a \u Cech $2$-cocycle $(\omega_{e'})_{e'}$, then $f^*[\xi']$ is represented by the \u Cech $2$-cocycle $(\omega_e)_e$ with $\omega_e = (f|_{A_e})^*\omega'_{f(e)}$. So
	\[
	\Res_{A_e}(\omega_e) = d_e(f)\cdot\Res_{A'_{f(e)}}(\omega'_{f(e)}) \,,
	\]
	and so $\psi(f^*[\xi']) = \deg(f)\cdot\psi([\xi'])$. According to \Cref{rmk:trace_normalisation}, this implies that~$\psi$ is equal to the trace map, up to multiplication by some global constant.
	
	To check that this global constant is~$1$, let~$[\xi]\in\rH^2_\deR(\bP^1_K/K)$ be the cohomology class defined in \Cref{rmk:trace_normalisation}. With respect to the semistable covering of~$\bP^{1,\an}_K$ consisting of the open ball of radius~$\ell$ and the complement of the closed ball of radius~$1$ centred on~$0$, $[\xi]$ is represented by the \u Cech $2$-cocycle whose value on the intersection~$A(1,\ell)$ is~$\frac{\rd t}t$. Thus~$\psi([\xi])=1$ and we are done.
\end{proof}

As a consequence, we can give a formula for the Poincar\'e pairing on~$\rH^1_\deR(X/K)$.
\begin{corollary}\label{cor:poincare_duality_cocycle_formula}
	Let~$[\omega_1],[\omega_2]\in\rH^1_\deR(X/K)$ be represented by \u Cech $1$-cocycles~$\omega_1=((\omega_{1,v})_v,(f_{1,e})_e)$ and $\omega_2=((\omega_{2,v})_v,(f_{2,e})_e)$. Then the Poincar\'e pairing of~$[\omega_1]$ and~$[\omega_2]$ is given by
	\[
	\langle[\omega_1],[\omega_2]\rangle = \sum_{e\in E(\Gamma)^+}\Res_{A_e}(f_{1,e}\omega_{2,\partial_1(e)}-f_{2,e}\omega_{1,\partial_0(e)}) \,.
	\]
\end{corollary}

\begin{remark}\label{rmk:cup_product_formula}
	\Cref{cor:poincare_duality_cocycle_formula} contains as a special case the following. Suppose that~$W_v$ is a strongly basic wide open, and fix a compactification~$X_v$. A strongly semistable covering of~$X_v$ is then given by~$W_v$ and the discs glued onto its bounding annuli. The corresponding reduction graph is a star, whose centre is the vertex~$v$ corresponding to~$W_v$. If~$\omega_1,\omega_2\in\Omega^1(W_v)$ are differentials with residue~$0$ on each bounding annulus of~$W_v$, then they determine de Rham cohomology classes on~$X_v$ via the exact sequence~\eqref{eq:wide_open_cohomology_de_rham}. These cohomology classes~$[\omega_i]$ for~$i=1,2$ are represented by \u Cech $1$-cocycles~$((\omega_{i,v'})_{v'},(f_{i,e})_e)$ where~$\omega_{i,v'}=\omega_i$ for~$v'=v$ and~$\omega_{i,v'}=0$ otherwise. The edge functions~$f_e$, for edges oriented away from~$v$, satisfy $\rd f_{i,e}=-\omega_i|_{A_e}$, i.e.\ $-f_{i,e}=\int\omega_i|_{A_e}$ is an antiderivative of~$\omega_i|_{A_e}$. Thus \Cref{cor:poincare_duality_cocycle_formula} specialises to the formula
	\[
	\langle[\omega_1],[\omega_2]\rangle = - \sum_e\Res_{A_e}\left(\omega_1\cdot\int\omega_2\right) = \sum_e\Res_{A_e}\left(\Bigl(\int\omega_1\Bigr)\cdot\omega_2\right)
	\]
	where the sum is taken over the bounding annuli of~$W_v$ with their usual orientation, cf.\ \Cref{rmk:bounding_annulus_orientation}.
\end{remark}

\begin{remark}
	The case of \Cref{cor:poincare_duality_cocycle_formula} when~$[\omega_1]$ and~$[\omega_2]$ lie in~$\rM_1\rH^1_\deR(X/K)$ is stated as \cite[(109)]{DarmonRotger}, though is not proved. Our formula differs from that in \cite{DarmonRotger} by a sign; we presume that this is due to differing conventions for how to orient bounding annuli of wide opens, which is not made explicit in \cite{DarmonRotger}.
\end{remark}

\begin{proof}[Proof of \Cref{cor:poincare_duality_cocycle_formula}]
	Fix a choice of orientation on the graph~$\Gamma$ (i.e.\ make a choice of one of each pair $\{e,e^{-1}\}$). The cup product of the cocycles~$\omega_1$ and~$\omega_2$ is the \u Cech $2$-cocycle where
	\[
	(\omega_1\cup\omega_2)_e = f_{1,e}\omega_{2,\partial_1(e)}|_{A_e}-f_{2,e}\omega_{1,\partial_0(e)}|_{A_e}
	\]
	for each edge~$e$ with the chosen orientation\footnote{The need to choose an orientation here comes from the fact that this formula is the formula for the cup product of \emph{non-alternating} \u Cech cochains, cf. \cite[01FP]{stacks-project}. This formula induces the correct cup product on cohomology, independently of the choice of orientation.}. Hence we have
	\[
	\langle[\omega_1],[\omega_2]\rangle = \tr(\omega_1\cup\omega_2) = \sum_e\Res_{A_e}(f_{1,e}\omega_{2,\partial_1(e)}-f_{2,e}\omega_{1,\partial_0(e)})
	\]
	using \Cref{lem:trace_cocycle_formula}, and we are done.
\end{proof}

Now we are ready to prove the compatibility of the Coleman--Iovita isomorphism with Poincar\'e duality.

\begin{proof}[Proof of \Cref{prop:poincare_duality_compatibility}]
	If~$[\omega_1]$ lies in~$\rM_0\rH^1_\deR(X/K)$, then we may represent it by a \u Cech $1$-cocycle $((\omega_{1,v})_v,(f_{1,e})_e)$ in which $\omega_{1,v}=0$ for all~$v$, and hence each~$f_{1,e}$ is constant. It follows from \Cref{cor:poincare_duality_cocycle_formula} that
	\[
	\langle[\omega_1],[\omega_2]\rangle = \sum_e\Res_{A_e}(f_{1,e}\omega_{2,\partial_1(e)}) = \sum_ef_{1,e}\Res_{A_e}(\omega_{2,\partial_0(e)}) = \langle\phi_0([\omega_1]),\phi_2([\omega_2])\rangle
	\]
	for all~$[\omega_2]\in\rH^1_\deR(X/K)$. In particular, if~$[\omega_2]\in\rM_1\rH^1_\deR(X/K)$, then we have $\langle[\omega_1],[\omega_2]\rangle = 0$ and so the Poincar\'e pairing vanishes on~$\rM_0\otimes\rM_1$. Since~$\rM_0$ and~$\rM_1$ have complementary dimensions, they are exact annihilators.
	
	The same formula establishes that the associated graded of the Poincar\'e pairing agrees with the pairing on~$\rH^1_\rig(\Gamma/K)$ in degrees~$0$ and~$2$. It remains to deal with the pairing in degree~$1$. Suppose now that~$[\omega_1],[\omega_2]\in\rM_1\rH^1_\deR(X/K)$, so they are represented by \u Cech $1$-cocycles with $\Res_{A_e}(\omega_{i,\partial_0(e)})=0$ for all edges~$e$ and $i=1,2$. We can decompose~$\omega_1$ and~$\omega_2$ into components corresponding to the vertices of~$\Gamma$ as follows. For a vertex~$v_1$, the restriction of~$\omega_{1,v_1}$ to each bounding annulus~$A_e$ of~$W_{v_1}$ is an exact form. We can thus define a \u Cech $1$-cochain $((\omega_{1,v}^{v_1})_v,(f_{1,e}^{v_1})_e)$ where~$\omega_{1,v}^{v_1}=\omega_{1,v}$ if~$v=v_1$ and~$\omega_{1,v}^{v_1}=0$ otherwise. The functions~$f_{1,e}^{v_1}$ are chosen so that~$\rd f_{1,e}^{v_1}=-\omega_{1,v_1}|_{A_e}$ if~$\partial_0(e)=v_1$, and~$f_{1,e}^{v_1}=0$ for all edges~$e$ without~$v_1$ as an endpoint. This $1$-cochain is a $1$-cocycle by construction, and defines a cohomology class~$[\omega_1^{v_1}]\in\rM_1\rH^1_\deR(X/K)$. We have $[\omega_1] = \sum_{v_1\in V}[\omega_1^{v_1}]+[\omega_{10}]$ for some~$[\omega_{10}]\in\rM_0\rH^1_\deR(X/K)$, and also a similar decomposition $[\omega_2] = \sum_{v_2\in V}[\omega_2^{v_2}]+[\omega_{20}]$ for~$[\omega_2]$. Since the Poincar\'e pairing is bilinear and vanishes on~$\rM_0\otimes\rM_1$, it suffices to prove the result in the case~$[\omega_1]=[\omega_1^{v_1}]$ and~$[\omega_2]=[\omega_2^{v_2}]$ for some vertices~$v_1,v_2$, which may or may not be equal to one another. In other words, we are free to assume that~$\omega_{1,v}=0$ for all~$v\neq v_1$ and~$f_{1,e}=0$ for all edges~$e$ without~$v_1$ as an endpoint, and that the corresponding assertions hold for~$[\omega_2]$.
	
	Now in the Poincar\'e pairing formula
	\[
	\langle[\omega_1],[\omega_2]\rangle = \sum_e\Res_{A_e}(f_{1,e}\omega_{2,\partial_1(e)}-f_{2,e}\omega_{1,\partial_0(e)}) \,,
	\]
	the summands on the right-hand side vanish for edges~$e$ which do not have~$v_1$ and~$v_2$ as their endpoints. If~$v_1\neq v_2$ are opposite endpoints of~$e$, then the corresponding summand again vanishes (e.g.\ by orienting~$e$ from~$v_2$ to~$v_1$), and hence we see that $\langle[\omega_1],[\omega_2]\rangle=0$ for~$v_1\neq v_2$. If instead~$v_1=v_2$ is the source of edge~$e$ then we have
	\[
	\Res_{A_e}(f_{1,e}\omega_{2,\partial_1(e)}-f_{2,e}\omega_{1,\partial_0(e)}) = -\Res_{A_e}(f_{2,e}\omega_{1,v_1}) = \Res_{A_e}\Bigl(\omega_{1,v_1}\cdot\int\omega_{2,v_1}\Bigr) \,,
	\]
	and so we find
	\[
	\langle[\omega_1],[\omega_2]\rangle = \sum_{\partial_0(e)=v_1}\Res_{A_e}\Bigl(\omega_{1,v_1}\cdot\int\omega_{2,v_1}\Bigr) \,.
	\]
	But by \Cref{rmk:cup_product_formula}, the right-hand side is equal to the Poincar\'e pairing of the classes $[\omega_{1,v_1}],[\omega_{2,v_1}]\in\rH^1_\rig(\bar\dX_{v_1}/K)$ determined by~$\omega_{1,v_1}$ and~$\omega_{2,v_1}$. This shows that $\langle[\omega_1],[\omega_2]\rangle = \langle\phi_1([\omega_1]),\phi_1([\omega_2])\rangle$ for any $[\omega_1],[\omega_2]\in\rM_1\rH^1_\deR(X/K)$ and so we are done.
\end{proof}

\subsection{Compatibility with pullbacks and pushforwards}

The second property which we need to check -- and the more important one -- is that the Coleman--Iovita isomorphism is compatible with pullbacks and pushforwards along finite morphisms of curves, and hence that it is compatible with pushforwards along correspondences. If~$f\colon\Gamma\to\Gamma'$ is a finite harmonic morphism of metrised complexes of $k$-curves (\Cref{def:finitegraphmap}), then we have already defined in \S\ref{ss:harmonic_morphisms} pullback and pushforward maps
\[
f^*\colon \rH_1(\Gamma',\bZ) \to \rH_1(\Gamma,\bZ) \quad\text{and}\quad f_*\colon \rH_1(\Gamma,\bZ) \to \rH_1(\Gamma',\bZ)
\]
on homology, and similarly on cohomology. These maps are, by construction, adjoint to one another under the intersection length pairing. Moreover, since~$f$ is a morphism of metrised complexes of curves, it comes with a finite morphism $f_v\colon \bar\dX_v \to \bar\dX'_{f(v)}$ of curves over the residue field~$k$ for each vertex~$v\in V$. Hence there are also pullback and pushforward maps
\begin{align*}
	f^*=\bigoplus_v f_v^*\colon \bigoplus_{v'\in V'}\rH^1_\rig(\bar\dX'_{v'}/K) \to \bigoplus_{v\in V}\rH^1_\rig(\bar\dX_v/K) \,, \\ f_* = \bigoplus_v f_{v,*}\colon \bigoplus_{v\in V}\rH^1_\rig(\bar\dX_v/K) \to \bigoplus_{v'\in V'}\rH^1_\rig(\bar\dX'_{v'}/K) \,,
\end{align*}
which are adjoint under the Poincar\'e pairing (this is how the pushforward is defined). All in all, these pullback and pushforward maps define graded $K$-linear maps
\[
f^*\colon \rH^1_\rig(\Gamma'/K) \to \rH^1_\rig(\Gamma/K) \quad\text{and}\quad f_*\colon \rH^1_\rig(\Gamma/K) \to \rH^1_\rig(\Gamma'/K) \,,
\]
making~$\rH^1_\rig(\Gamma/K)$ both contra- and covariantly functorial in~$\Gamma$ with respect to finite harmonic morphisms. The maps~$f^*$ and~$f_*$ are adjoint with respect to the Poincar\'e pairing of~\eqref{eq:combinatorial_pairing_formula}.

We now carefully prove the following.
\begin{theorem}[Push--pull compatibility of Coleman--Iovita]\label{prop:coleman-iovita_functoriality}
	Let~$f\colon X\to X'$ be a finite morphism of smooth projective curves over~$\bC_\ell$, inducing a finite harmonic morphism $f\colon\Gamma\to\Gamma'$ between their reduction graphs as in \Cref{thm:harmonic_maps_on_skeleta}. Then the pullback and pushforward maps
	\[
	f^*\colon \rH^1_\deR(X'/K) \to \rH^1_\deR(X/K) \quad\text{and}\quad f_*\colon \rH^1_\deR(X/K) \to \rH^1_\deR(X'/K)
	\]
	on de Rham cohomology preserve the monodromy filtration (not necessarily strictly) and are compatible with the Coleman--Iovita isomorphism in the sense that both of the following squares commute.
	\begin{equation}\label{eq:push-pull_compatibility}
	\begin{tikzcd}
		\gr^\rM_\bullet\rH^1_\deR(X'/K) \arrow[r,"\phi_\bullet","\sim"']\arrow[d,"f^*"] & \rH^1_\rig(\Gamma'/K) \arrow[d,"f^*"] & \gr^\rM_\bullet\rH^1_\deR(X/K) \arrow[r,"\phi_\bullet","\sim"']\arrow[d,"f_*"] & \rH^1_\rig(\Gamma/K) \arrow[d,"f_*"] \\
		\gr^\rM_\bullet\rH^1_\deR(X/K) \arrow[r,"\phi_\bullet","\sim"'] & \rH^1_\rig(\Gamma/K) & \gr^\rM_\bullet\rH^1_\deR(X'/K) \arrow[r,"\phi_\bullet","\sim"'] & \rH^1_\rig(\Gamma'/K)
	\end{tikzcd}
	\end{equation}
\end{theorem}

For the proof, fix a strongly semistable vertex set~$V'\subset|X^{\prime\an}|_\II$ of size~$\geq2$ satisfying the conditions of \Cref{thm:harmonic_maps_on_skeleta}, so that~$V'$ and~$V=f^{-1}V'$ determine semistable coverings of~$X^{\prime\an}$ and~$X^\an$. Per \Cref{prop:morphisms_and_coverings}, for any bounding annulus~$A'_{e'}$ in~$X^{\prime\an}$, the inverse image of~$A'_{e'}$ in~$X^\an$ is
\[
f^{-1}A'_{e'} = \bigcup_{e\in f^{-1}(e')}A_e \,,
\]
and the restriction of~$f$ to~$A_e$ is a finite morphism $A_e\to A'_{e'}$ of degree~$d_e(f)$. According to \cite[Proposition~2.2]{Baker2013NonArchimedeanAnalyticCurves}, if $t'$ and $t$ are parameters on the annuli $A'_{e'}$ and $A_e$, then the restriction~$f|_{A_e}$ is given by $t'=\alpha t^{d_e(f)}(1+g(t))$ where~$g\in\cO(A_e)$ is an analytic function with norm~$<1$ at all points of~$A_e$. So
\[
f^*\frac{\rd t'}{t'}\Big|_{A_e} = \rd\log\left(\alpha t^{d_e(f)}(1+g(t))\right) = d_e(f)\frac{\rd t}t + \rd\log(1+g(t)) \,.
\]
Since~$\log(1+g(t))$ is convergent on the open annulus~$A_e$, it follows that
\[
\Res_{A_e}(f^*\frac{\rd t'}{t'}) = d_e(f) \,.
\]
Since the class of~$\frac{\rd t'}{t'}$ is a basis of~$\rH^1_\deR(A'_{e'}/K)$, it follows that
\[
\Res_{A_e}(f^*\omega') = d_e(f)\Res_{A'_{e'}}(\omega')
\]
for all $\omega'\in\Omega^1(A'_{e'})$. Fixing $[\omega']\in\rH^1_\deR(X'/K)$ and summing over edges of~$\Gamma$ gives
\begin{align*}
	\phi_2(f^*\omega') &= \sum_e\Res_{A_e}(f^*\omega')\cdot e \\
	 &= \sum_{e'}\sum_{e\in f^{-1}(e')}d_e(f)\Res_{A'_{e'}}(\omega')\cdot e = f^*\phi_2(\omega')
\end{align*}
using \Cref{prop:pullback_formula}. This tells us that~$f^*(\rM_1\rH^1_\deR(X'/K))\subseteq \rM_1\rH^1_\deR(X/K)$, and that the left-hand square in \eqref{eq:push-pull_compatibility} commutes in degree~$2$.

Next, for any $1$-form~$\omega'$ on~$A'_{e'}$ we have $(f|_{A_e})_*(f|_{A_e})^*\omega'=d_e(f)\omega'$ for all~$e\in f^{-1}(e')$ and so
\[
\Res_{A'_{e'}}((f|_{A_e})_*(f|_{A_e})^*\omega') = d_e(f)\Res_{A'_{e'}}(\omega') = \Res_{A_e}((f|_{A_e})^*\omega') \,.
\]
This implies that
$\Res_{A'_{e'}}((f|_{A_e})_*\omega) = \Res_{A_e}(\omega)$
for all $1$-forms~$\omega$ on~$A_e$, and so we have
\begin{align*}
	\phi_2(f_*\omega) &= \sum_{e'}\Res_{A'_{e'}}(f_*\omega)\cdot e' \\
	 &= \sum_{e'}\Res_{A'_{e'}}\left(\sum_{e\in f^{-1}(e')}(f|_{A_e})_*\omega\right)\cdot e' \\
	 &= \sum_{e'}\sum_{e \in f^{-1}(e')}\Res_{A_e}(\omega)\cdot e' = f_*\phi_2(\omega) \,.
\end{align*}
using \Cref{prop:pullback_formula}. So $f_*(\rM_1\rH^1_\deR(X/K)) \subseteq \rM_1\rH^1_\deR(X'/K)$ and the right-hand square in \eqref{eq:push-pull_compatibility} commutes in degree~$2$.

Next, for any~$[\omega']\in\rM_1\rH^1_\deR(X'/K)$ represented by a \u Cech $1$-cocycle $((\omega'_{v'})_{v'},(f'_{e'})_{e'})$ we know that the inverse image of each basic wide open~$W'_{v'}$ in~$X^{\prime\an}$ is
\[
f^{-1}W'_{v'} = \bigcup_{v\in f^{-1}(v')}W_v \,.
\]
For each~$v\in f^{-1}(v')$, the pullback $(f|_{W_v})^*\omega'_{v'}$ is a $1$-form on~$W_v$ with residue zero on each bounding annulus. According to \Cref{prop:morphisms_and_coverings} (combined with \cite[Proposition~8.2.10]{LeStumRigidCohomology}), the class of~$(f|_{W_v})^*\omega'_{v'}$ in rigid cohomology is the pullback of the class of~$\omega'_{v'}$ along the map $f_v\colon\bar\dX_v\to\bar\dX'_{v'}$. Taking this over all~$v$ implies that
\[
\phi_1(f^*[\omega']) = f^*\phi_1([\omega']) \,,
\]
which tells us that $f^*(\rM_0\rH^1_\deR(X'/K))\subseteq\rM_0\rH^1_\deR(X/K)$, and that the left-hand square in~\eqref{eq:push-pull_compatibility} commutes in degree~$1$.\smallskip

The remaining cases follow by Poincar\'e duality. For example, since pullback and pushforward are adjoint under the Poincar\'e pairing on~$\rH^1_\deR(X/K)$, it follows from the fact that~$f^*$ preserves~$\rM_1$ that $f_*$ preserves its annihilator~$\rM_0$. And then for any class~$[\omega]\in\rM_1\rH^1_\deR(X/K)$ we have
\begin{align*}
	\langle\phi_1(f_*[\omega]),\phi_1([\omega'])\rangle &= \langle f_*[\omega],[\omega']\rangle = \langle[\omega],f^*[\omega']\rangle = \langle\phi_1([\omega]),\phi_1(f^*[\omega'])\rangle \\
	 &= \langle\phi_1([\omega]),f^*(\phi_1([\omega']))\rangle = \langle f_*(\phi_1([\omega])),\phi_1([\omega'])\rangle \,.
\end{align*}
Since the Poincar\'e pairing is perfect, this implies that $\phi_1(f_*[\omega])=f_*(\phi_1([\omega]))$, and hence the right-hand square in \eqref{eq:push-pull_compatibility} commutes in degree~$1$.

A similar argument shows that the squares in~\eqref{eq:push-pull_compatibility} commute in degree~$0$, completing the proof of \Cref{prop:coleman-iovita_functoriality}. \qed

\subsubsection{Compatibility with correspondences}

If~$Z\subset X\times X'$ is a correspondence from~$X$ to~$X'$, having reduction graphs~$\Gamma$ and~$\Gamma'$, respectively, then we can define a pushforward
\[
Z_*\colon\rH^1_\rig(\Gamma/K) \to \rH^1_\rig(\Gamma'/K)
\]
in the usual way, namely take the composite $\tilde\pi_{2,*}\circ\tilde\pi_1^*$ where~$\tilde\pi_1\colon\tilde Z\to X$ and~$\tilde\pi_2\colon\tilde Z\to X'$ are the projections from the normalisation. It is a formal consequence of \Cref{prop:coleman-iovita_functoriality} that the Coleman--Iovita isomorphism is also compatible with pushforward along correspondences.

In our computations, we will use this to compute the action of~$Z_*$ on~$\rH_1(\Gamma,\bZ)$, as well as the vertex traces~$\tr_v(Z)$. Once we know the action of~$Z_*$ on~$\rH^1_\deR(X/K)$, we can use the Coleman--Iovita isomorphism to determine the action on~$\rH^1_\rig(\Gamma/K)$. In degree~$2$, we read off the action of~$Z_*$ on~$\rH_1(\Gamma,K)$, while in degree~$1$ we read off the action on~$\bigoplus_v\rH^1_\rig(\bar\dX_v/K)$. Taking an appropriate block in the matrix representing this action and taking the trace, we obtain the vertex traces~$\tr_v(Z)$ by Lemma~\ref{lem:traceonH1}.

\section{Explicit Coleman--Iovita for hyperelliptic curves}\label{sec:ExplicitColemanIovitaHyp}

We now specialise to the case of hyperelliptic curves in odd residue characteristic, and begin to explain what the Coleman--Iovita isomorphism means explicitly. In this section we first explain background on cluster pictures. We then give a description of how to obtain a strongly semistable cover of a hyperelliptic curve over a finite extension $K/\bQ_\ell$ in odd residue characteristic recovering the results of \cite{StollUniformBounds,KatzKaya} and reinterpreting them in the language of cluster pictures. We provide a definition of the bases of $\rH_1(\Gamma,\bZ)$, $\rH^1_\deR(X/K)$, and $\rH^1_\rig(\dX_v/K)$ that we will use and give an explicit description of the Coleman--Iovita isomorphism with respect to these bases. We also explicitly describe the Berkovich skeleton of $X^{\an}_{\bC_\ell}$ using the semistable covering of $X$. We use this description to give a sufficient condition for when a component of the associated semistable model $\calX$ has no $\bQ_\ell$-points reducing to it.

\subsection{A semistable covering}
Let $\ell>2$ be a prime. Let $\pi:X\to\P^1$ be a nice hyperelliptic curve defined over a finite extension $K/\bQ_\ell$.  In the standard affine patch, suppose $X$ is given by an equation $y^2 = f(x)$ and $\deg f \geq 3$, with $f$ separable.  Enlarging $K$ if necessary, we may assume that $K$ contains the biquadratic extension of the field generated by $\calR$, the set of roots of $f$ in $\bar{K}$.  In this section, we use cluster theory \cite{DDMM,UsersGuide} to write down a semistable covering of $X^\an$ without explicitly finding a semistable model of $X/K$.

Following \cite{DDMM} we define the following terms. See Example~\ref{ex:clusterpicture} for an example of a cluster picture and corresponding Berkovich space and Berkovich skeleton in genus $2$.
\begin{definition}[Cluster vocabulary]
\hfill
\begin{itemize}
\item A \emph{cluster} $\ss$ is a non-empty subset of the roots of $f$ cut out by a disc in $\A^{1, \an}$. That is, $\ss = D \cap \calR$ for some open or closed disc $D \subset \A^{1, \an}$.
\item A cluster $\ss$ is \emph{proper} if $\# \ss>1$. It is \emph{odd} if $\#\ss$ is odd and \emph{even} if $\#\ss$ is even.
\item The \emph{top} cluster is $\calR$ itself.
\item A cluster $\ss'$ is a \emph{child} of $\ss$ (and $\ss$ is the \emph{parent} of $\ss'$), denoted by $\ss' < \ss$, if $\ss'$ is a maximal subcluster of $\ss$ (not equal to $\ss$ itself). We write $\ss'' \leq \ss$ to denote that $\ss'$ is a child or equal to $\ss$. For any non-top cluster we denote by $P(\ss)$ the parent of $\ss$. 
\item An even cluster $\ss$ is called \emph{\"ubereven} if every child of $\ss$ is also even.
\item The \emph{depth} of a proper cluster $\ss$ is 
\[d_\ss \colonequals  \min_{\alpha_1, \alpha_2 \in \ss}  v(\alpha_1 - \alpha_2).\] If $\ss$ is not the top cluster, then the \emph{relative depth} of $\ss$ is $\delta_\ss\colonequals d_\ss - d_{P(\ss)}$.
\item Let $\ss \wedge \ss'$\ be the smallest cluster containing both $\ss$ and $\ss'$.
\item Define the invariant
\[\nu_\ss \colonequals  v(c) + \sum_{r \in \mathcal{R}} d_{\{ r\} \wedge \ss} \] for each proper cluster $\ss$, where $c$ is the leading coefficient of $f$.
\item A \emph{cluster picture} is the information $(\Sigma, \mathcal{R}, d)$,  where $\Sigma$ is the set of all clusters and $d$ is a function that assigns to each proper cluster $\ss$ its depth $d_\ss$. 
We draw a cluster picture by drawing the roots $r \in \mathcal{R}$ with dots and grouping the roots with  ovals to represent clusters of size $> 1$, for example in the figure below. 
\begin{center}
\includegraphics[width=100pt]{figures/cluster022_1_2_0.png}
\end{center}
The subscripts on the largest cluster is its depth; on all smaller clusters it is the relative depth.

\end{itemize}
\end{definition}

\begin{remark}
Our depths differ from those in \cite{DDMM} by a multiplicative constant, because they normalise their valuation to have valuation group $\Z$, while we normalise $v(\ell) = 1$.
\end{remark}

\begin{definition}[Open discs and annuli associated to clusters]\label{def:UsandAs}
 For a cluster $\ss$, we define the following subsets of $\P^{1, \an}_K$.
\hfill
\begin{itemize}
\item Let $D_\ss^c$ be the smallest closed disc in $\A^{1, \an}_K$ such that $\ss = D^c_\ss \cap \calR$.

\item For $\ss$ not the top cluster, let $D^o_\ss$ be the largest open disc in $\A^{1, \an}_K$ such that $\ss = D^o_\ss \cap \calR$. If $\ss$ is the top cluster we set $D_\ss^o = \P^{1, \an}_K$.

\item For $\ss$ not the top cluster, define the open annulus $A_\ss \colonequals D_\ss^o \setminus D_\ss^c$.

\item Define
\[ U_\ss \colonequals D_\ss^o \setminus  \bigcup_{\substack{\ss' < \ss \\ \ss' \text{ proper}}}  D_{\ss'}^c  \quad \subseteq \P^{1, \an}_K.\]
\end{itemize}

\begin{figure}[h!]
\includegraphics[width = 6cm]{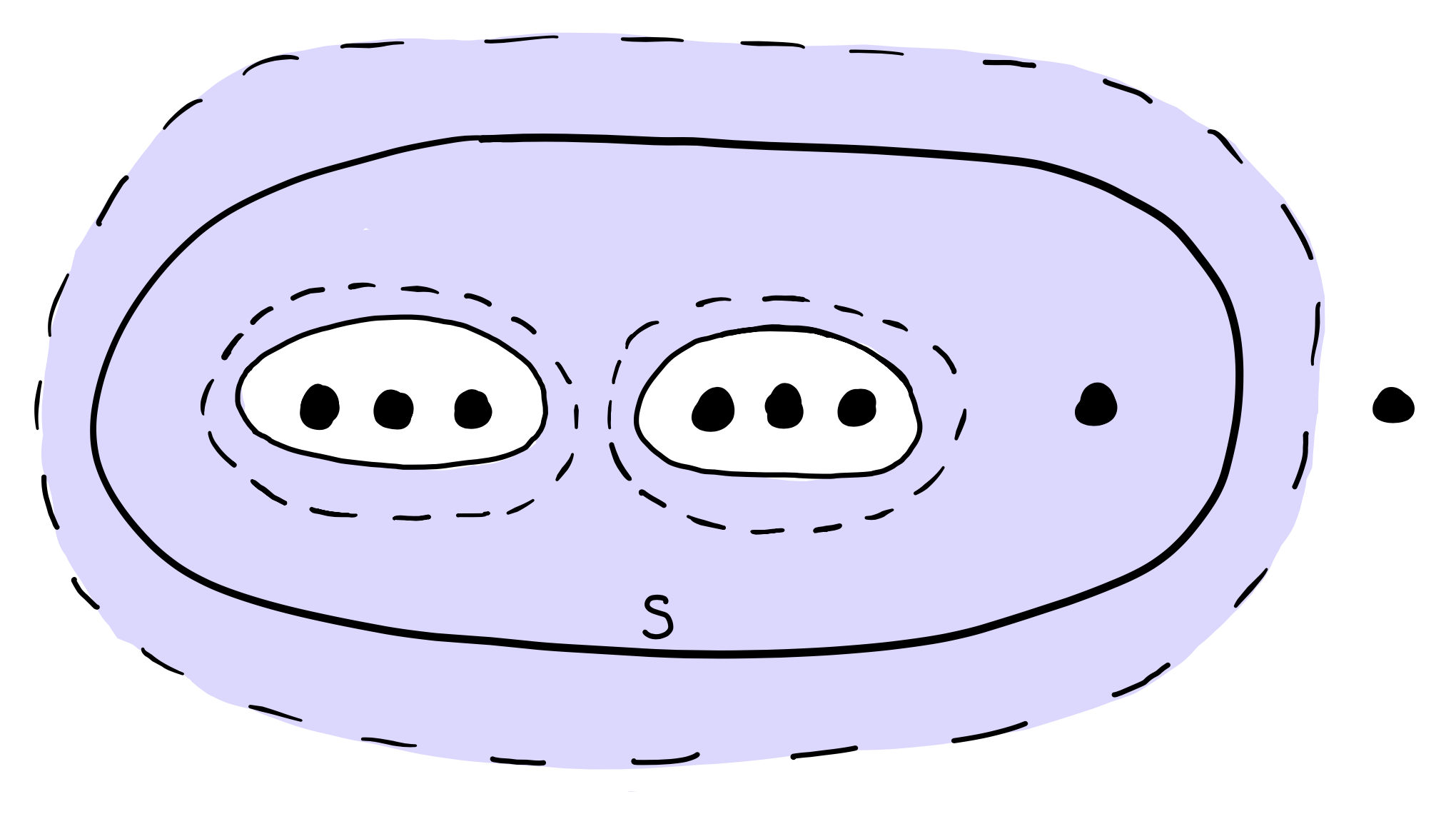}
\caption{A picture of the analytic space $U_\ss$.}
\end{figure}

\end{definition}

\begin{proposition}
The analytic spaces $\{U_\ss\}_{\ss\in \Sigma}$ constitute a semistable covering of $\P^{1, \an}_K$.
\end{proposition}

\begin{proof}
Given $x \in \A^{1, \an}_K(K)$, let $D^c$ be the smallest closed disc containing $x$ and at least two roots of $f$. Let $\ss \colonequals D^c \cap \calR$, which is a proper cluster. Then $x \in D^o_\ss$. Indeed, since $D^c \cap D^o_\ss \neq \emptyset$, we have either $D^c \subset D_\ss^o$ and so $x \in D_\ss^o$, or $D_\ss^o \subset D^c$. In the latter case, there is a 
larger open disc $D_\ss^o \subsetneq D^o \subset D^c$ with $D^o \cap \calR = \ss$ contradicting the maximality of $D^o_\ss$. By minimality of $D^c$, for any proper child $\ss' < \ss$, we have $x \notin D_{\ss'}^c$. So $x \in U_\ss$ and hence the $U_\ss$ cover $\A^{1, \an}_K$.

To show the covering is semistable, we first note that for a cluster $U_\ss$, the pair \[\left(U_{\ss}, D_\ss^c \setminus  \bigcup_{\substack{\ss' < \ss \\ \ss' \text{ proper}}}  D_{\ss'}^o\right)\] forms a basic wide open. Secondly, the intersection $U_\ss \cap U_{\ss'}$ is by construction the annulus $A_{\ss'}$ if $\ss' < \ss$, and is empty if $\ss \not\leq \ss'$ and $\ss' \not\leq \ss$. Hence all double intersections are annuli or empty, and all triple intersections are empty.
\end{proof}

This semistable covering also appears in  \cite{StollUniformBounds}, where it is used to give uniform bounds for rational points on hyperelliptic curves of small Mordell--Weil rank and \cite[Section~4]{KatzKaya} where it plays a role in defining a $p$-adic integral (where $p = \ell$) on bad reduction hyperelliptic curves. Here, we tie it to cluster pictures, as speculated about in \cite[Chapter 4]{EnisThesis}.

We will repeatedly make use of the following lemma to describe equations for $X$ restricted to annuli and wide opens of the semistable covering.
\begin{lemma}[{\cite[Lemma 4.10]{KatzKaya}}]
\label{lem:inversesqroot}
Let $D^o$ be an open disc in $\A^{1, \an}_K$ and let $D_1^c, \dots, D_n^c$ be pairwise disjoint closed subdiscs. Set $U \colonequals D^o \setminus \cup_i D_i^c$. Suppose that $h(x) \in K[x]$ is a monic polynomial such that
\begin{enumerate}[(a)]
\item $h(x)$ has no zeros in $U$, and
\item $h(x)$ has an even number of zeros in each $D_i^c$.
\end{enumerate}
Then there is an invertible rigid analytic function $h(x)^{1/2} \in \calO(U_{K})^{\an}$ whose square is $h(x)$. Moreover $h(x)^{1/2}$ is explicitly computable.
\end{lemma}

\begin{proof}

 Write $h(x) = C \cdot \prod_{D_k^c} (\prod_{i=1}^{2s_k} (x - \alpha_i)) \cdot \prod_{j=1}^t (x - \beta_j)$ where $\alpha_i \in D_k^c$,  $\beta_j\notin D^o$, and $C \in K$. Then 
\[  \left( \prod_{i=1}^{2s_k} (x - \alpha_i) \right)^{1/2} = (x - \alpha_1)^s \prod_{i=1}^{2s_k} \left(1 - \frac{\alpha_i - \alpha_1}{x - \alpha_1}\right)^{1/2}\] converges outside of $D_k^c$ for all $1 \leq k \leq n$.
Furthermore, 
\[\left( C\cdot \prod_{j=1}^t (x - \beta_j)  \right)^{1/2} = C^{1/2}\cdot  \prod_{i=1}^t ((x-\beta_1) - (\beta_i - \beta_1))^{1/2} = C^{1/2}\cdot\prod_{i=1}^t \left((\beta_1 - \beta_i)^{1/2} \left(1- \frac{x-\beta_1}{\beta_1-\beta_i}\right)^{1/2}\right)\] converges on $D^o$.
The product of the square roots is $h(x)^{1/2}$ and belongs to $\calO(U_{K'})^{\an}$.
\end{proof}

\begin{definition}
We fix an ordering on $\calR$. For every odd cluster $\ss$ (including singletons), choose the smallest element $\alpha_\ss \in \ss$.
For a non-singleton cluster $\ss$, we define
\[ g_\ss(x) \colonequals \prod_{\substack{\ss' < \ss\\ \ss' \text{ odd}}}(x-\alpha_{\ss'})  \quad \text{and} \quad h_\ss(x) \colonequals f(x)/g_\ss(x). \]
For a singleton cluster $\ss$ we define $ g_\ss(x) \colonequals x-\alpha_{\ss}$ and $h_\ss(x) \colonequals f(x)/g_\ss(x).$

Let $X_\ss$ be the projective hyperelliptic curve defined by the affine equation $u_\ss^2 = g_\ss(x)$.
\end{definition}

\begin{remark}
If $\ss$ is \"ubereven then $g_\ss = 1$ and so $X_\ss \simeq \P_K^{1, \an} \sqcup \P_K^{1, \an}$ is a disconnected hyperelliptic curve. This will not be a problem.
\end{remark}

We can now describe equations for $X$ restricted to subsets of the semistable covering and annuli.
\begin{proposition}
For any cluster $\ss$, we have that $X|_{U_\ss}$ and $X_\ss|_{U_\ss}$ are isomorphic over $U_\ss$ as rigid analytic spaces.
\end{proposition}

\begin{proof}
By construction, the polynomial $h_\ss(x)$ has no roots in $U_\ss$ and an even number of roots in $D_{\ss'}^c$ for each proper child $\ss'<\ss$. So it has a square root $h_\ss(x)^{1/2}$ by \Cref{lem:inversesqroot}. The map
\begin{equation}
\begin{split}
\label{eqn:XtoUs}
  X_\ss|_{U_\ss} &\to X|_{U_\ss}\\
  (x,u_s) & \mapsto (x, u_s \cdot h_\ss(x)^{1/2})
  \end{split}
  \end{equation}
  is the desired isomorphism.
\end{proof}

\begin{lemma} \label{lem:step1}
If $\ss$ is a proper, non-top cluster, then $X|_{A_\ss}$ is isomorphic to
\begin{enumerate}
\item a disjoint union of two open annuli of width $\delta_\ss$, each mapping isomorphically to $A_\ss$ if $\ss$ is even;
\item an open annulus of width $\delta_\ss/2$ mapping to $A_\ss$ via the squaring map if $\ss$ is odd.
\end{enumerate}
\end{lemma}

\begin{proof}
If $\ss$ is an even cluster, then $g_\ss(x)^{1/2} h_\ss(x)^{1/2}$ is a square root of $f(x)$ on the annulus $A_\ss$, and so gives a splitting 
\begin{equation}
\begin{split}
\label{eqn:Xannuluseven}
 A_\ss \times \{\pm1\} & \simeq X|_{A_\ss} \\
(x,\varepsilon) &\mapsto (x, \varepsilon \cdot g_\ss(x)^{1/2} h_\ss(x)^{1/2})
\end{split} 
\end{equation}of $X$ over $A_\ss$. 
Otherwise, if $\ss$ is odd, then $\left(\frac{g_\ss(x)}{x-\alpha_s}\right)^{1/2} h_\ss(x)^{1/2}$ is a square root of $f(x)/(x- \alpha_\ss)$ on $A_\ss$ and so, gives an isomorphism
\begin{align}
\label{eqn:Xannulusodd}
(\text{curve given by } t_\ss^2 = (x- \alpha_\ss))|_{A_\ss}\simeq X|_{A_\ss}.
\end{align}
The curve $(t_{\ss}^2 = x- \alpha_\ss)|_{A_\ss}$ is isomorphic to an annulus of width $\delta_\ss/2$ via projection onto the $t_{\ss}$-coordinate.
\end{proof}

\begin{remark}
Lemma~\ref{lem:step1} has appeared as \cite[Lemma 4.12]{KatzKaya}, though not in the context of cluster pictures, and without implementation. By using the language of cluster pictures, we make the procedure more amenable to implementation, and prove several other statements about the resulting coverings of $X^\an$, such as Theorem~\ref{the:minimalSkeleton} and Theorem~\ref{thm:reductiontocomponents}.
\end{remark}

Since the isomorphisms above depend on a choice of square root of $h_\ss$ and $g_\ss$ we fix these choices once and for all, as per the following definition.
\begin{definition}[Choice of square roots]
We fix the following data:
\begin{itemize}
\item for every cluster $\ss$, a square root $h_\ss(x)^{1/2}$ of $h_\ss$ on $U_{\ss}$.
\item for every even cluster $\ss$ (not the top cluster), a square root $g_\ss(x)^{1/2}$ on the annulus $A_\ss$;
\item for every odd cluster $\ss$ (not the top cluster), a square root $\left(\frac{g_\ss(x)}{(x-\alpha_\ss)} \right)^{1/2}$ on the annulus $A_\ss$.
\end{itemize}
\end{definition}

Our choices of $g_\ss(x)^{1/2}$ and $h_\ss(x)^{1/2}$ give trivialisations 
\begin{align*}
X|_{U_\ss} &\simeq U_\ss \times \{\pm1\} \quad \text{ for $\ss$ \"ubereven};\\
X|_{A_\ss} &\simeq A_\ss \times \{\pm1\} \quad \text{ for $\ss$ even, not the top cluster.}
\end{align*}
\begin{definition}[Notation for annuli and semistable covering]\label{def:covering_of_X}

We write $\tilde{U}_\ss$ for  $X|_{U_\ss}$ when $\ss$ is not \"ubereven. Otherwise, write $\tilde{U}_\ss^\pm$ for the two irreducible components of $X|_{U_\ss^\pm}$.
We let $\tilde{A}_\ss$ denote $X|_{A_\ss}$, and $\tilde{A}_{\ss}^{\pm}$ the two different annuli if $\ss$ is even.
If we write $\tilde{U}_{\ss}^\pm$ or $\tilde{A}_{\ss}^{\pm}$ for a non-\"ubereven or odd cluster, we mean $\tilde{U}_{\ss}$ or $\tilde{A}_\ss$. 
\end{definition} 

We described what $X$ looks like over each $U_\ss$ and $A_\ss$. We want to understand how to glue these together compatibly over pairwise intersections of $U_\ss$. When $\ss$ is \"ubereven, for every $\ss'<\ss$, we know that $U_\ss^\pm \supset A_{\ss'}^\pm$, and $U_\ss^\pm \supset A_\ss^\pm$ if $\ss$ is not the top cluster. The compatibility conditions ensure that the trivialisations of $X|_{U_\ss}$ and $X|_{A_\ss}$ are compatible with one another (i.e., that $A_{\ss}^\pm \subset U_{\ss}^\pm$ instead of $A_{\ss}^\mp \subset U_{\ss}^\pm$).

\begin{definition}[\"Ubereven compatibility criterion]
We assume a compatibility condition between the various inverse square roots of $g_\ss$ and $h_\ss$, which takes place over the \"ubereven clusters. Namely, we suppose the following.
\begin{itemize}
\item If $\ss$ is \"ubereven and not the top cluster, then $g_\ss(x)^{1/2} = 1$. (Since $g_\ss(x) = 1$ here, this says we take the obvious choice of square root.)
\item If $\ss$ is \"ubereven and $\ss'$ is a child of $\ss$, then $h_\ss(x)^{1/2}|_{A_{\ss'}} = g_{\ss'}(x)^{1/2}h_{\ss'}(x)^{1/2}|_{A_{\ss'}}$. (Both sides are square roots of $f$ so we are again fixing a sign.)
\end{itemize}
It is always possible to make these conditions hold after possibly changing some of the $h_\ss(x)^{1/2}$ by a sign.
\end{definition}

The following lemma will be used in the proof of \Cref{the:minimalSkeleton}.

\begin{lemma}\label{lem:step2}
Let $\calR'$ be $\calR$ if $\deg(f)$ is even and $\calR \cup \{ \infty \}$ if $\deg(f)$ is odd. So $\calR'$ is the set of ramification points of $X \to \P^1_K$. Let $D \subset \P^{1,\an}_{\bC_\ell}$ be an open disc (i.e. either an open disc in $\A^{1, \an}_{\bC_\ell}$ or the complement in $\P^{1,\an}_{\bC_\ell}$ of a closed disc in $\A^{1, \an}_{\bC_\ell}$) containing at most one ramification point. Then $X_{\bC_\ell}|_D$ is isomorphic to
\begin{enumerate}
\item a disjoint union of two open discs, each mapping isomorphically to $D$ if $\#D \cap \calR' = 0 $;
\item an open disc, mapping to $D$ by the squaring map if $\#D \cap \calR' =1$.
  \end{enumerate}
\end{lemma}

\begin{proof}
	The proof is analogous to the proof of \Cref{lem:step1}.
\end{proof}

The main theorem in this section is the description of a semistable covering of $X^{\an}$.

\begin{theorem}
\label{thm:semistablecoverX}
Let $\pi: X^\an \to \P^{1, \an}_{\bC_\ell}$ be the map induced by $\pi: X \to \P^1$.  
Then $\fU \colonequals \{\tilde{U}_\ss^\pm \}$ is a semistable covering of $X^{\an}$ (where $\tilde{U}_\ss^\pm= \tilde{U}_\ss$ when $\ss$ is not \"ubereven).
\end{theorem}

\begin{proof}

Note that the canonical reduction of $\tilde{U}_{\ss}$ is a non-empty open inside the curve $y^2 = g_\ss(x)$. This automatically implies that $\tilde{U}_{\ss}$ is irreducible if $\ss$ is not \"ubereven, and splits as a disjoint union of two irreducibles if $\ss$ is \"ubereven. We see that for any $\ss$, the inverse image $\tilde{U}_{\ss,0}^\pm$ of $D_\ss^c \setminus  \bigcup_{\substack{\ss' < \ss \\ \ss' \text{ proper}}}  D_{\ss'}^o \subset U_\ss$ inside $\tilde{U}_{\ss}^\pm$ makes $(\tilde{U}_{\ss}^\pm, \tilde{U}_{\ss,0}^\pm)$ a basic wide open. Double intersections $\tilde{U}_\ss \cap \tilde{U}_\ss'$ are empty, or one of the annuli $\tilde{A}_\ss$ or $\tilde{A}_{\ss'}$, and triple intersections are empty.
\end{proof}

\subsection{The minimal skeleton of a hyperelliptic curve}\label{subsec:minimalSkeleton}

We can now write down the minimal skeleton of $X^\an_{\bC_\ell}$. The results here give an interpretation of the results in \cite[Section~8]{DDMM} in the language of Berkovich spaces.

For a proper cluster $\ss$, let $\eta_\ss \in |\P^{1, \an}_{\bC_\ell}|$ denote the Gauss point of the disc $D_\ss^c$ (this corresponds to a vertex in the semistable vertex set in $\P^{1,\an}$ corresponding to the semistable covering of Definition~\ref{def:UsandAs}). This is a type $\II$ point. If $\ss$ is not the top cluster, let $e_\ss \subset |\P^{1, \an}_{\bC_\ell}|$ be the skeleton of the annulus $A_{\ss, \bC_\ell}$. This is an open interval of length equal to the width of $A_\ss$ (i.e. the relative depth $\delta_\ss$) embedded isometrically into $|\P^{1, \an}_{\bC_\ell}|$. Moreover, the skeleton $e_\ss$ is canonically oriented where $\partial_0(e_\ss) = \eta_\ss$ and $\partial_1 (e_\ss) = \eta_{P(\ss)}$ are the ``left and right hand limit points''.

\begin{theorem}\label{the:minimalSkeleton}
Let $\pi: X \to \P^1$ be the projection, inducing the map $\pi: |X^{\an}_{\bC_\ell}| \to |\P^{1, \an}_{\bC_\ell}|$ on the underlying topological spaces of Berkovich spaces. Then \\
(a) for a proper cluster $\ss$, the preimage of $\eta_\ss$ in $|X^\an_{\bC_\ell}|$ consists of:
\begin{itemize}
\item two type $\II$ points $\tilde{\eta}_\ss^+$ and $\tilde{\eta}_\ss^-$ of genus $0$ if $\ss$ is \"ubereven;
\item one type $\II$ point $\tilde{\eta}_\ss$ of genus $g(\ss) \colonequals \lceil \frac{\#\text{odd children of $\ss$}}{2} - 1 \rceil $ otherwise.
\end{itemize}
(b) For a proper non-top cluster $\ss$, the preimage of $e_\ss$ in $|X^{\an}_{\bC_\ell}|$ consists of 
\begin{itemize}
\item two embedded oriented intervals $\tilde{e}_\ss^+$ and $\tilde{e}_\ss^-$ of length $\delta_\ss$, each mapping oriented-isomorphically onto $e_\ss$ with $\partial_0(\tilde{e}_\ss^{\pm}) = \tilde{\eta}_\ss^\pm$ and  $\partial_1(\tilde{e}_\ss^{\pm}) = \tilde{\eta}_{P(\ss)}^\pm$ if $\ss$ is even. (We permit ourselves to write $\tilde{\eta}^\pm_\ss = \tilde{\eta}_\ss$ if $\ss$ is not \"ubereven.)
\item one embedded oriented interval $\tilde{e}_\ss$ of length $\delta_\ss/2$ mapping onto $e_\ss$ by an oriented dilation of scale factor $2$, with $\partial_0(\tilde{e}_\ss) = \tilde{\eta}_\ss$ and $\partial_1(\tilde{e}_\ss) = \tilde{\eta}_{P(\ss)}$ if $\ss$ is odd.
\end{itemize}
(c) The points $\tilde{\eta}_\ss^\pm$ for $\ss$ a proper cluster form a semistable vertex set $V$ for $X^{\an}_{\bC_\ell}$, whose associated skeleton $\Gamma$ is the union of the $\{\tilde{\eta}_\ss^\pm\}$ and the $\tilde{e}_\ss^\pm$. Moreover, the choices of 
$\tilde{\eta}_\ss^\pm$ for $\ss$ \"ubereven and of $\tilde{e}_\ss^\pm$ for $\ss$ even are canonical given our choices of inverse square roots (i.e. our choices of inverse square roots allow us to distinguish between $\tilde{\eta}_\ss^+$ and $\tilde{\eta}_\ss^-$).
\end{theorem}

\begin{remark}
One can check that $\Gamma$ is the minimal skeleton of $X^\an_{\bC_\ell}$. However, $V$ is not necessarily a minimal vertex set for $X^{\an}_{\bC_\ell}$. A minimal vertex set is given by the points $\tilde{\eta}^\pm_\ss$ for proper non-top clusters $\ss$ except twins (clusters of size 2), cotwins (non-\"ubereven clusters with a child of size 2 times the genus of $X$), as well as the top cluster if it has $\geq 3$ children. See \cite[Theorem~8.5]{DDMM}.
\end{remark}

\begin{example}
\label{ex:clusterpicture}
The left side of \Cref{fig:Berkovichskeleton} shows an example cluster picture
\begin{figure} 
\includegraphics[scale=.25]{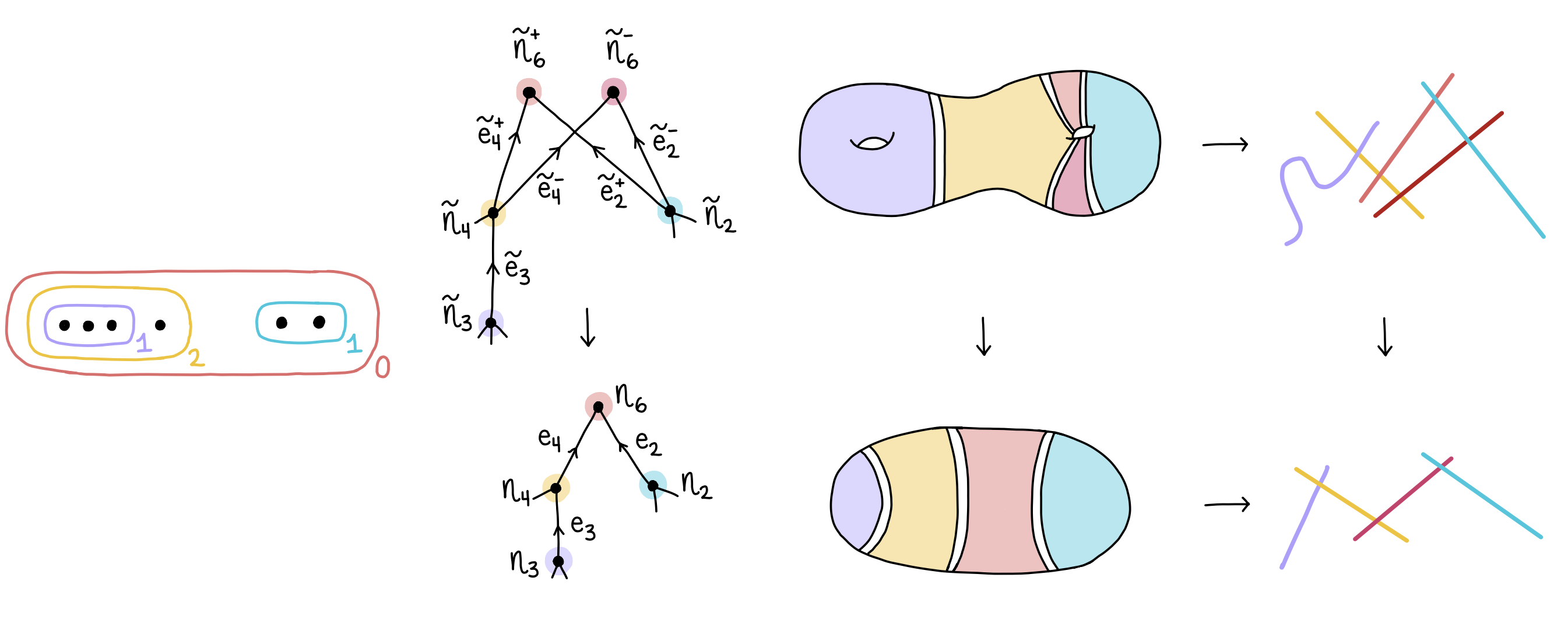}
  \caption{Cluster, Berkovich skeleton, special fibre, and reduction map}
  \label{fig:Berkovichskeleton}
 \end{figure}
 where subscripts denote relative depths. We label the four proper clusters $s_2, s_3, s_4,$ and  $s_6$ according to their size. Let us read off the minimal skeleton. We have five vertices, $\tilde{
 \eta}_2, \tilde{
 \eta}_3, \tilde{
 \eta}_4, \tilde{
 \eta}_6^+, \tilde{
 \eta}_6^-$ and five oriented edges $\tilde{e}^+_2, \tilde{e}^-_2, \tilde{e}_3, \tilde{e}^+_4, \tilde{e}^-_4$. These are connected up as in the graph on the upper left of \Cref{fig:Berkovichskeleton}.
Moreover, $\tilde{
 \eta}_3$ has genus $1$ and all other vertices have genus $0$. The edge lengths are given by $l(\tilde{e}^\pm_2) = 1$, $l(\tilde{e}_3) = 1/2$, and $l(\tilde{e}^\pm_4) = 2$.

 This Berkovich skeleton associated to $X$ covers the Berkovich skeleton of a semistable model of $\P^1(\calR)$. This skeleton is depicted in the lower left. Moving right, we have an artistic depiction of the Berkovich spaces of $X$ and $\P^{1,\an}$ decomposed into wide opens and annuli, corresponding to the semistable covering. These come equipped with reduction maps to their special fibres, depicted to its right.
\end{example}

\begin{proof}[Proof of \Cref{the:minimalSkeleton}] 
(a) If $\ss$ is an \"ubereven cluster, we know that $X|_{U_\ss} \simeq X_\ss |_{U_\ss} = U_\ss \times \{\pm 1\}$ via the isomorphism in \eqref{eqn:XtoUs}. Since $\eta_\ss \in |U_{\ss, \bC_\ell}|$, its preimage  $\pi^{-1}(\eta_\ss)\subseteq X|_{U_\ss}$ thus consists of the two type $\II$ points $\tilde{\eta}_\ss^{\pm} = (\eta_\ss, \pm 1)$, both of genus $0$.

If $\ss$ is an odd cluster, then the fibres of the map $\pi$ are described in ($\ast$) in the proof of \cite[Proposition~3.4.6]{Berkovich1990SpectralTheoryAnalyticGeometry}; this shows that the fibres are size 1 or 2.  Therefore, the cover $|X^{\an}_{\bC_\ell}| \to |\P^{1, \an}_{\bC_\ell} |$ is a ramified topological double covering.  In particular, the ramification locus is closed.  For any odd non-top cluster $\ss$, \Cref{lem:step1} shows that $e_\ss$ lies in the ramification locus (the squaring map on annuli is ramified along the skeleta).  Then $\eta_\ss$ lies in the ramification locus, since it lies in the closure of $e_\ss$.  If $\ss$ is the top cluster and odd, then $f$ has odd degree, so $\infty \in \calR'$.  In this case, \Cref{lem:step2} implies that $X_{\bC_\ell}^\an|_{D_\infty}$ consists of a single disc mapping via the squaring map to $D_\infty = \P^{1, \an}_{\bC_\ell} \setminus D^c_\ss$. Hence the map $|X^{\an}_{\bC_\ell}| \to |\P^{1, \an}_{\bC_\ell} |$ is again ramified over the open arc from $\eta_\ss$ to $\infty$ in $|\P^{1, \an}_{\bC_\ell}|$, and so $\eta_s$ lies in its ramification locus. In either case, the point $\eta_\ss$ lies in the ramification locus, so has one preimage in $|X^{\an}_{\bC_\ell}|$, and therefore is of type $\II$.

To finish part (a), we need to prove the assertion regarding the genus. For this we use the fact that, if $\overline{k}$ denotes the residue field of $\bC_\ell$  (an algebraic closure of the residue field $k$ of $K$), and if $\calX_{\tilde{\eta}_\ss}$ denotes the smooth projective curve over $\overline{k}$ associated to $\tilde{\eta}_\ss$, then there is a bijection
\[
\calX_{\tilde{\eta}_\ss}(\overline{k}) \simeq T_{\tilde{\eta}_\ss} X^\an_{\bC_\ell} = \text{ tangent directions at } \tilde{\eta}_\ss \in |X^{\an}_{\bC_\ell}| 
\]
\cite[Lemma~5.12(3)]{Baker2013NonArchimedeanAnalyticCurves}.  Moreover, this identification is natural with respect to finite morphisms of curves, so we have the induced map 
\[ \calX_{\eta_\ss}(\overline{k}) \to P_{\eta_\ss}(\overline{k}) \simeq \P^1_{\overline{k}}(\overline{k}),\]
where  $P_{\eta_\ss}(\overline{k})$ is the component of $\P^1_{\overline{k}}$ given by $\overline{\red(\eta)}$, see \eqref{eq:berkovich_reduction}.

Our strategy now is to describe the tangent directions in $|\P^{1, \an}_{\bC_\ell}|$ at $\eta_\ss$ and their lifts to tangent directions at $\tilde{\eta}_\ss$ in $|X^{\an}_{\bC_\ell}|$.   There are three different classes of tangent directions.   First, the ``upward tangent direction'' goes along $e_\ss$ for $\ss$ non-top; if $\ss$ is the top cluster then we mean the tangent direction on the arc from $\eta_\ss$ to $\infty$.  These lift to two tangent directions if $\ss$ is even, and one if $\ss$ is odd.  Second, the ``downward tangent direction'' goes along $e^{-1}_{\ss'}$ for every $\ss'<\ss$ proper child.  These lift to two tangent directions if $\ss'$ is an even proper child, and one if $\ss'$ is odd proper child.  Third, there is a tangent direction for every disc component of $\P^{1, \an}_{\bC_\ell} \setminus V$ whose closure contains $\eta_\ss$. In this case there are two lifts except if $D$ contains a root of $f$ (necessarily unique).

As a consequence, the number of ramification points of $\calX_{\tilde{\eta}_\ss} \to \P^1_{\overline{k}}$ is equal to the number of odd children of $\ss$, plus $1$ if $\ss$ itself is odd. Since $\calX_{\tilde{\eta}_\ss} \to \P^1_{\overline{k}}$ has degree $2$, this implies that $\calX_{\tilde{\eta}_\ss}$ is a hyperelliptic curve over $\overline{k}$ of genus $g(\ss)$.

(b) Let $\ss$ be a proper, non-top cluster.  Since $e_\ss$ is the skeleton of $A_\ss$, it follows from \Cref{lem:step1} that for a proper, non-top even (resp. odd) cluster $\ss$, the set $\pi^{-1}(e_\ss)$ is the union of two (resp. 1) embedded open intervals, each of length $\delta_\ss$ (resp. $\delta_\ss/2$).

It remains to compute the endpoints of $\pi^{-1}(e_\ss)$.  We know that $\partial_0(\tilde{e}_\ss^{\pm})$ must be a preimage of $\eta_\ss = \partial_0(e_\ss)$. If $\ss$ is non-\"ubereven, there is only one preimage. When $\ss$ is \"ubereven, we want to show that $\partial_0(\tilde{e}_\ss^{+}) = \tilde{\eta}_\ss^+$ rather than $\tilde{\eta}_\ss^- $. This is a consequence of our compatibility assumptions. A similar argument deals with $\partial_1$.

(c) Recall that a semistable vertex set $V$ is a finite set of type $\II$ points such that $X^{\an}_{\bC_\ell} \setminus V$ is a disjoint union of open discs and finitely many open annuli. By definition, its skeleton is the union of $V$ and the skeleton of the annuli in $X^{\an}_{\bC_\ell} \setminus V$.

Let $V_0 \subset |\P^{1,\an}_{\bC_\ell}|$ be the set consisting of the points $\eta_\ss$, where $\ss$ ranges over all proper clusters.  This is a semistable vertex set, since it is a finite set and  $\P^{1,\an}_{\bC_\ell} \setminus V_0$ is the disjoint union of the annuli $A_\ss$ (for $\ss$ proper, non-top) and open discs each of which contains at most $1$ ramification point of $f$. Let $\Gamma_0$ be the skeleton of $V_0$.  Since $\pi: X \to \P^1$ is finite, it preserves types of points and so $\pi^{-1} (V_0) = V$ is a finite set of type $\II$ points. Moreover, $X_{\bC_\ell}^\an \setminus V$ is the disjoint union of $X^\an_{\bC_\ell}|_{A_\ss}$ (for $\ss$ proper, non-top) and $X^\an_{\bC_\ell}|_{D}$ for $D$ an open disc component in $\P^{1, \an}_{\bC_\ell} \setminus V_0$. These are (disjoint unions of) open annuli and discs by Lemmas~\ref{lem:step1}~and~\ref{lem:step2}, so $V$ is a semistable vertex set.
Moreover, \Cref{lem:step1} shows that the preimage of the skeleton of $A_\ss$ is the skeleton of $X^{\an}_{\bC_\ell}|_{A_\ss}$, so $\pi^{-1}(\Gamma_0) = \Gamma$ is the skeleton associated to $V$, concluding our proof.  
\end{proof}

Finally, we give a theorem to decide if a component of the dual graph contains the reduction of $\bQ_\ell$-points.  This can be used to decide when local heights are trivial, which makes the quadratic Chabauty method simpler. First we set up some notation.
\begin{definition}[Components of special fibres associated to clusters]
Let $X/\bQ_\ell$ be a hyperelliptic curve. Let $\ss$ be a cluster of $X$. Let $\calX$ be the semistable model of $X_{\bC_\ell}$corresponding to the semistable covering from \Cref{thm:semistablecoverX}. We write $\bar{\calX}_\ss^\pm$ for the component(s) of the special fibre of $\calX$ corresponding to $\ss$ and $g(\ss)$ for the genus of these components.
\end{definition}

\begin{theorem}\label{thm:reductiontocomponents}
Let $X/\bQ_\ell$ be a hyperelliptic curve.
Let $\ss$ be a proper cluster of $X$.
Write $\bar{\calX}_\ss^{\pm, \circ} = \bar{\calX} \setminus \cup_{\ss' \neq \ss} \bar{\calX}_{\ss'}^\pm$.
If $\bar{\calX}_\ss^{\pm, \circ}$ contains the reduction of a $\bQ_\ell$-point, then at least one of the following is true:
\begin{enumerate}
\item $\ss$ has a singleton child which is $\bQ_\ell$-rational;
\item $\ss$ is the top cluster and $X$ has a $\bQ_\ell$-rational point at infinity;
\item $\nu_\ss \in 2 \bZ$.
\end{enumerate}
Furthermore, if $\bar{\calX}_\ss^\pm \cap \bar{\calX}_{\ss'}^\pm$ contains the reduction of a $\bQ_\ell$-point and $\ss = P(\ss')$ then $(\nu_\ss, \nu_\ss')$ contains an even integer.
\end{theorem}
\begin{proof}
	Let~$V_0\subset|\bP^{1,\an}_{\bC_\ell}|$ be the set of Gauss points~$\eta_\ss$ attached to proper clusters~$\ss$. Let~$\Gamma_0$ be the skeleton of~$\bP^{1,\an}_{\bC_\ell}$ corresponding to~$V_0$, and let~$\Gamma_1$ be the skeleton of the open curve~$\bA^{1,\an}_{\bC_\ell}\smallsetminus\calR$ where~$\calR$ is the set of roots of~$f$. Though we have not explained the full definition of the skeleton of an open curve \cite[Definition~3.3]{Baker2013NonArchimedeanAnalyticCurves}, in this case, we can say explicitly that~$\Gamma_1$ is the union of~$\Gamma_0$ and a number of open rays isometric to~$(0,\infty)$. We have one such ray for each root~$r$ of~$f$, connecting~$r\in|\bP^{1,\an}_{\bC_\ell}|$ to $\eta_\ss$ where~$\ss$ is the smallest proper cluster containing~$r$, and one additional ray connecting the Gauss point attached to the top cluster to the point at~$\infty$. According to \cite[Definition~3.7]{Baker2013NonArchimedeanAnalyticCurves}, there is a canonical retraction~$\tau_1\colon |\bA^{1,\an}_{\bC_\ell}|\smallsetminus\calR\to\Gamma_1$.
	
	Now suppose that~$(x,y)\in X(\bQ_\ell)$ is~$\bQ_\ell$-rational. We assume without loss of generality that~$(x,y)$ is not a point at infinity or a Weierstrass point. Note that~$(x,y)$ reduces to~$\bar\calX^{\pm}_\ss$ if and only if~$x\in U_\ss$, and~$(x,y)$ reduces to~$\bar\calX^{\pm}_\ss\cap\bar\calX^{\pm}_{P(\ss)}$ if and only if~$x\in A_\ss$. There are four cases to consider:
	
	Firstly, suppose that~$\tau_1(x)$ lies on the ray connecting a root~$r$ of~$f$ to~$\eta_\ss$, where~$\ss$ is the parent of~$\{r\}$. This ray is the skeleton of a punctured open disc~$D\smallsetminus\{r\}$ where~$D\ni x$. Since~$D$ contains a $\bQ_\ell$-rational point, it must be setwise invariant under the action of the Galois group of~$\bQ_\ell$, and so too must be~$\{r\}=D\cap\calR$. So~$r$ is a $\bQ_\ell$-rational root of~$f$ which is a singleton child of~$\ss$ and~$x$ reduces to~$\bar\dX_\ss^{\pm,\circ}$.
	
	Secondly, suppose that~$\tau_1(x)$ lies on the ray connecting~$\eta_\ss$ to~$\infty$ for~$\ss$ the top cluster. This ray is the skeleton of a punctured disc~$D\smallsetminus\{\infty\}$ where~$D\subset\bP^{1,\an}_{\bC_\ell}$ is an open disc around~$\infty$ (complement of a closed disc in~$\bA^{1,\an}_{\bC_\ell}$). There are two possibilities, depending on the degree of~$f$. Either the preimage of~$D$ inside~$X^\an_{\bC_\ell}$ is an open disc mapping to~$D$ via a finite degree~$2$ map ramified over~$\infty$, or it is a disjoint union of two open discs mapping isomorphically to~$D$. In either case, $X$ has a $\bQ_\ell$-rational point at~$\infty$ (e.g.\ in the second case, one of the two discs must contain~$(x,y)$, and so be fixed by the natural Galois action). In this case,~$(x,y)$ reduces to~$\bar\calX_\ss^{\pm,\circ}$ for~$\ss$ the top cluster.
	
	Thirdly, if~$\tau_1(x)=\eta_\ss$ is the Gauss point of some cluster~$\ss$, then we have
	\[
	v(f(\eta_\ss)) = v(f(x)) = 2v(f(y)) \in 2\bZ
	\]
	by the slope formula \cite[Theorem~5.15]{Baker2013NonArchimedeanAnalyticCurves}. But it is easy to check, e.g.\ by factorising~$f$, that~$\nu_\ss=v(f(\eta_\ss))$, so~$\nu_\ss\in 2\bZ$. In this case, $x$ lies in~$U_\ss$ but not any of the annuli $A_{\ss'}$, so~$(x,y)$ reduces to~$\bar\calX_\ss^{\pm,\circ}$.
	
	Finally, if~$\tau_1(x)$ lies in the skeleton of an annulus~$A_\ss$, then arguing as above, we have that $v(f(\tau_1(x)))=v(f(x))\in 2\bZ$. But the skeleton of~$A_\ss$ is an open interval in~$|\bP^{1,\an}_{\bC_\ell}|$ connecting the two Gauss points~$\eta_\ss$ and~$\eta_{P(\ss)}$. Since the valuation of~$f$ is a linear function along this interval \cite[Theorem~5.15(2)]{Baker2013NonArchimedeanAnalyticCurves}, it follows that~$v(f(\tau_1(x)))\in (\nu_{P(\ss)},\nu_\ss)$ and we are done.
\end{proof}

\begin{example}\label{ex:trivialHeights}
	Suppose that~$X$ has the cluster picture shown in \Cref{fig:cluster022_12_12_0} at~$\ell$ and that the leading coefficient of~$X$ is a unit in~$\bZ_\ell$.
	Then the local height of any~$\bQ_\ell$-point on~$X$ is~$0$.

\begin{figure}[h!]
\begin{center}
\includegraphics[width=4cm]{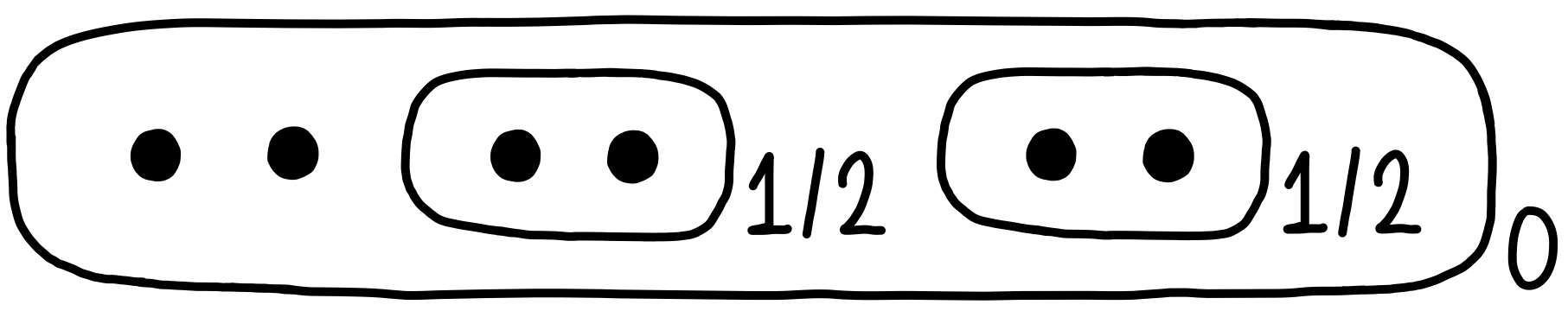}
\end{center}
\caption{Cluster picture}
\label{fig:cluster022_12_12_0}
\end{figure}
\end{example}

\subsection{Computing the Coleman--Iovita maps}
In this section, we use the semistable covering $\fU$ to describe the maps from \S\ref{subsec:descriptionCI} for hyperelliptic curves. Recall that $\phi_2$ is needed to translate the action of the endomorphism $Z_\ast$ on $\rH^1_\dR(X/K)$ to the action of $Z_\ast$ on $\rH_1(\Gamma,K)$.  In addition, we need $\phi_1$ to compute the trace of $Z_\ast$ on $\rH^1_\rig(\bar{\dX}_\ss/K)$ for higher genus clusters $\ss$. 

The following lemma gives an explicit way of describing the map $\phi_2:\rH^1_\dR(X/K)\to \rH_1(\Gamma, K)$ for hyperelliptic curves by reducing the calculation to computing residues on $\P^1$.

\begin{lemma}\label{lem:phi2}
	Let $X$ be a hyperelliptic curve, and $\ss$ be a non-top cluster with associated annulus $A_\ss\subseteq\A^{1,\an}_K$.    Let $\omega=\frac{P(x)\rd x}{2y}\in \rH^1_\dR(X/K)$, where $P(x)$ is a polynomial.  If $\ss$ is odd, then $\Res_{\pi^{-1}(A_\ss)}(\omega)=0$.  If $\ss$ is even, recall that $\pi^{-1}(A_\ss)=\tilde{A}_\ss^+\sqcup \tilde{A}_\ss^-$.  In this case,
	\begin{equation*}
		-\Res_{\tilde{A}_\ss^-}(\omega) =\Res_{\tilde{A}_\ss^+}(\omega) = \Res_{A_\ss}\left(\frac{P(x)\rd x}{2g_\ss(x)^{1/2}h_\ss(x)^{1/2}}\right),
	\end{equation*}
	where $t$ is any parameter on the annulus.
\end{lemma}

\begin{proof}
	For any proper cluster~$\ss'$ contained in~$\ss$, applying \Cref{lem:cohomology_of_wide_open_de_rham} to $\tilde U_{\ss'}$ (or both components of~$\tilde U_{\ss'}$ if~$\ss'$ is \"ubereven) shows that
	\[
	\sum_{\ss''<\ss'}\Res_{\tilde A_{\ss''}^\pm}(\omega) - \Res_{\tilde A_{\ss'}}(\omega) = 0 \quad\text{or}\quad \sum_{\ss''<\ss'}\Res_{\tilde A_{\ss''}^\pm}(\omega) - \Res_{\tilde A_{\ss'}^+}(\omega) - \Res_{\tilde A_{\ss'}^-}(\omega) = 0 \,,
	\]
	according to whether~$\ss'$ is odd or even. We use our conventions for the orientations of annuli as in \Cref{ex:wide_open_in_projective_line}. Summing this identity over all proper clusters~$\ss'$ contained in~$\ss$ shows that $\Res_{\tilde A_\ss}(\omega)=0$ if~$\ss$ is odd, and $\Res_{\tilde A_\ss^+}(\omega)+\Res_{\tilde A_\ss^-}(\omega)=0$ if~$\ss$ is even.
	
	Finally, when~$\ss$ is even, the isomorphism $A_\ss \simeq \tilde A_\ss^+$ is given by $x\mapsto(x,g_\ss(x)^{1/2}h_\ss(x)^{1/2})$ by~\eqref{eqn:Xannuluseven}. So~$\omega=\frac{P(x)\rd x}{2y}$ pulls back to the differential form $\frac{P(x)\rd x}{2g_\ss(x)^{1/2}h_\ss(x)^{1/2}}$ on~$A_\ss$ and we are done.
\end{proof}

In order to explicitly compute $\phi_1$, we use the semistable covering $\fU$.  Let $\omega\in\ker(\phi_2)$.  Recall that $\omega|_{\tilde{U}_\ss}$ has residue 0 over all bounding annuli of $\tilde{U}_\ss$.  Therefore, $\omega\in\rH^1_\dR(\bar{\calX}_{\ss}/K)$.  Our goal is to compute this restriction in terms of a basis $\eta_1,\ldots,\eta_{2g(\ss)}$ of $\rH^1_\dR(\bar{\calX}_{\ss}/K)$.  There exist $a_i\in K$ such that
\begin{equation}\label{eqn:omegaInBasis} 
	\omega|_{\tilde{U}_\ss}=\sum_{i=1}^{2g(\ss)}a_i\eta_i|_{\tilde{U}_\ss} + \text{exact form}.
\end{equation}
To compute the coefficients $a_i$, we use global and local symbols (see \cite{BesserSyntomicII} and \cite[Definition~2.4]{BalaBesserIMRN}).

\begin{definition}
Let~$X$ be a smooth projective curve over $K$, let~$D_1^c,\dots,D_k^c$ be some pairwise disjoint closed discs in the analytification~$X^\an$, and let~$U\coloneqq X^\an\smallsetminus\bigcup_iD_i^c$ be the complement of these discs. Choose some open discs $D_i^o\supset D_i^c$, still pairwise disjoint, and let~$A_i\coloneqq D_i^o\smallsetminus D_i^c\subset U$ be the corresponding annulus bounding the disc~$D_i^c$. 
If~$\omega,\eta$ are two differential forms on~$U$, each with residue~$0$ on each annulus~$A_i$, then the \emph{global symbol}  $\langle \omega, \eta\rangle $ is defined to be a sum of \emph{local symbols} $\langle \omega, \eta \rangle_{A_i}$
\[
\langle\omega,\eta\rangle \coloneqq \sum_{A_i} \langle \omega, \eta \rangle_{A_i},
\] where
\[
\langle \omega, \eta \rangle_{A_i} \colonequals -
\Res_{A_i}\left(\omega\cdot\int\eta\right) \in K.
\]
\end{definition}

Since $\eta$ has residue~$0$ on each~$A_i$, it has a formal antiderivative $\int\eta$ on~$A_i$, which is unique up to an additive constant of integration. Because $\omega$ also has residue $0$, the residue of~$\omega\cdot\int\eta$ on~$A_i$ is independent of this constant of integration.

We now describe how to practically compute these local symbols in our use case, where $\omega \in \ker\phi_2$, and $\eta$ is a differential on $X_\ss$ for some cluster $\ss$. First, we remark that in any and all calculation, to simplify the formulas both $\omega$ and the $\eta_i$ are chosen to be in the $-1$ eigenspace of the hyperelliptic involution $\iota$ of $X$, i.e. of the form $a(x)\frac{\rd x}{2y}$ for $a$ a polynomial.

Recall that the choice of a square root $h_\ss^{1/2}$ determines an isomorphism $X|_{U_{\ss}} \to X_\ss|_{U_{\ss}}$. We let $\tilde{\omega}$ denote the differential on $X_{\ss}$ that $\omega$ maps to under this isomorphism. The formula for local symbols on $X_{\ss}$ becomes
\begin{equation}
\label{eq:globalsymbolhyperelliptic}
\langle \tilde{\omega}, \eta \rangle =
- \sum_{\ss' \leq \ss} \Res_{\tilde{A}_{\ss'}^\pm}\left(\tilde{\omega}\cdot\int\eta\right) \in K.
\end{equation}
Here we are taking the orientation on $\tilde{A}_{\ss'}^\pm$ induced by their inclusion in the wide open $U_\ss$.
To compute $\Res_{\tilde{A}_\ss^\pm}\left(\tilde{\omega}\cdot\int\eta\right)$, pick a parameter of $\tilde{A}_{\ss}^\pm$. If $\ss$ is even, this parameter will be a parameter of $A_\ss^\pm$, and if $\ss$ is odd it will be a square root of a parameter on $A_\ss$. 
We can then express everything in this parameter by using our choice of a local square root of $g_\ss$, the defining polynomial for $X_\ss$. 
Note that the resulting local symbol does not depend on the sign choice we made for $g_\ss^{1/2}$, since by assumption $\iota(\omega) = -\omega$ and $\iota(\eta) = -\eta$. A different sign choice would multiply the global symbol by $(-1)^2$. 
In particular, we find $\Res_{\tilde{A}_\ss^+}\left(\tilde{\omega}\cdot\int\eta\right) = \Res_{\tilde{A}_\ss^-}\left(\tilde{\omega}\cdot\int\eta\right)$. Write $\omega = a(x)\rd x/(2y)$ on the curve $X$. Write $\tilde{\omega} = a(x) \rd x / (2 h_\ss^{1/2} u_{\ss})$ and $\eta = b(x)\rd x/(2u_{\ss})$ on the curve $u_{\ss}^2 = g_\ss(x)$. We compute the formula.
\begin{equation}
\label{eq:globalsymbolhyperellipticexplicit}
\langle \tilde{\omega}, \eta \rangle =
\sum_{\ss' \leq \ss} \sigma_{\ss,\ss'}\begin{cases} \frac12 \Res_{x = \alpha_{\ss'}}\left(a(x)h_{\ss}^{-\frac12}g_{\ss}^{-\frac12 }\rd x\cdot\int b(x) g_{\ss}^{-\frac12}\rd x\right) & \text{if } \ss' \text{ is even}\\ \Res_{t_{\ss'} = 0}\left( a(x)h_{\ss}^{-\frac12}(x) g_{\ss,\ss'}^{-\frac12}(x)\rd t_{\ss'} \cdot \int b(x)g_{\ss,\ss'}^{-\frac12}(x) \rd t_{\ss'}\right) &\text{if } \ss' \text{ is odd}\end{cases}\in K
\end{equation}
where $\sigma_{\ss,\ss'} = -1$ if $\ss = \ss'$ and $+1$ otherwise, and in the odd case we recall $t_{\ss'}$ is the coordinate on $X|_{A_{s'}}$ from \eqref{eqn:Xannulusodd}, and $x = t_{\ss'}^2 + \alpha_{\ss'}$, and in the odd case $g_{\ss,\ss'}$ is shorthand for $g_{\ss}/(x-\alpha_{\ss'})$. The choice of square root of $g_{\ss,\ss'}$ does not matter,since a different choice multiplies the global symbol by $(-1)^2$. (However, we have already made a choice of square root of $g_{\ss,\ss'}$; this choice can be determined by asking that $h_{\ss}^\frac12 g_{\ss,\ss'}^{\frac12}$ and $h_{\ss'}^{\frac12} g_{\ss',\ss'}^{\frac12}$ are the same square root of $f/(x-\alpha_{\ss'})$.)

\begin{proposition}{\cite[Proposition~4.10]{BesserSyntomicII} cf.\ \cite[Proposition~2.5]{BalaBesserIMRN}}
Suppose $\omega$ and $\eta$ are restrictions of algebraic $1$-forms of the second kind on $X$. Let~$[\omega],[\eta]\in\rH^1_\deR(X/K)$ be the associated de Rham cohomology classes.
\[
\langle\omega,\eta\rangle = [\omega]\cup[\eta] \in \rH^2_\deR(X/K) = K \,.
\]
In particular, $\langle\omega,\eta\rangle$ depends only on the de Rham cohomology classes represented by~$\omega$ and $\eta$. 
\end{proposition}

Now we are ready to solve for the coefficients $a_i$ from \eqref{eqn:omegaInBasis}.  For each $j=1,\ldots, 2g(\ss)$, we compute
\begin{equation}
\label{eq:omegaeta}
\langle\omega,\eta_j\rangle=\sum_{i=1}^{2g(\ss)}a_i\langle \eta_i,\eta_j\rangle.
\end{equation}
This gives a system of $2g(\ss)$ linear equations in the $2g(\ss)$ variables $a_1,\dots,a_{2g(\ss)}$, and this system of equations has a unique solution since the cup product pairing on $\rH^1_\deR(\bar{\calX}_{\ss}/K_\ell)$ is perfect. 

\section{Computations}\label{sec:computations}
We give more details in the computation of the Coleman--Iovita morphisms, as well as how this results in the provably correct rational matrix of the action of $Z_\ast$ on $H_1(\Gamma,\Z)$, and integer traces $\tr_v(Z)$. We do this by giving explicit bounds on the absolute values of the entries of the matrix and the integer traces, valid for any curve, and then showing how for hyperelliptic curves we can compute these invariants up to any desired precision.

We first focus on the computations of the Coleman--Iovita morphisms from \Cref{sec:ExplicitColemanIovitaHyp}. As explained in \Cref{lem:phi2}, determining $\phi_2$ for a hyperelliptic curve reduces to computing residues over annuli $A_\ss$ on $\P^1$ for all even clusters $\ss$.  To do this explicitly, we use \Cref{lem:inversesqroot} to construct an inverse square root of $f(t)$, where $t$ is a choice of parameter in $A_\ss$.  In order to have a provably correct result, it is important for our calculations to keep track of the $p$-adic precision of the approximations we make; this is why we do the following $p$-adic analysis.

\subsection{Finite precision calculus on annuli}
\label{sub:finpreccalculus}

In this section, we assume $(K, v: K \to \R_{\geq 0} \cup \{\infty\})$ is a field with a non-archimedean valuation. In applications, we will always take $K$ to be a finite extension of $\Q_\ell$ together with its normalised valuation $v$.

Analytic functions on an annulus have a concrete characterisation. In order to carry out explicit computations, we need to work on the ring of analytic functions on \emph{closed} annuli. That is, if $A[\ell^{-v_1}, \ell^{-v_2}]$ is the standard closed annulus centered at 0 with inner and outer radii $\ell^{-v_1}$ and $\ell^{-v_2}$, we work in the ring of power series $ \sum_{i \in \Z} a_i t^i  \in \calO(A)$ where the $a_i$ satisfy \[\lim_{i \to -\infty} v(a_i) + v_1\cdot i = \infty \quad  \text{ and }  \quad \lim_{i \to \infty} v(a_i) + v_2\cdot i = \infty.\]

\begin{definition}
Let $A\subseteq \P^{1,\an}_K$ be a closed annulus centered at $0$ with inner and outer radii $\ell^{-v_1}$ and $\ell^{-v_2}$.

For each $w$ in the closed interval $[v_2,v_1]$ define the valuation $v^w: \calO(A) \to \Q$, given by $v^w(h) = \min{v(a_i) + w\cdot i : i \in \Z}$, where $h = \sum_{i \in \Z} a_i t^i \in \calO(A)$. We let $\mm \subset \calO(A)$ denote the ideal of elements $h$ with $v^{v_1}(h) > 0$ and $v^{v_2}(h) > 0$.
\end{definition}

Most of the time, we are interested in the units of the ring $\calO(A)$, because many of the operations, like taking the inverse and taking a square root, are only applicable to units. To both recognise units, and easily work with them, we provide the following notion of decomposition for units.

By considering the Newton polygon, we find the following proposition characterising the units of $\calO(A)$.
\begin{proposition}\label{prop:decomposition}
	Let $A\subseteq \P^{1,\an}_K$ be a closed annulus centered at $0$ with inner and outer radii $\ell^{-v_1}$ and $\ell^{-v_2}$. Let $t$ be a parameter on $A$ and let $h=\sum_{i} a_it^i\in\calO(A)$.  For $k=1,2$, we define the quantities $$i_k\colonequals \argmin(v(a_i)+iv_k),$$
where  $\argmin(v(a_i)+iv_k)$ is the minimal index that minimises the quantity $v(a_i)+iv_k$.  Then, the following are equivalent.
\begin{enumerate}[(i)]
\item There is a factorisation 
$$h(t)=ct^d(1+g(t)),$$
where $d\in\Z$, $c \in K^\times$, and the valuation of $g(t)$ at the Gauss points of the closed balls $B(0,\ell^{-v_1})$ and $B(0,\ell^{-v_2})$ is positive;
\item the minimum for each $i_1$ and $i_2$ is achieved once, and $i_1=i_2$;
\item $h$ has no zeros on $A$;
\item $h$ is an unit in $\calO(A)$.
\end{enumerate}
Moreover, if these hold, then in the factorisation $d=i_1$, $c=a_d$, and $g(t)=h(t)t^{-d}/c$, and we call $(c,d,g)$ a \emph{decomposition} of $h$. 
\end{proposition}

\begin{proof}
For the equivalence of (i), (iii), and (iv) see \cite[Proposition~2.2]{Baker2013NonArchimedeanAnalyticCurves}. We will show the equivalence of (i) and (ii). This follows from the following observations. If we write
\[
h = c t^{d} (1+ g(t))
\]
for some $c \in K^\times, d \in \Z, g\in \calO(A)$, then $d$ minimises $(v(a_i) + i v_k)$ if and only if the valuation at the Gauss point of the closed ball $B(0,\ell^{-v_k})$ is non-negative, and $d$ is the unique index where this is minimised if and only if the valuation is positive.
\end{proof}

To do practical computations on an annulus $A$, we have to represent elements $h \in \calO(A)$ up to finite precision in some specified way. For rings like $\Q_p$ and $\R$, this is typically done using ball arithmetic: one gives the center $x$ (often taken to be in some countable subring, for example $\Q$), some bound on the error $b$, and this then represents any number $y$ with $|y-x| \leq b$. These balls can be added and multiplied, and one can easily propagate the error bounds. The equivalent notion for $\calO(A)$ has two error bounds, corresponding to the two different valuations $v^{v_1},v^{v_2}$ on $\calO(A)$.
\begin{definition}
A \emph{representation} of an element $h \in \calO(A)$ is a triple of elements $(\tilde{h}, B_1, B_2)$ where $\tilde{h} \in K[t,t^{-1}]$, $B_i \in \R \cup \{\infty \}$ and $v^{v_i}(h - \tilde{h}) \ge B_i$ for $i \in \{1,2\}$. We say $\tilde{h}$ \emph{represents} $h$.
\end{definition}
\begin{remark}
A representation $(\tilde{h}, B_1, B_2)$ represents infinitely many different $f$.
\end{remark}

\begin{remark}
Let $h \in \calO(A)$ be represented by $(\tilde{h}, B_1, B_2)$, where $h$ has coefficients $(a_i)_{i \in \Z}$ and $\tilde{h}$ has coefficients $(\tilde{a}_i)_{i \in \Z}$. Then we see $v(a_i - \tilde{a}_i) \ge \max(-iv_1 + B_1, -iv_2 + B_2)$.
\end{remark}

\begin{lemma}
\label{lem:decomp:approx}
Let $(\tilde{h}, B_1, B_2)$ be a representation for $h \in \calO(A)$. If $\tilde{h}$ has a decomposition $(c, d, \tilde{g})$ and $v^{v_i}(\tilde{h}) < B_i$ for $i = 1,2$, then $h$ has a decomposition $(c,d,g)$ where $(\tilde{g},B_1 - v(c) - v_1\cdot d,B_2 - v(c) - v_2\cdot d)$ represents $g$. If not, then $\tilde{h}$ represents an element which achieves the value $0$ on $A$.
\end{lemma}
\begin{proof}
If $\tilde{h}$ is not a unit (which happens if and only if it does not have a decomposition), then the lemma is clear. Otherwise, let $(c, d, \tilde{g})$ be the decomposition of $\tilde{h}$, and define $g = h/(ct^d) - 1$. Then $(\tilde{g},B_1 - v(c) - v_1\cdot d,B_2 - v(c) - v_2\cdot d)$ represents $g$. If $v^{v_i}(\tilde{h}) < B_1,B_2$, then we see $g$ must lie in $\mm \subset \calO(A)$, and $(c,d,g)$ represents $h$. Otherwise, there is a choice of $g'$ represented by $(\tilde{g},B_1 - v(c) - v_1\cdot d,B_2 - v(c) - v_2\cdot d)$ that achieves the value $-1$, and then $h' = c t^d(1+g')$ is represented by $(\tilde{h}, B_1, B_2)$ and achieves the value $0$.
\end{proof}

In the formula for the global symbol \eqref{eq:globalsymbolhyperellipticexplicit} we need to compute both square roots and integrals while keeping track of propagating error terms. We show how to perform both operations on representations of elements in $\calO(A)$.

\begin{lemma}\label{lem:precsqrt}
Assume $v(2) = 0$. Let $h \in \calO(A)$ by represented by $(\tilde{h}, B_1, B_2)$. Assume $\tilde{h}$ has a decomposition $(c,d,\tilde{g})$ as in \Cref{lem:decomp:approx}. If $d$ is odd or $c$ is not a square, then $h$ has no (inverse) square root. If $d$ is even and $c$ is a square, then $h^{-\frac12}$ is represented by
\[
((c^{-1/2} x^{d/2} \sum_{i = 0}^{k} \binom{-\frac12}{i} \tilde{g}^i, \min (v(c^{-1/2}) + d v_1 + (k+1) v^{v_1}(\tilde{g}),  \frac12B_1), \min (v(c^{-1/2}) + d v_2 + (k+1) v^{v_2}(\tilde{g}),  \frac12B_2))
\]
for any $k \in \Z_{\geq 1}$.
\end{lemma}
\begin{proof}
This follows immediately from the power series expansion of $(1+x)^{-\frac12}$, and the fact that $\binom{-\frac{1}{2}}{i}$ has valuation at least $0$.
\end{proof}

\begin{lemma}
Let $h \in \calO(A)$ be a function given by a representation $(\tilde{h}, B_1, B_2)$, with the coefficient of $a_{-1}$ being $0$ for $h$ and $\tilde{h}$. Let $A'$ be an annulus with inner and outer radii $\ell^{-v_1 + \varepsilon}$ and $\ell^{-v_2 - \varepsilon}$. Let $\delta(\varepsilon)$ denote the minimal value of $\varepsilon \cdot i - v(i)$ for $i \in \Z_{>0}$. Then the integral of $h$ on $A'$ is well-defined and represented on $A'$ by 
$(\int \tilde{h}, B_1', B_2')$, where \[B_1' = \min(B_1 + v_1 - \delta(\varepsilon), B_2 + v_2  - \delta(v_1 - \varepsilon -v_2))\] and \[B_2' = \min(B_2 + v_2 - \delta(\varepsilon), B_1 + v_1  - \delta(v_1 - \varepsilon -v_2)).\]
\end{lemma}
\begin{proof}
We present a sample computation, showing that $B_1' \geq B_1 + v_1 - \delta(\varepsilon)$. Assume for simplicity that $h = a t^i$ with $i < -1$ and $\tilde{h} = b t^i$, with $v(a-b) + i v_1 \geq B_1$. Then
\begin{align*}
v\left(\frac{1}{i+1} (a-b)\right) + (i+1) (v_1 - \varepsilon) &\geq B_1 + v_1 - v_p(i+1) - \varepsilon(i+1) \\
&\geq B_1 + v_1 - \delta(\varepsilon).
\end{align*} The cases where $i \geq 0$, or we bound $B_2$ instead of $B_1$ proceed similarly.
\end{proof}

\subsection{Action of the correspondence on de Rham cohomology}
In order to compute  $Z_*$ on  $\rH_1(\Gamma,\Z)$ and $\tr_v(Z)$, we also need a matrix for the action of $Z_*$ on $\rH^1_\dR(X/K)$. We now describe a method to compute such a matrix with rational coefficients for hyperelliptic curves. 
This method is entirely new, and solves the problem of how to determine the matrix of an endomorphism on $\rH^1_{\dR}(X/K)$ precisely. Previous methods for quadratic Chabauty use the Hecke operator $Z = T_p$ as the endomorphism and employ Eichler--Shimura to determine the action of $T_p$ on de Rham cohomology to arbitrary $p$-adic precision \cite[\S3.5.2]{examplesandalg}. These methods work for any modular curve (not necessarily hyperelliptic), but can not compute the action exactly.
In order to determine the local heights at $\ell \neq p$ as a rational multiple  $r \chi(\varpi_K)$ for $r \in \Q$ (as in Section \ref{sec:localheights}), we do need to know the action of $Z_*$ exactly. (Otherwise we can only determine $r$ to some $\ell$-adic precision.) To solve this problem, we produce a method to rigorously compute the integer matrix for the action of $Z_*$.

\begin{proposition}\label{prop:pushpull}
Let $\pi_1$ and $\pi_2$ be the projections from $X \times X \to X$. Let $\omega$ be a 1-form on $X$. Let $Z \subset X \times X$ be a correspondence. Let $\omega_0$ be a fixed non-zero $1$-form.  Define $g \colonequals  (\pi_1^* \omega/ \pi_2^*\omega_0  )$; this is a rational function on $Z$. Then
\[Z_* \omega = \pi_{2,*}(\pi_1^*\omega) = (\pi_{2, *} g ) \omega_0  .\]
Here, the pushforward $\pi_{2,*} g$ denotes the trace of $g$ in the extension $K(Z)/K(X)$ of function fields.
\end{proposition}

\begin{proof}
By the projection formula \cite[Theorem~7.5]{HartshorneDeRham} \[\pi_{2,*}(\pi_1^*\omega) = \pi_{2,*}(g \cdot \pi_2^*\omega_0) = \pi_{2,*}(g) \cdot \omega_0.\]
\end{proof}

For hyperelliptic curves, we choose $\omega_i$ for $i=\{0,\ldots,2g-1\}$ to be a basis for $\rH^1_\dR(X/K)$ constructed by taking linear combinations of $x^i \rd x/2y$ for $i = 0, \dots, 2g$. We use \Cref{prop:pushpull} to compute $Z_\ast(\omega_i)$ for all $i=1,\ldots,2g-1$.  The last step is to rewrite these elements in terms of the basis we chose.  

Note that the class of any differential form $g(x,y)\rd x\in\rH^1_\dR(X/K)$, where $g(x,y)\in K(X)$, is equivalent in cohomology to the class of
$2y(g(x,y)-g(x,-y))\frac{\rd x}{2y},$
where $2y(g(x,y)-g(x,-y))$ is a rational function only in $x$.  So now we reduce to the case of forms $g(x)\frac{\rd x}{2y}$, where $g(x)\in K(x)$.  Then, a linear algebra argument allows us to write $g(x)\frac{\rd x}{2y}$ in terms of the basis $\{\omega_i\}_{i=0}^{2g-1}$.

\subsection{Action of the correspondence on graph homology and traces}
We now show how to compute the action of $Z_*$ on $\rH_1(\Gamma,\Z)$ and $\tr_v(Z)$ using the integer matrix describing the action of $Z_*$ on $\rH^1_\dR(X/K)$. In this section the curve $X$ is not necessarily hyperelliptic.

\begin{proposition}\label{prop:elladicmatrix}
Let $e_1,\dots,e_N$ be oriented edges of $\Gamma$ that form the complement of a spanning tree.

Let $T$ be the matrix with $$T_{ij} = \Res_{A_{e_i}}\omega_j.$$
This has a right inverse $T^{-1}$.
Let $M$ be the matrix of $Z_*$ acting on $\rH^1_\dR(X/K)$.

Then the matrix of $Z_*$ on $\rH_1(\Gamma,K)$ is $ T M T^{-1}.$
\end{proposition}{
\begin{proof}
We can identify $\rH_1(\Gamma, K)$ with $K^{\{e_1, \dots, e_N\}}$; the contraction of the complementary spanning tree in the dual graph is a bouquet of $N$ circles.
Note that $T$ is the matrix of the $K$-linear morphism $\phi_2$ sending $\omega_j \mapsto \sum_{i} \Res_{A_{e_i}} \omega_j$ under this identification. By \Cref{prop:coleman-iovita_functoriality}, the pushforward $Z_*$ is preserved under $\phi_2$. The matrix for $Z_*$ acting on $\rH_1(\Gamma, K)$ is then given by conjugating $M$ by $T$. 
\end{proof}

Using \Cref{thm:bounds}, we can bound the operator norm of the matrix with entries in $K$ obtained from \Cref{prop:elladicmatrix}. If the matrix $T$ is computed with enough $\ell$-adic precision, this allows us to obtain a matrix with entries in $\Z$. In the rest of this section we determine precise bounds. We start by bounding the size of the integers appearing in the matrix.

\begin{remark}
By \Cref{sub:finpreccalculus} and \Cref{sec:ExplicitColemanIovitaHyp}, in particular \eqref{eq:globalsymbolhyperellipticexplicit}, we can compute $T$ up to arbitrary $\ell$-adic precision for hyperelliptic curves.
\end{remark}

\begin{proposition}
\label{prop:boundintegermatrix}
Let $Z \subset X \times X$ be an effective correspondence of degrees $d_1$ and $d_2$ over $X$, respectively.
Let $M$ be the integer matrix representing $Z_*$ on $\rH^1(\Gamma, K)$.
Let $C$ denote the matrix of the intersection length pairing on homology. 
Write $C = P P^T$ for some invertible matrix $P$. Let $\|\cdot\|_{\max}$ be the norm given by the maximum of the absolute values of the entries. Then
we have the bound \[\|M\|_{\max} \leq \frac{{2g}^{2}}{\|P\|_{\max} \|P^{-1}\|_{\max}} \sqrt{d_1 d_2}.\]
\end{proposition}
\begin{proof}
By \Cref{thm:bounds}, the operator norm of $Z_*$ with respect to this intersection pairing $\|Z\|_L$ is bounded by $\sqrt{d_1 d_2}$ where $d_i$ is the degree of the projection $\pi_i: Z \to X$. Since $C$ is a positive definite symmetric matrix, we can write $C = P P^T$ for some invertible matrix $P$.
This allows us to give a bound on the entries of  $M$. It can be checked by evaluating on standard basis vectors that $\|\cdot\|^2_I \geq \|\cdot\|^2_{\max}$ where $I$ is the identity matrix (and $\|\cdot\|_{\max}$ is the maximum over the absolute value of the entries). Therefore
\begin{equation*}
\|Z\|_L = \|PZ P^{-1}\|_I  \geq \|P Z P^{-1}\|_{\max} = \|P M P^{-1}\|_{\max}.
\end{equation*}
We obtain the bound $\|M\|_{\max} \leq \frac{{2g}^{2}}{\|P\|_{\max} \|P^{-1}\|_{\max}} \sqrt{d_1 d_2}$.
\end{proof}

\begin{corollary}
\label{cor:padicapproximationcorrect}
Let $M, C, P$ be as above. Let $c \colonequals \frac{{2g}^{2}}{\|P\|_{\max} \|P^{-1}\|_{\max}} \sqrt{d_1 d_2}$.   Suppose that $M'$ is an integer matrix such that $\|M'\|_{\max} \leq c$ and $c + \|M'\|_{\max} < \ell^k$, where $k \geq 0$ is an integer such that $v_\ell(M'_{ij} - M_{ij}) \geq k$. Then $M = M'$.
\end{corollary}

Using the techniques for working with $\calO(A)$ developed in this section, one can compute $M$ up to arbitrary $p$-adic precision, and \Cref{cor:padicapproximationcorrect} allows us to determine when we have computed enough to uniquely determine $M$. An example of this is worked out in \S\ref{subsec:shimuracurve}.

We end by discussing how to compute the trace of the endomorphism $Z$ acting on the de Rham cohomology of a single vertex once we have determined $M$.

\begin{proposition}
\label{prop:traceblocks}
Let $v \in V(\Gamma)$. Let $\omega'_0, \dots, \omega_d'$ be a basis for $\ker(\phi_2)$. Choose a basis $\eta_1, \dots, \eta_{g(v)}$ of $\rH^1_{\dR}(\bar{\calX}_{v}, K)$. For each $\omega_i'$, define $s_{ij} \in K$ such that $ \omega'_i =  \sum_{j=0}^{2 g(v)} s_{ij} \eta_j$. 
Let $S_{v}$ be the matrix with entries $S_{ij } =  s_{ij}$. Let $S$ be the block matrix consisting of all $S_{v}$, representing the map $\phi_1\colon \ker \phi_2 \to \bigoplus_v\rH^1_\rig(\bar\dX_v/K)$. This has a right inverse $S^{-1}$. Let $M$ be the matrix of $Z_*$ acting on $\ker \phi_2 \subset \rH^1_\dR(X/K)$. 
Then $SMS^{-1}$ is the matrix representing the induced action of $Z_*$ on $\bigoplus_v \rH^1_\rig(\bar\dX_v/K)$, and $\tr(Z_v)$ is the trace of the diagonal block of $S M S^{-1}$ corresponding to $\rH^1_{\dR}(\bar{\calX}_{v}, K) \cong \rH^1_{\rig}(\bar\dX_{v}/K)$.
\end{proposition}

\begin{proof}
Note that $S$ is the matrix of the $K$-linear morphism $\phi_1$. By \Cref{prop:coleman-iovita_functoriality}, the pushforward $Z_*$ is preserved under $\phi_1$. The matrix for $Z_*$ acting on $\rH_1(\Gamma, K)$ is then given by conjugating $M$ by $S$. 
\end{proof}

For hyperelliptic curves, we can again compute these up to arbitrary $\ell$-adic precision using \eqref{eq:omegaeta} and the chosen bases of $\rH^1_{\dR}(\bar{\calX}_{v}, K)$ to compute the $s_{ij}$, and then use the bounds of \Cref{thm:bounds} to obtain a precise answer.


\section{Examples}\label{sec:examples}

In this section we work through some illustrative examples of our method.

\subsection{Bielliptic curve}\label{subsec:BiellipticExample} As a warm-up, let $q$ be an odd prime number and let $\alpha$ and $\beta$ be positive integers.  We consider the genus 2 bielliptic curve with affine equation
\begin{equation}\label{eq:bielliptic_example}
	X:y^2 = (x^2-q^\alpha)((x-1)^2-q^\beta)((x+1)^2-q^\beta).
\end{equation}
This curve has bad reduction at $q$.  We will compute the local height contribution at $q$ using the formula from \Cref{thm:local_heights_formula} and via intersection theory.  The bielliptic involution $\phi : x \mapsto -x$ gives rise to a trace 0 endomorphism of the Jacobian of $X$; we use the correspondence $Z$ arising from this endomorphism.

First, we use the method described in \S\ref{subsec:localHeightFormula}.  The cluster picture for this curve is presented in \Cref{fig:clusterBielliptic}.
\begin{figure}
\includegraphics[width =4cm]{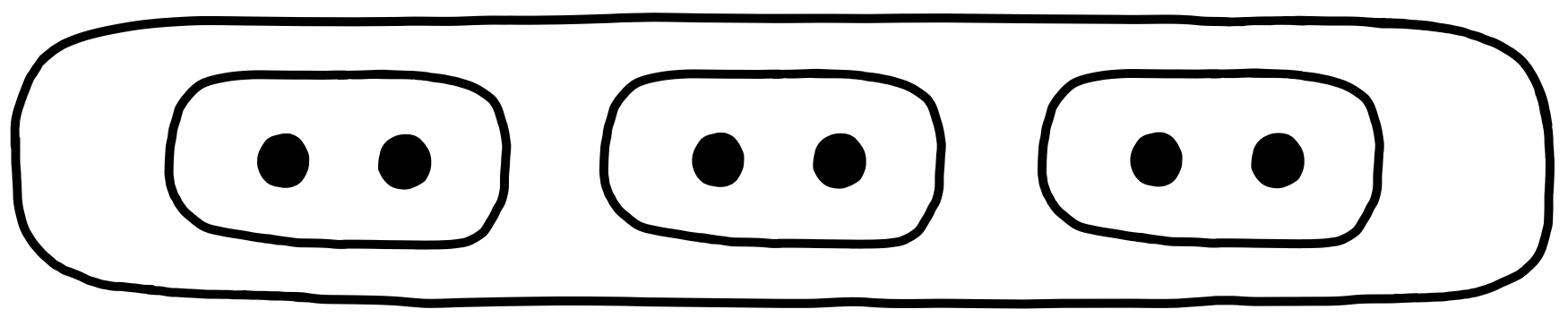}
\caption{Cluster picture for example in \S\ref{subsec:BiellipticExample}}\label{fig:clusterBielliptic}
\end{figure}
The reduction graph associated to the semistable cover is given in \Cref{fig:reductionbielliptic}
\begin{figure}
\begin{center}
\includegraphics[width=3cm]{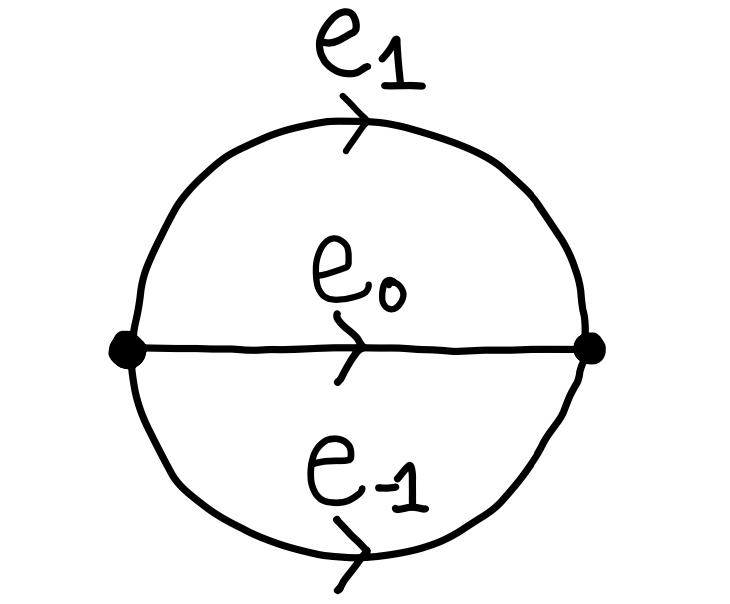}
\end{center}
\caption{Reduction graph for bielliptic example}
\label{fig:reductionbielliptic}
\end{figure}
where $l(e_0)=\alpha$ and $l(e_1)=l(e_{-1})=\beta$.
The bielliptic involution defines the action $Z_*(e_0)=-e_0$,  $Z_\ast(e_1)=-e_{-1}$,  and $Z_\ast(e_{-1})=-e_1$. 

We fix the midpoint of the edge $e_0$ as our base point.  To apply the formula from \Cref{thm:local_heights_formula}, we first compute the orthogonal projections of the edges in the first homology of the dual graph.  For that, we note that the pairing on the span of the edges is given by
\begin{equation*}
\langle e_i,e_j\rangle = \begin{cases}l(e_i) &\text{if }i=j;\\
0 &\text{otherwise.}\end{cases}
\end{equation*}
Thus, the orthogonal projections are
\begin{align*}
\pi(e_0) = \frac{\alpha}{2\alpha+\beta}(2e_0-e_1-e_{-1});& \quad 
\pi(e_1) = \frac{1}{2\alpha+\beta}((\alpha+\beta)e_1-\beta e_0-\alpha e_{-1}); \\ 
\pi(e_{-1}) &= \frac{1}{2\alpha+\beta}((\alpha+\beta)e_{-1}-\beta e_0-\alpha e_{1}).
\end{align*}
Then, we compute
\begin{align*}
	\langle e_0,Z_\ast \pi(e_0)\rangle=-\frac{2\alpha^2}{2\alpha+\beta}; \quad \quad
	\langle e_1,Z_\ast \pi(e_1)\rangle=\langle e_{-1},Z_\ast \pi(e_{-1})\rangle=\frac{\alpha\beta}{2\alpha+\beta}
\end{align*}
to obtain the measure
\begin{equation*}
	\mu_Z=-\frac{4}{2\alpha+\beta}|\rd s_0|+\frac{2\alpha}{\beta (2\alpha+\beta)}|\rd s_1|+\frac{2\alpha}{\beta(2\alpha+\beta)}|\rd s_{-1}|,
\end{equation*}
where $|ds_i|$ is the arc-length measure along the edge $e_i$ for $i=-1,0,1$.  We conclude that $\tilde{h}_{Z,q}$ is given by
$$\begin{cases}
-\frac{2}{2\alpha+\beta}s_{e_0}^2 +a_0s_{e_0}+b_0&\text{ on }e_0,\text{ with }0\le s_{e_1}\le \alpha;\\
\frac{\alpha}{\beta(2\alpha+\beta)}s_{e_1}^2+a_1s_{e_1}+b_1 &\text{ on }e_1,\text{ with }0\le s_{e_1}\le \beta;\\
\frac{\alpha}{\beta(2\alpha+\beta)}s_{e_{-1}}^2+a_{-1}s_{e_{-1}}+b_{-1} &\text{ on }e_{-1},\text{ with }0\le s_{e_{-1}}\le \beta.
\end{cases}$$
with constants $a_i$ and $b_i$ for $i=-1,0,1$.  To find the exact values of the constants, we use that the polynomials must agree at endpoints, must vanish at the chosen base point, and that $\Delta^2(\tilde{h}_{Z,q}) - \mu_{Z}= 0$ imposes a condition on the derivatives.  This yields the local height function
\begin{equation}
\label{eqn:localheightsbielliptic}
\begin{cases}
-\frac{2}{2\alpha+\beta}\left(\frac{\alpha}{2}-s_{e_0}\right)^2 &\text{ on }e_0,\text{ with }0\le s_{e_0}\le \alpha; \\
\frac{\alpha}{\beta(2\alpha+\beta)}\left(\frac\beta2-s_{e_1}\right)^2-\frac\alpha4 &\text{ on }e_1,\text{ with }0\le s_{e_1}\le \beta;  \\ 
\frac{\alpha}{\beta(2\alpha+\beta)}\left(\frac{\beta}{2}-s_{e_{-1}}\right)^2-\frac{\alpha}{4} &\text{ on }e_{-1},\text{ with }0\le s_{e_{-1}}\le \beta. 
\end{cases}
\end{equation}

Now we compute local heights via intersection theory: the local height at $q$ for a prime of bad reduction is defined in terms of intersection theory on the special fibre of a regular model $X_{\Z_q}$ as in the beginning of \Cref{sec:localheights}.
Let $\XX/\Z_q$ be the minimal regular model of $Y$.  Then the following are true.
\begin{itemize}
	\item The special fibre $\XX_{\F_q}$ consists of two copies of $\P^1_{\F_q}$ connected by three chains of projective lines $\P^1_{\F_q}$ of lengths $\alpha-1$, $\beta-1$, and $\beta-1$, respectively.
	\item The action of $\phi$ on $\XX_{\F_q}$ interchanges the two distinguished copies of $\P^1_{\F_q}$, reverses the chain of length $\alpha-1$, and reverses and interchanges the two chains of length $\beta-1$.
\end{itemize}

Let $Z = \Gamma_\phi$ be the graph of $\phi$. Let $P\in \XX(\Z_q)$ be any point.
Following  \eqref{eq:DZ}, let
$$D_Z(b, P)= X^\phi-\phi(P)-\phi(b)$$ be the divisor on the generic fibre, where $X^\phi$ is the set of fixed points under $\phi$. 
We denote its closure in $\XX$ also by
$$D_Z(b, P)= \XX^\phi-\phi(P)-\phi(b).$$

Our aim is then to find the vertical divisor $V$ such that $\calD_Z(b,P) \colonequals D_Z(b, P) + V$ has intersection number $0$ with every vertical divisor.  By normalising $V$ so that the coefficient of the component containing $\bar{b}$ is zero, then the local height $\tilde{h}_{Z,q}(P)$ is the coefficient of the component of $V$ containing $\bar{P}$.

Write  the component of $\XX_{\F_q}$ containing the point $\bar{P}$ as $[\bar{P}]$. Equivalently, without normalising, the height is given by 
\[ (\text{coefficient of } [\bar{P}] -  \text{coefficient of }[\bar{b}]).\]

Write $L$ for the matrix of the intersection pairing on components of the special fibre. We want a $\Q$-linear combination $V$ of components such that \[ L \cdot V = D(\bar{P}, \bar{b}) = [ \phi(\bar{P})] - [\bar{b}]\] (using that $\XX^\phi$ and $\phi(b)$ both reduce to the same component as $b$). Therefore
\begin{align}
\label{eqn:htpairingdeg0divisors}
\tilde{h}_{Z,q}(P) =  ([\bar{P}] - [\bar{b}]) \cdot L^{+} ( [ \phi(\bar{P})] - [ \bar{b}]).
\end{align}
(One should be careful since $L$ is not an invertible matrix. However $[ \phi(\bar{P})] - [ \bar{b}]$ is in its image, and the local height is independent of the choice of preimage.)

Now $L$ is also (by definition) the Laplacian matrix of the dual graph of $\XX_{\F_q}$, and we recognise \eqref{eqn:htpairingdeg0divisors} as the formula for the height pairing on degree $0$ divisors on a graph. So 
\begin{align}
\label{eqn:intersectionpairinggraphs}
\tilde{h}_{Z,q}(P) =   \langle [\bar{P}] - [\bar{b}], [ \phi(\bar{P})] - [ \bar{b}] \rangle.
\end{align}
Via the theory of electrical circuits on graphs, we can compute \eqref{eqn:intersectionpairinggraphs} explicitly, breaking it into cases depending on the component of the reduction of $P$. We obtain local heights agreeing with \eqref{eqn:localheightsbielliptic}.

We work though one explicit example when $\alpha = 2$ and $\beta = 1$. By blowing up at the singular point $x = y = q = 0$, we compute a regular model $\mathscr{X}/\Z_q$, given by the equations 
\begin{align*} 
&y^2 = (x^2 - q^2)((x-1)^2-q)((x+1)^2-q),\\
&Y^2 = (X^2 - Q^2)((x-1)^2 - q)((x+1)^2-q),  \\
&xY = Xy, \quad
yQ = Yq, \quad
qX = Qx
\end{align*}
in $\A^2_{\Z_q} \times \P^2_{\Z_q}$, which maps to the original curve by $(x,y;X,Y,Q) \mapsto (x,y)$.
\begin{figure}
\begin{center}
\includegraphics[width=5cm]{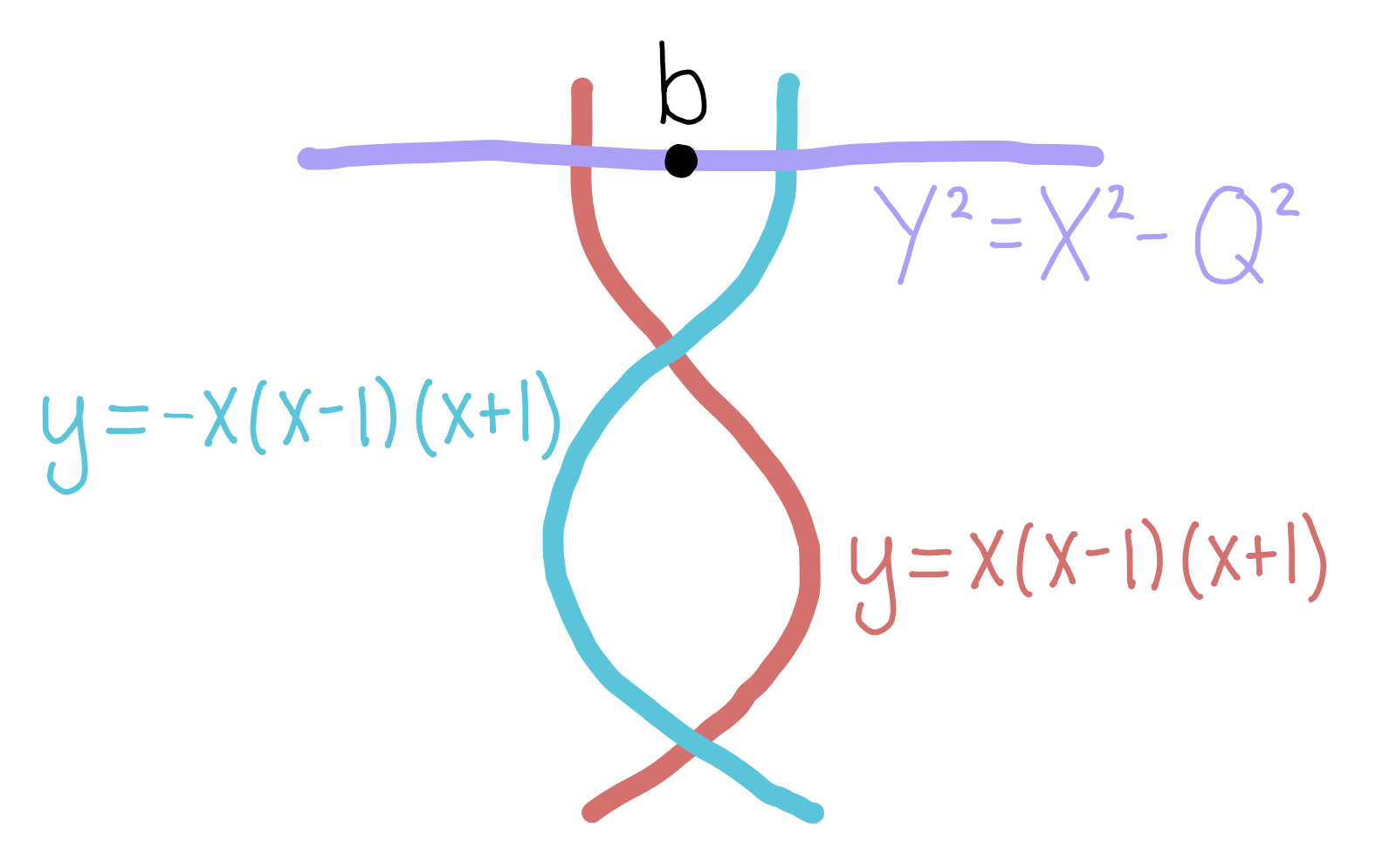}
\end{center}
\caption{Special fibre for bielliptic example when $\alpha = 2$, $\beta = 1$}
\label{fig:biellipticintersection}
\end{figure}
The special fibre $\mathscr{X}_{\F_q}$ is depicted in \Cref{fig:biellipticintersection}.
By counting intersections of components, we see that
\[L = \left[ 
\begin{matrix}
-3 & 2& 1 \\
2 & -3 & 1\\
1 & 1& -2
\end{matrix}
\right]
\]
where the basis of $L$ is the components in the order blue (left), pink (right), purple (top).

The graph $\Gamma_\phi \subset (\A^2 \times \P^2)^2 $ is given by the explicit equations $\{y= y', x = -x', X  = -X', Y = Y', Q = Q' \}$. 
Then $\mathscr{X}^\phi = \{ (0, iq(1-q); 0 , i(1-q),), (0, -iq(1-q); 0 , -i(1-q),1)\}$ are the fixed points under $\phi$.  This divisor reduces to two points, lying on the the purple (top) component.
We choose a base point $b = (0, q; 0,1,1)$ reducing to the midpoint of the purple (top) component $Y = X^2 - Q^2$.

For any point $P$, we can now calculate the local height using intersection theory. For example, let $P$ be a point reducing to the blue (left) component. Then $D_Z(b, P) = \mathscr{X}^\phi -\phi(P) - \phi(b)  $ has multidegree $(0, -1, +1)$ on the components. We now want to add to $D_Z(b,P)$ a vertical divisor $V$ with no components containing $b$ in order to make $D_Z(b, P)$ multidegree $0$. Using linear algebra with $L$, we see $V$ is $-3/5$ times the pink (right) component plus $-2/5$ times the blue (left) component, so the local height is $-2/5$, agreeing with \eqref{eqn:localheightsbielliptic}.

\begin{remark}
This calculation of heights as an intersection multiplicity can be extended to work for any value of $\alpha$ and $\beta$ and again gives the same values as in \eqref{eqn:localheightsbielliptic}. As such, this gives an explicit check that the heights defined in terms of intersection multiplicities agree with those coming from \Cref{thm:local_heights_formula}.
\end{remark}

\begin{remark}
	Quadratic Chabauty for bielliptic genus~$2$ curves was studied by Bianchi and Padurariu \cite{Bianchi(bi)elliptic,OanaFrancesca}, building on work of Balakrishnan and Dogra \cite{QCI}. If~$X$ is a bielliptic genus~$2$ curve with equation $y^2=a_6x^6+a_4x^4+a_2x^2+a_0$, then the local height function appearing in these methods is the function $X(\bQ_q)\to\bQ_p$ given by
	\[
	h_q(z) = \lambda_q(\varphi_1(z)) - \lambda_q(\varphi_2(z)) - 2\log|x(z)|_q \,,
	\]
	where~$\varphi_1$ and~$\varphi_2$ are the projections from~$X$ to the two elliptic factors of $\Jac(X)$ explicitly described in \cite[\S2]{OanaFrancesca}, $\lambda_q$ is the local N\'eron function of an elliptic curve, and~$\log$ denotes the $p$-adic logarithm. This local height function is related to height functions in the sense of our \S\ref{sec:localheights} by \cite[Lemma~7.7]{QCI}, which says that for any base point~$b$ we have
	\[
	h_q(z) - h_q(b) = \bigl((z-b)\cdot\dD_{Z-Z'}(b,z)\bigr)\log(q) = \log(q)\tilde h_{Z-Z',q}(z) \,,
	\]
	where~$Z$ is the graph of the bielliptic involution~$\phi\colon(x,y)\mapsto(-x,y)$, and~$Z'$ is the graph of the other bielliptic involution~$\phi'\colon(x,y)\mapsto(-x,-y)$. Since~$\phi'$ is the composition of~$\phi$ with the hyperelliptic involution, which acts on the Jacobian by~$[-1]$, we have~$\tilde h_{Z-Z',q}=2\tilde h_{Z,q}$.
	
	In the particular case of the curve~\eqref{eq:bielliptic_example}, the relevant basepoint~$b$ is one of the two points with $x$-coordinate~$0$. We can evaluate the local height function~$h_q$ at~$b$ as follows. Firstly, $\lambda_q(\varphi_1(b))=0$ by \cite[Proposition~2.2(a)]{OanaFrancesca}. Secondly, while both~$\lambda_q(\varphi_2(z))$ and $\log|x(z)|_q$ have a logarithmic pole at~$z=b$, the leading terms of those poles in $-\lambda_q(\varphi_2(z))-2\log|x(z)|_q$ exactly cancel out and we have $-\lambda_q(\varphi_2(z))-2\log|x(z)|_q=-\log|a_0|_q=\alpha\log(q)$ by the same calculation as in the proof of \cite[Theorem~2.3(c)]{OanaFrancesca}. So we have~$h_q(b)=\alpha$. Putting everything together, we find that the local height used in \cite{QCI,Bianchi(bi)elliptic,OanaFrancesca} is related to ours by
	\[
	h_q(z) = \bigl(\alpha + 2\tilde h_{Z,q}(z)\bigr)\cdot\log(q) \,.
	\]
	
	One other point deserves further remark. Namely, in the version of quadratic Chabauty used in \cite{OanaFrancesca}, one computes a set~$\Omega_q$ which is a superset of the set of values attained by~$-h_q(z)$ for~$z\in X(\bQ_q)$. On the other hand, our methods allow one to compute the exact set of values, not just a superset. The difference can be quite stark: for example, for the curve~\eqref{eq:bielliptic_example} with~$q=17$, $\alpha=3$ and~$\beta=4$, the superset has size~$38$, whereas the set of actual values has size~$4$. Thus our local heights algorithm offers a potentially significant time saving, even in the well-studied setting of quadratic Chabauty for bielliptic curves.
\end{remark}

\subsection{A Shimura curve quotient}
\label{subsec:shimuracurve}

In \cite{GuoYang} the authors determine equations for all geometrically hyperelliptic Shimura curves $X_0(D,N)$ and Atkin--Lehner operators on these curves. Since $X_0(D,N)(\R) = \emptyset$, for the study of rational points, it makes sense to consider Atkin--Lehner quotients of these curves. By applying Jaquet--Langlands and considering the sign of the $L$-function of the corresponding modular forms space, we see there are only two quotient curves whose Mordell--Weil rank are equal to their genus. 

One is bielliptic, so its local heights at primes of bad reduction can be computed from the elliptic curve factors of its Jacobian. The other is the curve $X_0(93, 1)/\langle \omega_{93} \rangle $.  The rational points on this curve correspond to abelian surfaces defined over $\Q$ with potential quaternionic multiplication by a quaternion algebra of discriminant $93$.
It turns out that this curve has an exceptional isomorphism to the modular curve quotient $X_0(93)^*$ whose rational points were previously determined in \cite{BarsGonzalezXarles} using elliptic Chabauty. We now show how to compute the local heights at primes of bad reduction for $X \colonequals X_0(93, 1)/\langle \omega_{93} \rangle$.

The conductor of $X$ is $3^2 \cdot 31^2$. At $\ell = 31$, the curve $X: y^2 = x^6 + 2x^4 + 6x^3 + 5x^2 - 6x + 1$ has the cluster picture in \Cref{fig:cluster022_12_12_0} and so by \Cref{thm:reductiontocomponents} the local heights at $31$ are trivial.

\begin{figure}
\begin{center}
\includegraphics[width=4cm]{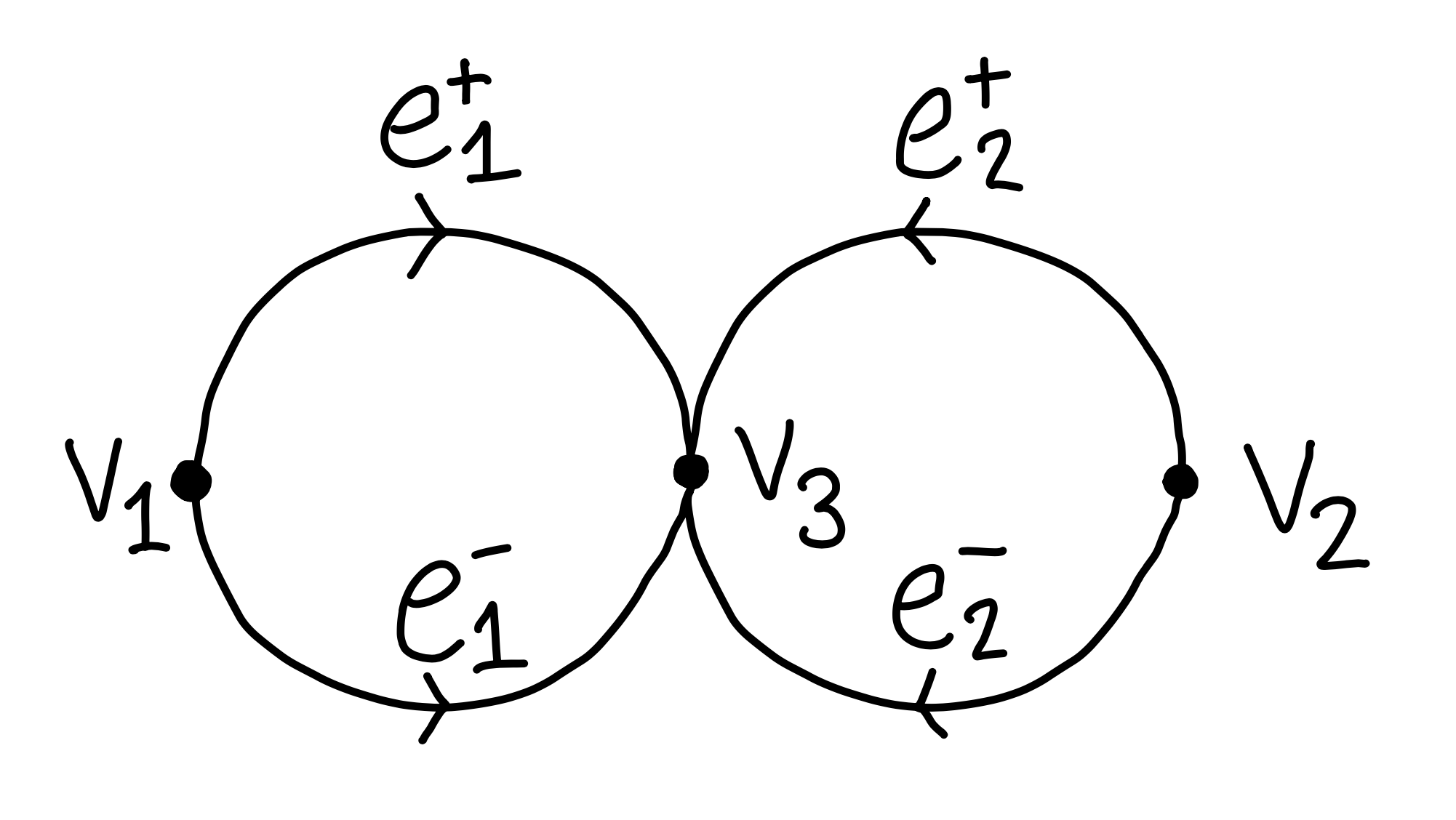}
\end{center}
\caption{Reduction graph for Shimura curve quotient}
\label{fig:graphshimura}
\end{figure}
At $\ell= 3$ the cluster picture is the cluster in \Cref{fig:cluster022} so more analysis is needed; the reduction graph $\Gamma$ associated to the semistable covering contains three vertices and $\Q_\ell$-points reduce to all three curves associated to the vertices. See \Cref{fig:graphshimura}.
Let $\tt_1$ be the twin cluster given by the two roots of $f$ defined over $\Q_3$, and $\tt_2$ be the other twin cluster. These correspond to vertices $v_1$ and $v_2$ respectively.  We label the edges $v_1$ to $v_3$ by $e_1^{\pm}$ and the edges $v_2$ to $v_3$ by $ e_2^{\pm}$. We have $l(e_i^\pm) = 1$ for all $i$.
\begin{figure}[h!]
\begin{center}
\includegraphics[width=4cm]{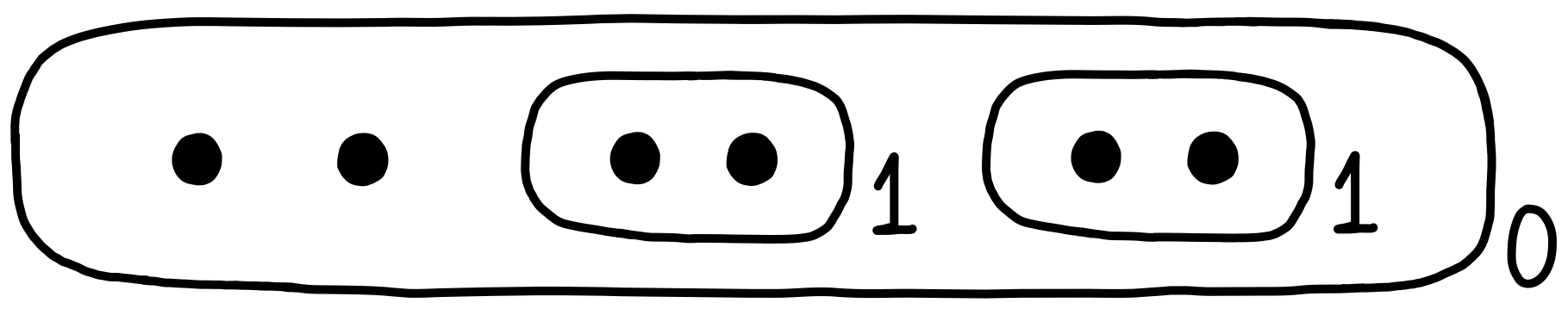}
\end{center}
\caption{Cluster picture}
\label{fig:cluster022}
\end{figure}

Let $Z$ be the (trace zero) endomorphism acting by $\sqrt{5}$ on the Jacobian. This acts by pushforward on the holomorphic differentials $\langle \rd x / (2y), x\rd x / (2y) \rangle$ by \[M = \left[\begin{matrix} -1 & 2 \\ 2 & 1 \end{matrix} \right].\]
By \cite[Theorem~5.1]{UsersGuide}, since $\nu_\ss=0$ when $\ss$ is the top cluster and inertia fixes the twins setwise, $X$ has a semistable model over $\Q_3$ and a split semistable model over $\Q_9$, the unramified quadratic extension of $\Q_3$.

For each $\tt_i$, we compute the map 
$A_{\tt_i}^+ \to X|_{A_{\tt_i}}$ given by sending $x \mapsto (x, g_{\tt_i}^{1/2} h_{\tt_i}^{1/2})$  from \eqref{eqn:Xannuluseven}. 
 We compute the matrix of $\phi_2: x^j\rd x/(2y) \mapsto \Res_{A_{\tt_i}^+} x^j \rd x/(2y)$ for $i = 1,2$ and $j = 0,1$ to be 
\[ \left[ \begin{matrix}   20a +20+ O(3^4)& 2a + 2+ O(3^4)\\ 
2a + 2 + O(3^4) & 22a+22 + O(3^{4})  \end{matrix} \right] \]
where $a$ has minimal polynomial $x^2 + 2x + 2$ over $\Q_3$. Using this change of basis, we conjugate $M$ to obtain the action of $Z_*$ on $H_1(\Gamma, \Q_9)$ by \Cref{prop:elladicmatrix} using the complement $e_1^{-}, e_2^{-}$ of a spanning tree
\begin{equation}\label{eqn:matrixZsqrt5}
	\left[ \begin{matrix}
-1 + O(3^{4}) & 2 + O(3^{4})\\
2 + O(3^{4}) & 1  + O(3^{4})
\end{matrix} \right].
\end{equation}
We now use \Cref{prop:boundintegermatrix} and \Cref{cor:padicapproximationcorrect} to recognise the integer matrix that we approximated.  Using \cite{RigorousEndo} we explicitly construct a correspondence $Z \subset X \times X$ for $M$. We compute that $Z$ has degrees $d_1 = 2$ and $d_2 = 10$ over $X$. We fix the (orthogonal) homology basis $\langle e_1^- - e_1^+, e_2^- - e_2^+\rangle$. 
The intersection pairing on homology has the matrix $C = \left[ \begin{smallmatrix}2 & 0 \\ 0 & 2 \end{smallmatrix} \right]$ in this basis. Writing $C = P P^T$, we see $\|P\|_{\max} = \sqrt{2}$ and $\|P^{-1}\|_{\max} = 1/\sqrt{2}$. Therefore the action of $Z_*$ on homology is represented by an integer matrix whose entries have absolute value at most $16 \sqrt{5}$. The entries of the matrix in \eqref{eqn:matrixZsqrt5} are all integers inside this range, plus a multiple of $3^{4}$. Since $2<16\sqrt{5}$ and $3^{4} > 16 \sqrt{5}+2$, we conclude that the action of $Z_*$ on the homology of the graph is given by the integer matrix
$ \left[ \begin{smallmatrix}
-1 & 2\\
2 & 1
\end{smallmatrix} \right].$

The projection of $e_1^{\pm}$ onto the basis of homology has coefficients $[\mp 1/2,0]$, and the projection of $e_2^{\pm}$ onto the basis of homology has coefficients $[0,\mp 1/2]$.  
This computation yields the Laplacian $\nabla^2(\tilde{h}_{Z,\ell})$ of the piecewise polynomial height function as in \Cref{thm:local_heights_formula} 
\[\mu_Z =  1 \cdot \rd s_{e_1^-}+  1\cdot \rd s_{e_1^+}- 1 \cdot \rd s_{e_2^-} -  1 \cdot \rd s_{e_2^+}.\]
Let $f_0$ be the piecewise polynomial function $ (-1)^i \frac{1}{2}s_{e_i^{\pm}}(s_{e_i^{\pm}}- 1)$, obtained from double integrating each part of $\mu_Z$. Then $f_0$ is 0 at each vertex and $\mu_Z - \nabla^2 (f_0) = 1 \cdot \delta_{v_1} - 1 \cdot \delta_{v_2} + 0 \cdot \delta_{v_3}$.  The weighted Laplacian matrix $L$ of $\Gamma$ is 
\[L \colonequals \left[ \begin{matrix}
	-2 &  0 &2\\
	0 & -2 &2\\
	  2 & 2 &-4 \\
   \end{matrix}\right].\]
Using $L$ we solve for the coefficient vector of a function whose Laplacian is the vector of the weights of the $\delta_{v_i}$.
\[ L \left[ \begin{matrix} -1 \\ 0  \\ -1/2  \end{matrix} \right] = 
\left[ \begin{matrix} 1 \\ -1 \\ 0  \end{matrix} \right].\]
Therefore the piecewise function $f_1 = \frac{1}{2}s_{e_1^\pm}-1$ and $f_1=-\frac{1}{2}s_{e_2^\pm}$
has the property that $\mu_Z = \nabla^2(f_0+f_1')$. Finally, we adjust $f_1'$ by constants to make the height function zero at a chosen base point. We fix a base point reducing to the component $v_3$. This is equivalent to requiring that $f_1(0) =0$ on each piecewise part, so $f_1 = (-1)^{i-1} \frac{1}{2}(s_{e_i^\pm}-1)$. Our height function is therefore the piecewise quadratic function $\tilde{h}_{Z,\ell}$ given by $(-1)^{i}1/2(s_{e_i^\pm}^2 - 2s_{e_i^\pm}+1) $ on $e_i^\pm$ and $0 \leq s_{e_i^\pm} \leq 1$.

Using $\tilde{h}_{Z,\ell}$, we compute the heights at the known rational points of $X$. 
For each point, we first find what component of $\Gamma$ it lies on. 
The points at infinity reduce to the unique vertex $v_3$ corresponding to the top cluster. For each finite point $P$, we compute the smallest cluster $\ss$ such that $P$ belongs to $U_\ss$. For example, if $P = (1 : -3 : 1)$, picking a root $(-5 + O(3^{3})) \in \tt_1$, then since $v_3(1 - (-5 + O(3^{3}))) > 0$ (where $0$ is the absolute depth of the top cluster)  we see that $P$ lies in $U_{\tt_1}$. To distinguish if $P$ reduces to the vertex $v_1$ or somewhere along an edge from $\tt_1$ to its parent, we check if $P$ belongs to the bounding annulus $A_{\tt_1}$. In this case, since $v_3(1 - (-5 + O(3^{3})))  = 1$ is not strictly less than $\delta_{\tt_1} = 1$, the point $P$ does not belong to the $A_{\tt_1}$.

Using the formula for $\tilde{h}_{Z,\ell}$ with $s_{e_1^-} = 1$ we see the height at $P$ is $1/2$. Similar reasoning shows any finite $\Q_3$-point $P = (x:y:z)$ where $x/z$ is congruent to $1$ modulo $3$ reduces to $v_1$ and has height $1/2$. If $x/z$ is congruent to $2$ modulo $3$, then $P$ reduces to $v_2$ and has height $-1/2$, and if $x/z$ is $0$ modulo $3$ then $P$ reduces to $v_3$ and has height $0$.

\begin{remark}
As the above example demonstrates, we are able to use very low $\ell$-adic precision when determining the values of the local heights at primes of bad reduction, and the answers we obtain are rigorous.
\end{remark}

\subsection{Atkin--Lehner quotients of modular curves}
\label{subsec:atkinlehnerquotiens}
In this subsection, we study the Atkin--Lehner quotients $X_0(N)^*$ of $X_0(N)$ by the full Atkin--Lehner group for $N = 147,$ $225,$ $330$. In \cite{AdzagaAtkinLehner} the rational points were computed using elliptic curve Chabauty \cite{BruinEC} instead of quadratic Chabauty. Here, we show all local heights are trivial.

\subsubsection{\texorpdfstring{$X_0(330)^*$}{X0(330)*}}
\label{subsubsec:330}

For $N = 330$, the primes of bad reduction are $3$, $5$ and $11$.
Note that $X = X_0(N)^*$ has endomorphism ring $\Z[\sqrt{2}]$, and hence we can take our trace $0$ endomorphism $Z$ of the Jacobian to be $\sqrt{2}$. Write $X: y^2 = x^6 + 8x^4 + 10x^3 + 20x^2 + 12x + 9$.

At $\ell = 5$ and $\ell = 11$ the cluster pictures are the cluster picture in \Cref{fig:cluster022}. By \Cref{thm:reductiontocomponents}, since the leading coefficient of $f$ is $1$, the local heights at these primes are trivial on $\Q_\ell$ points.
 
For $\ell = 3$ however, the graph $\Gamma$ associated to the semistable covering has two vertices  $v_1, v_2$ corresponding to genus zero components, with a loop $e_0$ of length 1 at $v_1$ and two edges $e_1, e_2$ of length 1 connecting $v_1$ with $v_2$, see \Cref{fig:atkinlehnerreduction}.
\begin{figure}
\begin{center}
\includegraphics[width=4cm]{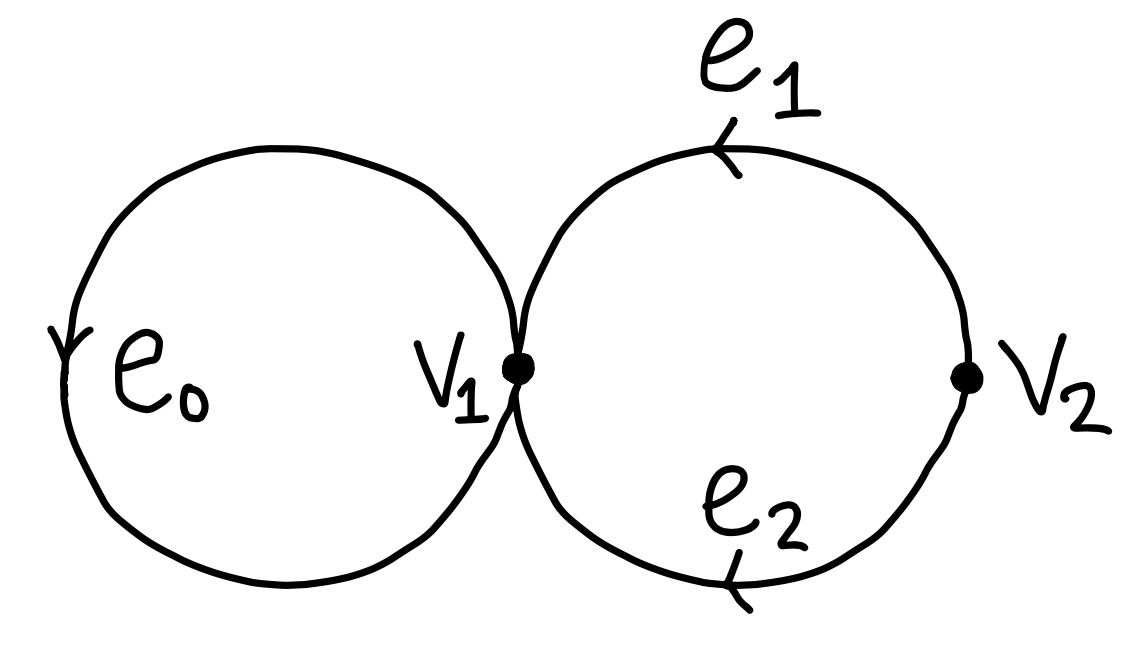}
\end{center}
\caption{Reduction graph for Atkin--Lehner quotient}
\label{fig:atkinlehnerreduction}
\end{figure}

The curves attached to $v_1$ and $v_2$ both contain $\F_3$-points.
 Since there are no higher genus vertices, we only need to compute the action of $Z_*$ on the homology of the $\Gamma$. Like in \S\ref{subsec:shimuracurve}, we can compute it $3$-adically and then use \Cref{cor:padicapproximationcorrect} to obtain the following matrix of $Z_*$ acting on the homology $H_1(\Gamma)$.
\begin{equation}
\label{eq:homologygraphactionAL}
\left[ \begin{matrix}
0 & -2\\
-1& 0 
\end{matrix}
\right]
\end{equation}
Here the basis of the homology of $\Gamma$ is $e_0,e_2-e_1$. By the local heights formula \Cref{thm:local_heights_formula} we then have that the Laplacian $\nabla^2(\tilde{h}_{Z,\ell})$ is $0$, and hence $\tilde{h}_{Z,\ell} = 0$. We see that the local heights vanish.

\subsubsection{\texorpdfstring{$X_0(255)^*$}{X0(255)*}}
For $N = 255$, the primes of bad reduction are $5$ and $17$.
Again, $X = X_0(N)^*$ has endomorphism ring $\Z[\sqrt{2}]$, and we can take our trace $0$ endomorphism $Z$ of the Jacobian to be $\sqrt{2}$. Write $X: y^2 = x^6 - 4x^5 - 12x^4 + 2x^3 + 8x^2 - 4x + 1$.
At $\ell = 17$, the cluster picture is in \Cref{fig:cluster022} and the leading coefficient of $f$ is $1$, so the heights vanish by \Cref{thm:reductiontocomponents}.
For $\ell = 5$, the semistable covering of the curve $X_0(N)^*$  has the same associated graph $\Gamma$ as in the previous example, and the matrix of the action on $H_1(\Gamma)$ is also \eqref{eq:homologygraphactionAL}.

\subsubsection{\texorpdfstring{$X_0(147)^*$}{X0(147)*}}
For $N = 147$, the primes of bad reduction are $3$ and $7$.
At $q = 7$, the curve $X_0(N)^*$ has potential good reduction, so the local height functions are trivial.
At $q = 3$, following the same procedures as above we compute the action of $Z_*$ on the associated homology graph to the semistable covering of $X_0(N)^*$ and find it is negative the matrix in \eqref{eq:homologygraphactionAL} and therefore the heights are also trivial.

\subsection{Higher genus example} 
We now provide an example of local height computations for the genus 7 modular curve with affine model $X:y^2=x^{16} - 4x^8 + 16$.  This curve can be found in the Modular Curves Database of the LMFDB under the label \href{https://beta.lmfdb.org/ModularCurve/Q/48.144.7.baq.1/}{\texttt{48.144.7.baq.1}}. This example shows that our method is practical even when the genus of $X$ is large. We choose a basis $x^i\d x / (2y), i = 0, \dots 2g-1$ for the de Rham cohomology of $X$. With respect to this basis, we choose the endomorphism $Z$ of the Jacobian to be associated to the diagonal matrix with entries $[4,4,4,-24,4,4,4]$.  The primes of bad reduction of $X$ are 2 and 3; we compute local heights at $q=3$.  In this case, the cluster picture is as in \Cref{fig:genus7cluster} and the reduction graph is as in \Cref{fig:genus7reduction}.
\begin{figure}[h]
\begin{center}
\includegraphics[width=8cm]{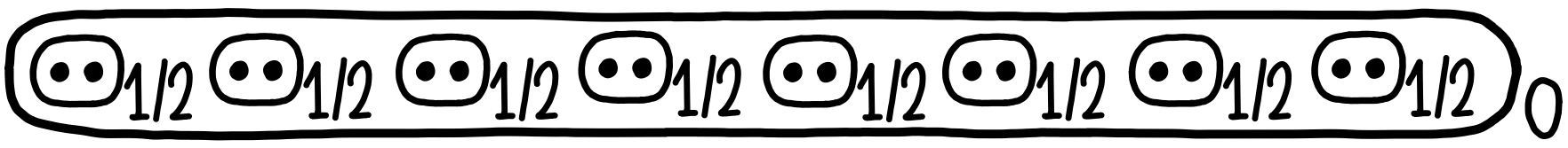}
\end{center}
\caption{Cluster picture for genus 7 curve}
\label{fig:genus7cluster}
\end{figure}
\begin{figure}[h]
\begin{center}
\includegraphics[width=4cm]{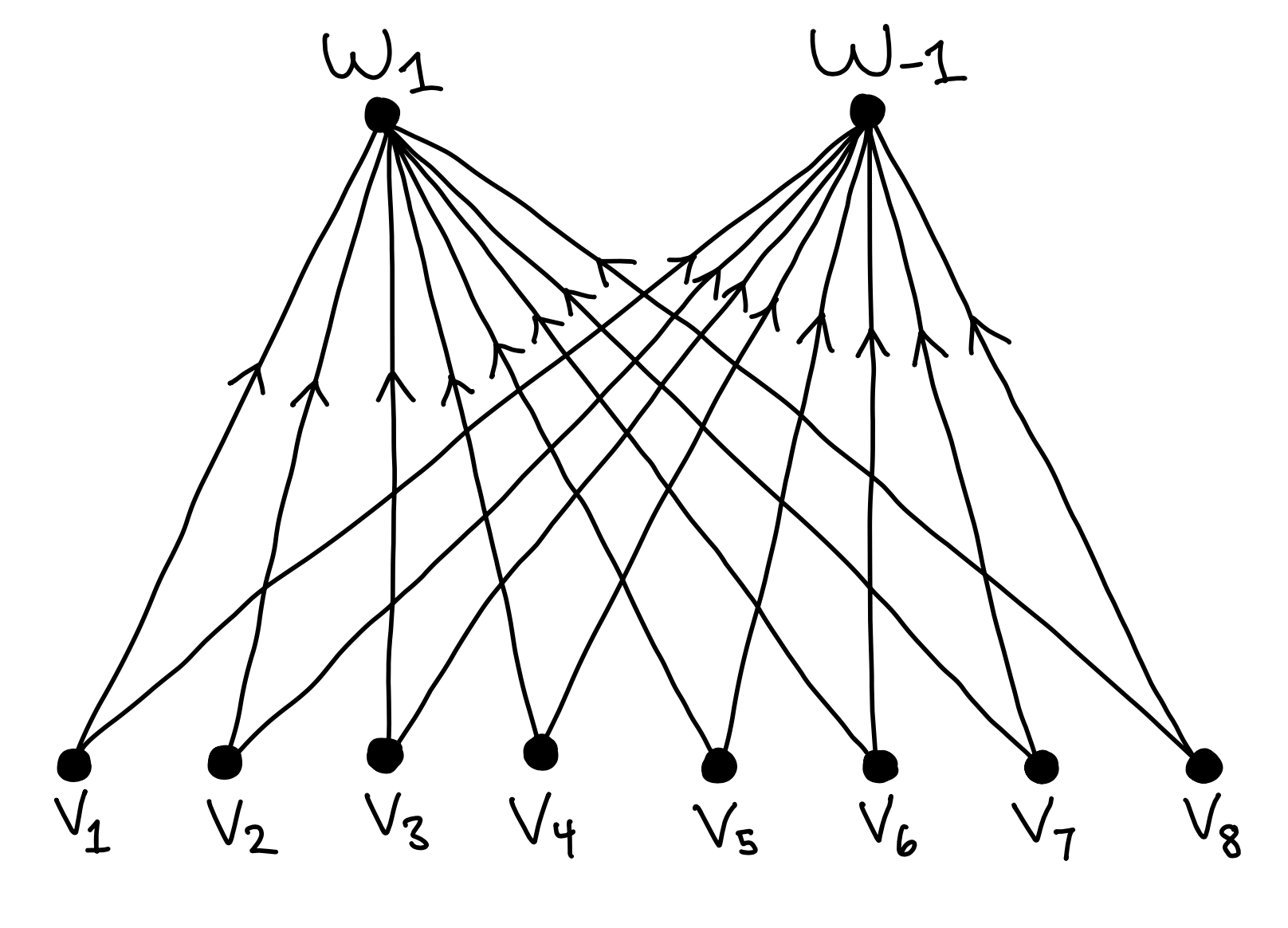}
\end{center}
\caption{Reduction graph for genus 7 curve}
\label{fig:genus7reduction}
\end{figure}
By  \Cref{thm:reductiontocomponents}, after choosing a basepoint mapping to $w_{-1}$, the heights at all rational points will be trivial, but we can still determine the (non-trivial) height functions. We find that the action of $Z_*$ on the homology of the dual graph is given by the matrix
\begin{equation*}
\left[ \begin{matrix}
 4 & 0 &  7& 7& 0& 7& 7 \\
 0& 4& 7& 7& 0& 7& 7 \\
 0& 0& -3& -7& 0& -7& -7 \\
 0& 0& -7& -3& 0& -7& -7 \\
 0& 0& 7& 7& 4& 7& 7 \\
 0& 0& -7& -7& 0& -3& -7 \\
 0& 0& -7& -7& 0& -7& -3 \end{matrix} \right].
\end{equation*}
This yields $8$ distinct piecewise polynomial functions (since the two edges from each vertex $v_i$ to $w_{\pm 1}$ have the same height function). For example, the edges from $v_1$ to $w_{\pm 1}$ have the height function
\[248873/44800 s^2 - 248873/44800s + 248873/179200 \text{ where } 0 \leq s \leq 1/2.\]

\subsection{A family of curves with genus 1 vertices}
Let $S_0$ be a set of 4 roots in $\Qbar_q$ such that $1 + q \Z_q$ is the smallest disc containing them. Let $n \in \Z$ with $n \mid (q-1)$, let $\zeta = \zeta_n$ be a primitive $n$th root, and consider $S_k = \zeta^k S_0$ for $0 \leq k < n$. Let $S^{(n)} = \bigcup_{n>k\geq 0} S_k$, and let $f^{(n)} = \prod_{r \in S^{(n)}} (x-r)$.

Let $X_n$ be the curve given by $y^2 = f^{(n)}(x)$.
By construction, this has cluster picture \Cref{fig:highergenuscluster}.
\begin{figure}[h]
\begin{center}
\includegraphics[width=4cm]{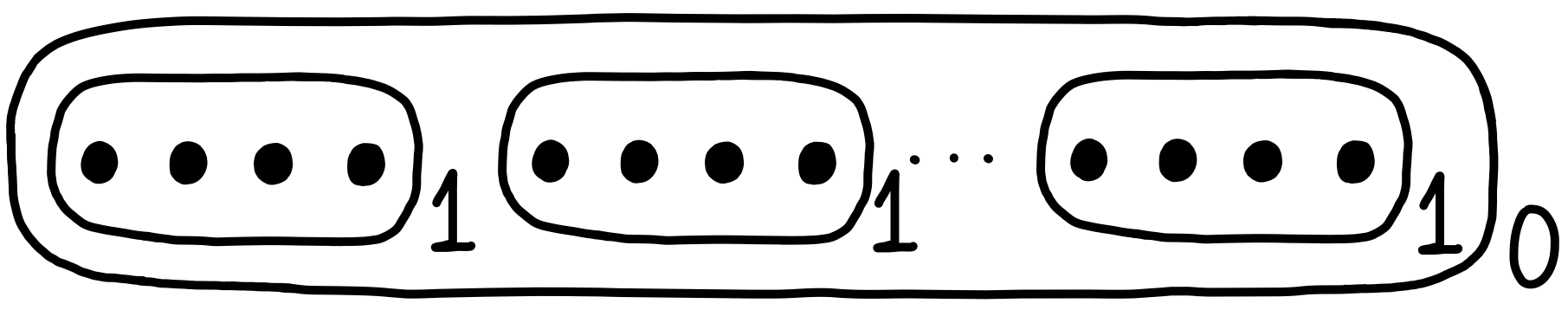}
\end{center}
\caption{Cluster picture for higher genus family}
\label{fig:highergenuscluster}
\end{figure}

Hence the semistable covering has the associated reduction graph $\Gamma$ in \Cref{fig:highergenusreduction}, where vertices $w_0$ and  $w_1$ correspond to components of genus $0$, and $v_0,\dots,v_{n-1}$ correspond to components of genus $1$. Denote the edges from $v_i$ to $w_0$ as $e_i^+$, and those from $v_i$ to $w_1$ as $e_i^-$.
\begin{figure}
\begin{center}
\includegraphics[width=4cm]{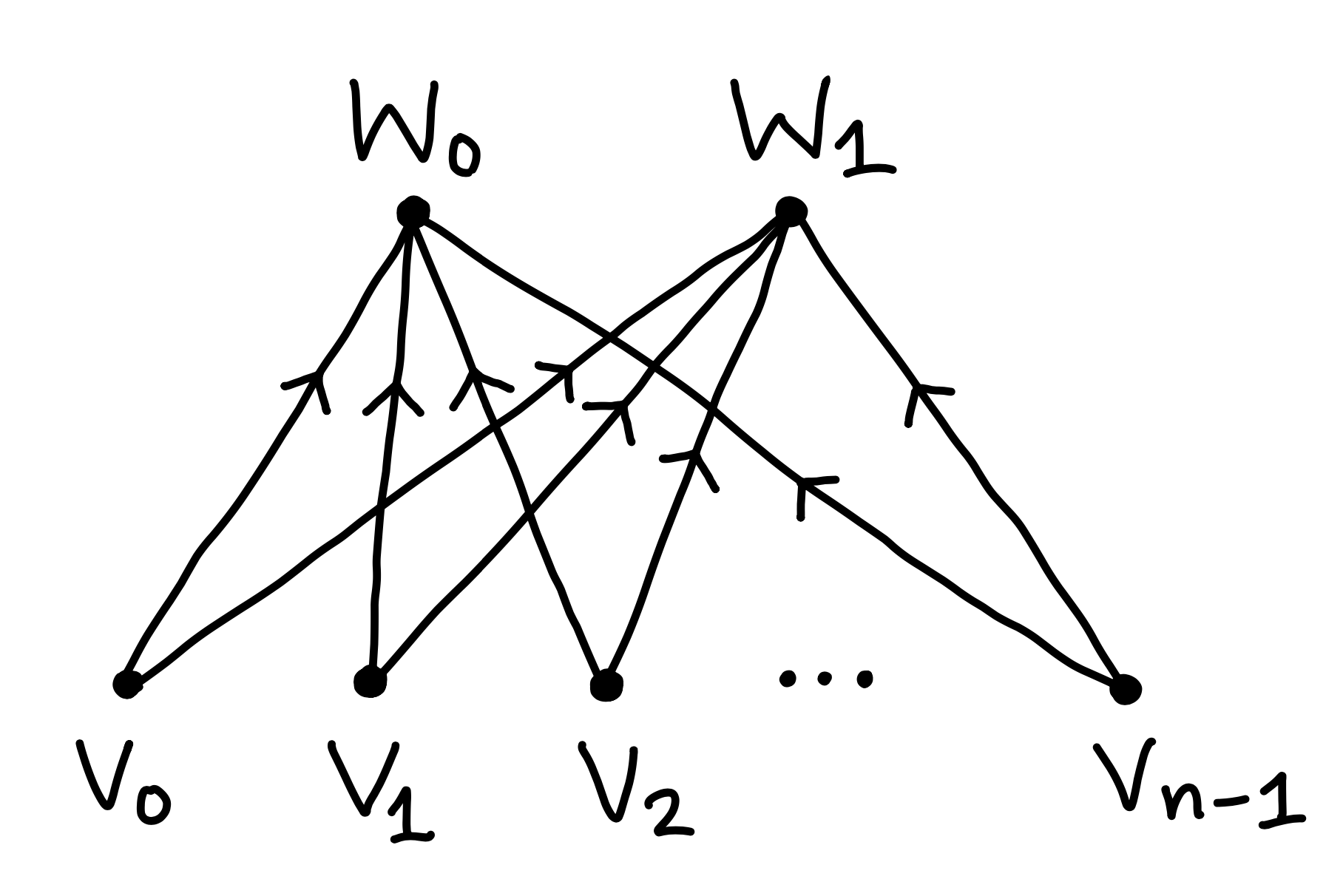}
\end{center}
\caption{Reduction graph for higher genus family}
\label{fig:highergenusreduction}
\end{figure}

Also by construction, the family of curves has an automorphism $\zeta: X_n \to X_n$, cyclically permuting the $n$ clusters. On the level of graphs, this cyclically permutes the $n$ genus $1$ vertices, while keeping the genus $0$ vertices invariant. We will also use $\zeta$ to denote the corresponding endomorphism of the Jacobian.

Although the local height only makes sense for a trace $0$ endomorphism, the formula for the measure in \Cref{thm:local_heights_formula} makes sense for any endomorphism. This measure will be of mass $0$ if the endomorphism is of trace $0$. We let $\mu_k$ denote the measure corresponding to $\zeta_k$.

We need to compute $\langle e_i^+,\zeta^k(\pi(e_i^+)) \rangle$ where $\pi: C_1(\Gamma) \to \rH_1(\Gamma)$ is the orthogonal projection. A straightforward calculation shows that \[\pi(e_i^+) = \frac{n-1}{2n}(e_i^+ - e_i^-) - \frac{1}{2n}\sum_{j \ne i} (e_j^+ - e_j^-).\]
and hence $\langle e_i^+,\zeta^k(\pi(e_i^+)) \rangle$ is $\frac{n-1}{2n}$ if $k = 0$ and $-\frac{1}{2n}$ otherwise.  By symmetry, the same holds for edges $e_i^-$.

Since $\zeta$ acts by cyclic permutation on the associated graph, we see that $\tr_{v_i}(\zeta^k)$ is $2$ if $k = 0$ and $0$ otherwise. In total, we find
\begin{align*}
\mu_0 = \frac{n-1}{n} \sum_{i} (|\rd s_{e_i^+}| +|\rd s_{e_i^-}|) + \sum_i 2\delta_v\\
\mu_k = \frac{-1}{n} \sum_{i} (|\rd s_{e_i^+}| +|\rd s_{e_i^-}|) \text{ when } k \neq 0.
\end{align*}
 These have total mass $2(n-1) + 2n = 4n-2$ and $-2$ respectively. Hence we see a linear combination $\sum_k a_k \zeta_k$ has trace $0$ if and only if $(2n-1)a_0 - \sum_{k \ne 0} a_k = 0$. For example, $Z = \zeta_0 - (2n-1)\zeta_1$ has trace $0$, and measure
\[
\mu_Z = -\sum_{i} (|\rd s_{e_i^+}| +|\rd s_{e_i^-}|) + \sum_i 2\delta_{v_i}.
\]

We find an inverse Laplacian using the method explained in \S\ref{subsec:localHeightFormula}. First we find a quadratic piecewise polynomial function $f_0$ such that $\mu_z - \nabla^2 f_0$ is a sum of $\delta$ measures. We see that the function that is $\frac{-1}{2}s(1-s)$ on every edge works. Then $\mu_z - \nabla^2 f_0 = -\frac{n}{2}(\delta_{w_0} + \delta_{w_1}) + \sum_i \delta_{v_i}$, and we need to find a piecewise linear function $f_1$ with this Laplacian. We see any function $f_1$ which has slopes $-\frac12$ along $e_i^{\pm}$ suffices.

Taking the basepoint to lie in $w_0$ uniquely determines $f_1$ by $f_1(w_0) = 0$, and we find that the normalised height is the piecewise polynomial given by
\[
\frac{1}{2}(s-1)^2\text{ on } e_i^{\pm} \text{ with } 0 \leq s \leq 1.
\]

\subsection{Applications to quadratic Chabauty computations}\label{sec:QC}
Consider the quadratic twist by $5$ of the genus 2 curve with LMFDB label \href{https://www.lmfdb.org/Genus2Curve/Q/18225/c/164025/1}{\texttt{18225.c.164025.1}}, given by the affine model $X: y^2 = x^6 + 18/5x^4 + 6/5x^3 + 9/5x^2 + 6/5x + 1/5$. The conductor is $3^6 \cdot 5^4$. 
The endomorphism algebra of $X/\Q$ is $\Q(\sqrt{13})$. We choose the endomorphism $\sqrt{13}$ on the Jacobian. We compute the local height functions at $3$ and $5$ and determine the set of rational points $X(\Q)$ using quadratic Chabauty. We will choose the prime $p = 53$ in order to apply a Mordell--Weil sieve \cite{BruinStollMW} after finding a superset of $p$-adic points containing the rational points.

\begin{remark}
As pointed out to use by Steffen M\"{u}ller, the Galois group of $x^6 + 18/5x^4 + 6/5x^3 + 9/5x^2 + 6/5x + 1/5$ is small. For this reason, the rational points $X(\Q)$ can also be determined using elliptic Chabauty (see \cite{BruinEC}). We include code in our repository for determining the points this way provided by Michael Stoll. We provide this example using quadratic Chabauty as a proof-of-concept example for the method.
\end{remark}

\begin{remark}
The implementation \cite{QCMod} chooses the endomorphism $Z$ to be $-4 \sqrt{13}$ when $p = 53$. For this reason, in our implementation we have to scale the heights accordingly.
\end{remark}

Applying \Cref{prop:pushpull}, after computing a correspondence $Z \subset X \times X$ for $\sqrt{13}$, we compute the action of $Z_*$ on $\rH^1_{\dR}(X)$ to be
\begin{align*}
 \left[ \begin{matrix}
1 & 2 & 4/5 & -3/10\\
 6 &   -1 & 9/5 & 7/10\\
0 & 0& -1 &1\\
0  & 0  & 12 & 1
\end{matrix} \right].
\end{align*} 
where the basis for $\rH^1_{\dR}(X)$ is $\{ \rd x/ (2y), x\rd x/ (2y), x^3 \rd x/(2y), (9/10)x^2 \rd x/(2y)-(-1/2)x^4 \rd x/(2y)\}$. Let $b = (-1/3: 1/27:1)$ be the basepoint.

\begin{figure}
\begin{center}
\includegraphics[width=3cm]{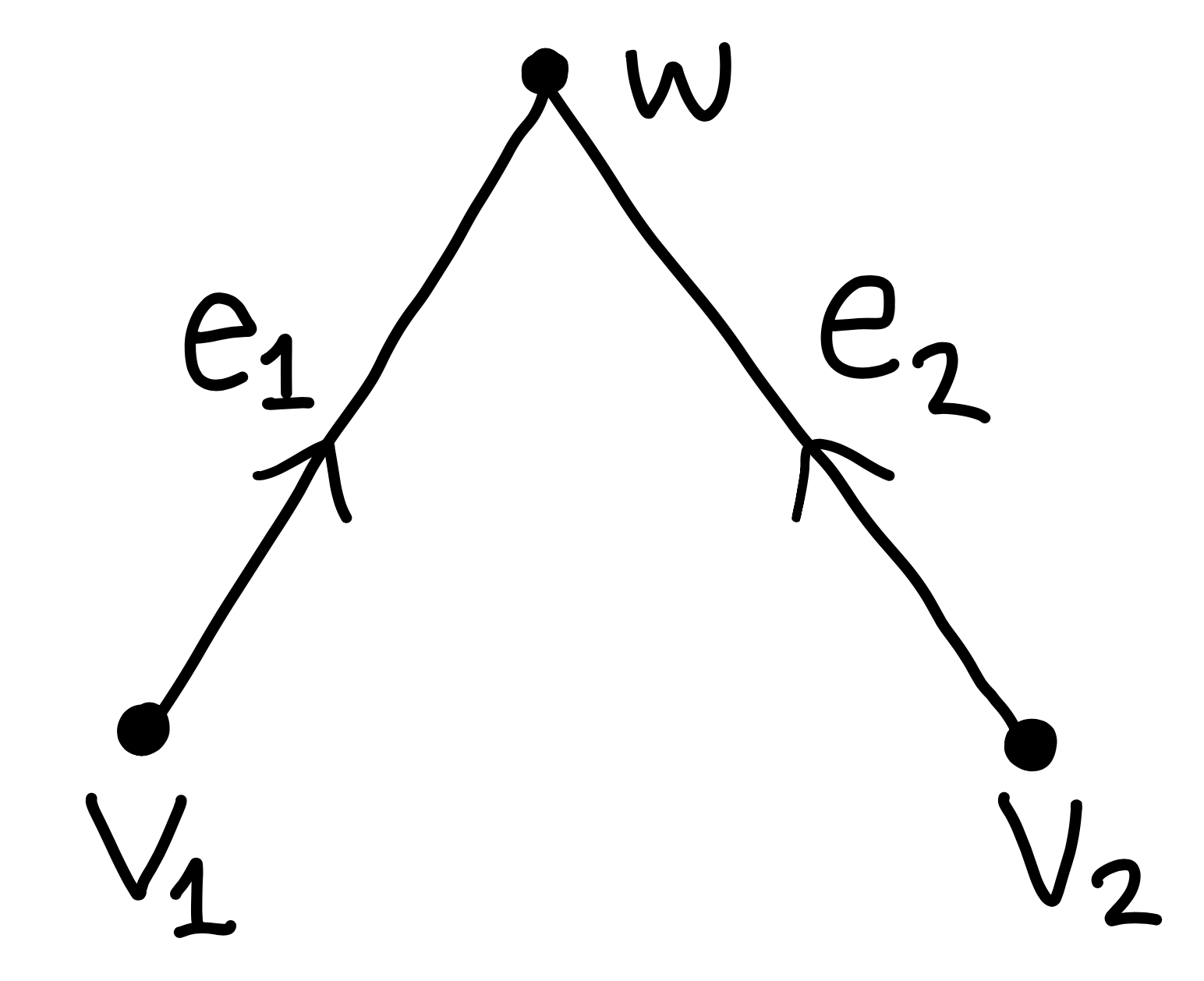}
\caption{Reduction graph for quadratic Chabauty example at $\ell = 3$ }
\label{fig:reductionqc}
\end{center}
\end{figure}
At $\ell = 3$, the cluster picture is \Cref{fig:cluster3QC} and the curve has unstable reduction; in fact, the curve has unstable reduction over every tame extension.
\begin{figure}[h]
\includegraphics[width=4cm]{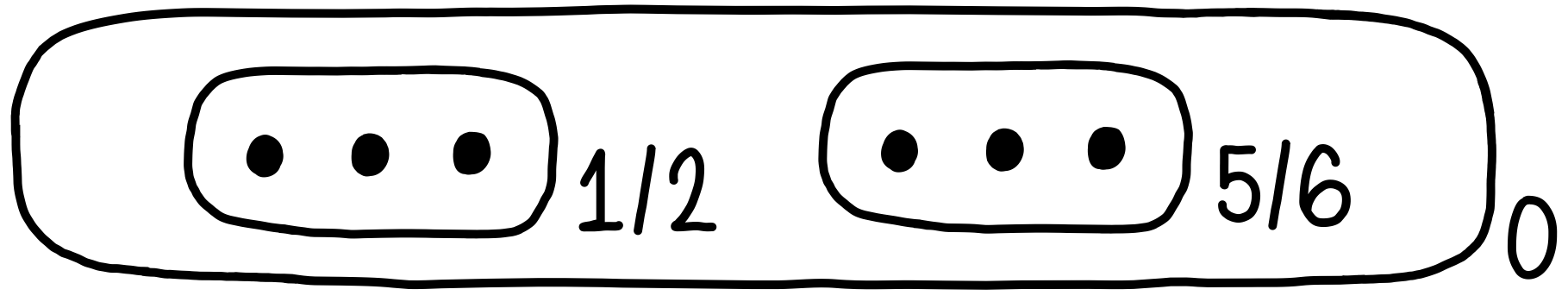}
\caption{Cluster picture for quadratic Chabauty example at $\ell = 3$}
\label{fig:cluster3QC}
\end{figure}
 The associated reduction graph $\Gamma$ has two genus 1 vertices $v_1$ and $v_2$, linked by a genus 0 vertex, by edges $e_1$ and $e_2$ respectively, as in \Cref{fig:reductionqc}. We must compute the traces $\tr_{v_i}(Z)$ for each genus $1$ vertex $v_i$. 
Applying \Cref{prop:traceblocks}, we compute the block matrices
 \begin{align*}
 \left[ \begin{matrix}
 -1 + O(3^3) & 0+  O(3^3)\\
0 + O(3^3)& -1 + O(3^3) \end{matrix}\right], \quad
 \left[ 
 \begin{matrix}
 1 + O(3^3) & 0+ O(3^3)\\
0+ O(3^3) & 1+ O(3^3)\end{matrix}\right]
\end{align*}
to be the action of $Z_*$ on $\rH^1_{\dR}(\calX_{v_i})$, for $i = 1, 2$ respectively.   By \Cref{thm:bounds}, the traces of these matrices are bounded by $2\cdot 1\cdot \max\{2,7\}$, so they are exactly $-2$ and $2$.  Using the formula from \Cref{thm:local_heights_formula}, the Laplacian of the local height function is given by the measure
\[\mu_Z =  -2 \delta_{v_1}  + 2 \delta_{v_2}.\] We solve for the height function, and find it is the piecewise polynomial function
\begin{equation*}
\begin{cases}
-2s +1/2 \text{ on } e_1 \text{ with } 0 \leq s \leq 1/4 \\
2s - 5/6 \text{ on } e_2 \text{ with } 0 \leq s \leq 5/12
\end{cases}
\end{equation*}
Using similar calculations to \Cref{subsec:shimuracurve}, we see that if $P = (x : y : z)$ is a point with $z \not\equiv 0 \mod 3$, then if $x/z \equiv 1 \mod 3$, then $P$ lies at distance $1/12$ along edge $e_2$. If $z \equiv 0 \mod 3$ then $P$ reduces to the genus 0 vertex. (By \Cref{thm:reductiontocomponents}, we cannot have $\Q_\ell$-points reducing to $\calX_{v_1}$.)
Therefore the normalised local height of  $P = (x : y : z)$ is
\begin{equation*}
\begin{cases}
0 &\text{ if } z \equiv 0 \mod 3,\\
-2/3 &\text{ if } x/z \equiv 1 \mod 3.
\end{cases}
\end{equation*}

Now consider $\ell = 5$, the other prime of bad reduction. Here the curve has unstable reduction.
The cluster picture is \Cref{fig:cluster5QC}.
\begin{figure}[h]
\includegraphics[width=4cm]{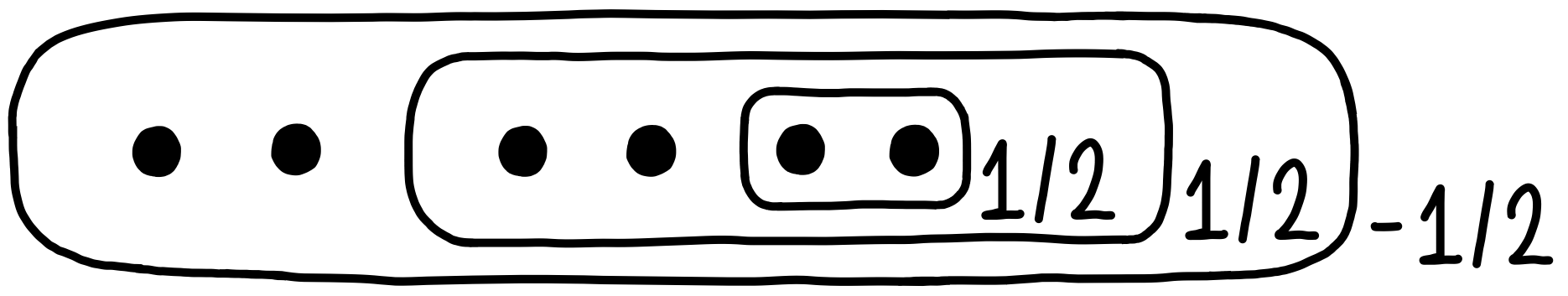}
\caption{Cluster picture for quadratic Chabauty example at $\ell = 5$}
\label{fig:cluster5QC}
\end{figure}

\begin{figure}[h]
\includegraphics[width=4cm]{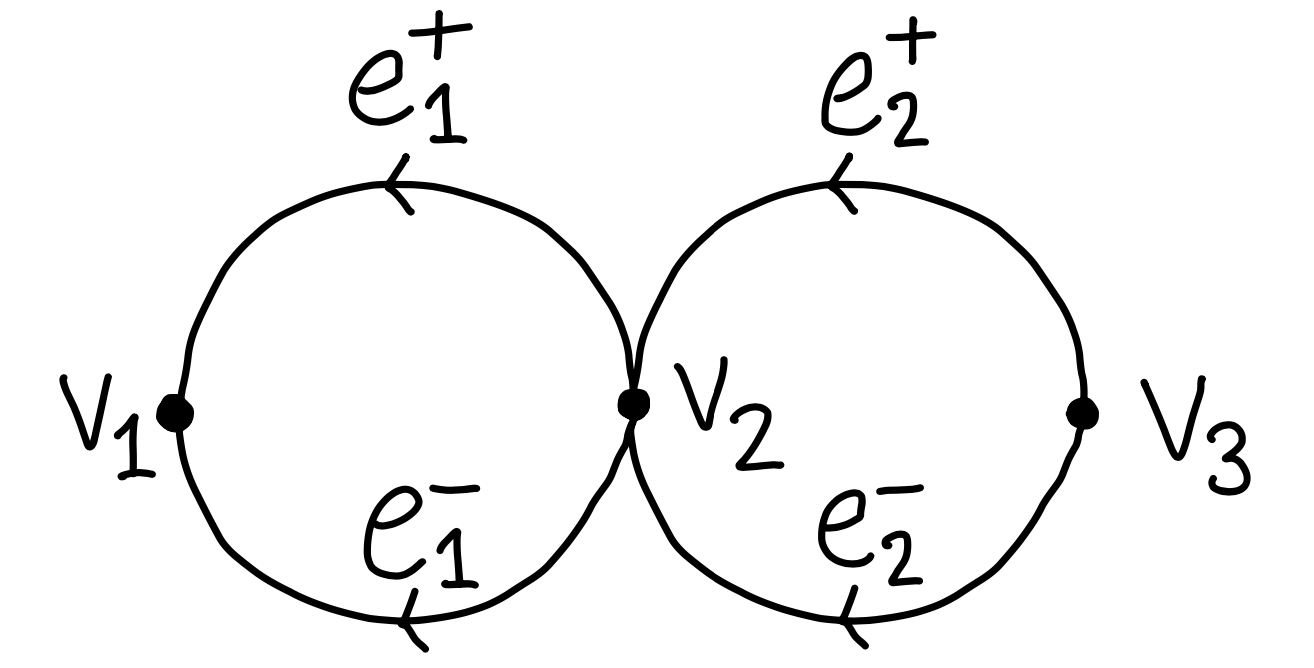}
\caption{Reduction graph for quadratic Chabauty example at $\ell = 5$}
\label{fig:reductionqc5}
\end{figure}

Let $e_1^\pm, e_2^\pm$ be the edges as labeled in \Cref{fig:reductionqc5}, with $l(e_i^\pm) =1/2$.
We compute the following piecewise polynomial height function.
\begin{equation*}
\begin{cases}
3s^2 &\text{ on } e_1^\pm \text{ with } 0 \leq s \leq 1/2  \\
-3s^2+3s-3/4 &\text{ on } e_2^\pm  \text{ with } 0 \leq s \leq 1/2
\end{cases}
\end{equation*}
If $P = (x : y : z)$ is a rational point with $z \not\equiv 0 \mod 5$ and $x/z \equiv 2 \mod 5$, then $P$ reduces to $v_3$. If $z \equiv 0 \mod 5$ then $P$ reduces to $v_1$. Otherwise $P$ reduces to $v_2$. 
Therefore the normalised local height of  $P = (x : y : z)$ is
\begin{equation*}
\begin{cases}
3/4 &\text{ if } z \equiv 0 \mod 5,\\
-3/4 &\text{ if } x/z \equiv 2 \mod 5, \\
0 &\text{ otherwise.}
\end{cases}
\end{equation*}

The curve $X$ has Mordell--Weil rank $2$, and therefore we can apply the quadratic Chabauty method at $p = 53$ (as implemented in \cite{QCMod}) to find a finite set of $p$-adic points containing the rational points $X(\Q)$. We then use a Mordell--Weil sieve (for more details see \cite{BruinStollMW}) at $97$ to rule out the extra points in all but two residue discs. We futher sieve at the primes $23, 283, 1259, 2447,$ and $5419$ to rule out these residue discs. 
We obtain the following theorem.
\begin{theorem}\label{thm:qctheorem}
The rational points of $X: y^2 = x^6 + 18/5x^4 + 6/5x^3 + 9/5x^2 + 6/5x + 1/5$ are the $10$ points
\begin{align*}
\{ &(-1/3 : -1/27 : 1), (-1/3 : 1/27 : 1), (-1/5 : -21/125 : 1), (-1/5 : 21/125 : 1), \\
&(1 : 3 : 1),(1 : -3 : 1), (1 : 1 : 0), (1 : -1 : 0), (-1/2 : -3/8 : 1), (-1/2 : 3/8 : 1)\}.
\end{align*}
\end{theorem}

\bibliographystyle{alpha} 
\bibliography{bibliography.bib}

\end{document}